\documentclass[a4paper,reqno,11pt]{amsart}
\usepackage[T1]{fontenc}
\usepackage[utf8]{inputenc}
\usepackage[english]{babel}
\usepackage[dvipsnames,svgnames,table]{xcolor}
\usepackage{amssymb, amsmath, amsthm, graphicx, enumerate, color, mathtools, comment, caption, subcaption, float, lmodern}
\usepackage[foot]{amsaddr}
\usepackage[shortlabels]{enumitem}
\usepackage[unicode=true]{hyperref}
\usepackage[capitalise, noabbrev]{cleveref}
\usepackage{thmtools, thm-restate}

\newcommand{\q}[1]{``#1''}
\DeclarePairedDelimiter\set{\{}{\}}

\newcommand{\closure}[1]{\overline{#1}}
\newcommand{\algo}[1]{\mathbb{#1}}
\newcommand{\defin}[1]{\emph{\textcolor{ForestGreen}{#1}}}
\newcommand{\defmath}[1]{\defin{\text{$#1$}}}

\DeclareMathOperator{\ext}{ext}
\DeclareMathOperator{\Inc}{Inc}
\DeclareMathOperator{\hposet}{height}
\DeclareMathOperator{\se}{se}

\newcommand{\OR}{R}
\newcommand{\OL}{L}

\DeclareMathOperator{\kelly}{kelly}

\DeclareMathOperator{\Int}{int}

\DeclareMathOperator{\gcpe}{gcpe}
\DeclareMathOperator{\lcse}{lcse}

\let\le\leqslant
\let\ge\geqslant
\let\leq\leqslant
\let\geq\geqslant

\let\preceq\preccurlyeq

\let\subset\subseteq

\let\epsilon\varepsilon

\let\eps\varepsilon

\newcommand{\WR}{W_{\OR}}
\newcommand{\WL}{W_{\OL}}

\newcommand{\bfs}{\textsc{bfs}}

\newcommand{\RR}{\mathcal{R}}

\newcommand{\NN}{\mathbb{N}}

\newcommand{\calC}{\mathcal{C}}

\newcommand{\calE}{\mathcal{E}}

\newcommand{\calI}{\mathcal{I}}

\newcommand{\calL}{\mathcal{L}}

\newcommand{\calQ}{\mathcal{Q}} 
\newcommand{\calR}{\mathcal{R}}
\newcommand{\calS}{\mathcal{S}}
\newcommand{\calT}{\mathcal{T}}

\newcommand{\Oh}{\mathcal{O}}

\makeatletter
\newcommand{\myitem}[1]{%
\item[#1]\protected@edef\@currentlabel{#1}%
}
\makeatother

\makeatletter
\newcommand{\leqnomode}{\tagsleft@true\let\veqno\@@leqno}
\newcommand{\reqnomode}{\tagsleft@false\let\veqno\@@eqno}
\makeatother

\newcommand{\problemdef}[3]{
  \vspace{1mm}
\noindent\fbox{
  \begin{minipage}{0.95\textwidth}
  #1 \\
  {\bf{Input:}} #2  \\
  {\bf{Output:}} #3
  \end{minipage}
  }
  \vspace{1mm}
}

\newcommand{\PDD}{\textsc{Planar Diagram Dimension}\ }

\newcommand{\SCPDD}{\textsc{Singly Constrained Planar Diagram Dimension}\ }

\usepackage{algorithm}
\usepackage[noend]{algpseudocode}
\algdef{SE}[DOWHILE]{Do}{doWhile}{\algorithmicdo}[1]{\algorithmicwhile\ #1}

\algnewcommand{\IIf}[1]{\State\algorithmicif\ #1\ \algorithmicthen}
\algnewcommand{\EndIIf}{\unskip}

\renewenvironment{enumerate}{\begin{enumorig}[label=\textup{(\roman*)}, noitemsep, 
topsep=2pt plus 2pt, labelindent=.2em, leftmargin=*, widest=iii]}{\end{enumorig}}

\newenvironment{enumerateNumI}{\begin{enumorig}[label=\textup{(I\arabic*)}, 
noitemsep, topsep=2pt plus 2pt, labelindent=.2em, leftmargin=*, widest=I10]}{\end{enumorig}}

\newenvironment{enumerateNumU}{\begin{enumorig}[label=\textup{(u\arabic*)}, 
noitemsep, topsep=2pt plus 2pt, labelindent=.2em, leftmargin=*, widest=u10]}{\end{enumorig}}

\newenvironment{enumerateNumeo}{\begin{enumorig}[label=\textup{(\arabic*)}, 
noitemsep, topsep=2pt plus 2pt, labelindent=.2em, leftmargin=*, widest=10]}{\end{enumorig}}

\newenvironment{enumerateNume}{\begin{enumorig}[label=\textup{(e\arabic*)}, 
noitemsep, topsep=2pt plus 2pt, labelindent=.2em, leftmargin=*, widest=e10]}{\end{enumorig}}

\newenvironment{enumerateNumo}{\begin{enumorig}[label=\textup{(o\arabic*)}, 
noitemsep, topsep=2pt plus 2pt, labelindent=.2em, leftmargin=*, widest=o10]}{\end{enumorig}}

\newenvironment{enumerateNumr}{\begin{enumorig}[label=\textup{(r\arabic*)}, 
noitemsep, topsep=2pt plus 2pt, labelindent=.2em, leftmargin=*, widest=r10]}{\end{enumorig}}

\newenvironment{enumerateNump}{\begin{enumorig}[label=\textup{(p\arabic*)}, 
noitemsep, topsep=2pt plus 2pt, labelindent=.2em, leftmargin=*, widest=s10]}{\end{enumorig}}

\newenvironment{enumerateNumt}{\begin{enumorig}[label=\textup{(t\arabic*)}, 
noitemsep, topsep=2pt plus 2pt, labelindent=.2em, leftmargin=*, widest=s10]}{\end{enumorig}}

\newenvironment{enumerateNumH}{\begin{enumorig}[label=\textup{(H\arabic*)}, 
noitemsep, topsep=2pt plus 2pt, labelindent=.2em, leftmargin=*, widest=s10]}{\end{enumorig}}

\newenvironment{enumerateNumHprim}{\begin{enumorig}[label=\textup{(H\arabic*')}, 
noitemsep, topsep=2pt plus 2pt, labelindent=.2em, leftmargin=*, widest=s10]}{\end{enumorig}}

\newenvironment{enumerateNumHR}{\begin{enumorig}[label=\textup{(HR\arabic*)}, 
noitemsep, topsep=2pt plus 2pt, labelindent=.2em, leftmargin=*, widest=s10]}{\end{enumorig}}

\newenvironment{enumerateNumHRprim}{\begin{enumorig}[label=\textup{(HR\arabic*')}, 
noitemsep, topsep=2pt plus 2pt, labelindent=.2em, leftmargin=*, widest=s10]}{\end{enumorig}}

\newenvironment{enumerateNumHL}{\begin{enumorig}[label=\textup{(HL\arabic*)}, 
noitemsep, topsep=2pt plus 2pt, labelindent=.2em, leftmargin=*, widest=s10]}{\end{enumorig}}

\newenvironment{enumerateNumHLprim}{\begin{enumorig}[label=\textup{(HL\arabic*')}, 
noitemsep, topsep=2pt plus 2pt, labelindent=.2em, leftmargin=*, widest=s10]}{\end{enumorig}}

\renewenvironment{itemize}{\begin{itemorig}[label=\tiny\textbullet, noitemsep, 
topsep=2pt plus 2pt, labelindent=.2em, leftmargin=*, widest=ii]}{\end{itemorig}}
\makeatletter
\def\thm@space@setup{
  \thm@preskip=4mm
  \thm@postskip=0mm
}
\makeatother

\usepackage{mdframed}

\mdfdefinestyle{dontsplit}{
  hidealllines=true,
  nobreak=true,
  leftmargin=0pt,
  rightmargin=0pt,
  innerleftmargin=0pt,
  innerrightmargin=0pt,
}

\newmdtheoremenv[style=dontsplit]{theorem}{Theorem}
\newmdtheoremenv[style=dontsplit]{lemma}[theorem]{Lemma}
\newmdtheoremenv[style=dontsplit]{obs}[theorem]{Observation}
\newmdtheoremenv[style=dontsplit]{remark}[theorem]{Remark}
\newmdtheoremenv[style=dontsplit]{proposition}[theorem]{Proposition}
\newmdtheoremenv[style=dontsplit]{question}[theorem]{Question} 
\newmdtheoremenv[style=dontsplit]{corollary}[theorem]{Corollary} 
\newmdtheoremenv[style=dontsplit]{problem}[theorem]{Problem}
\newmdtheoremenv[style=dontsplit]{conjecture}[theorem]{Conjecture}
\newtheorem*{conjecture*}{Conjecture}

\theoremstyle{remark}
\newmdtheoremenv[style=dontsplit]{claim}[theorem]{Claim}
\crefname{claim}{Claim}{Claims}
\newmdtheoremenv[style=dontsplit]{example}[theorem]{Example}

\newenvironment{proofclaim}[1][]
	{\vspace{-\topsep}\begin{proof}[Proof] }{\end{proof}}

\crefformat{equation}{#2(#1)#3}
\let\eqref\cref
\crefformat{subsection}{Subsection #2#1#3}

\input{style.table-of-contents}

\hypersetup{ 
    colorlinks,
    linkcolor={RoyalBlue},
    citecolor={RubineRed},
    urlcolor={blue!80!black},
    pdftitle={Planarity and dimension II}
}

\newcommand{\jedrzej}[1]{\textcolor{ProcessBlue}{Jędrzej: #1}}
\newcommand{\michal}[1]{\textcolor{CarnationPink}{Michał: #1}}

\newcommand{\fig}{\textcolor{DeepPink}{TODO fig}}

\newif\ifshowcomments
\showcommentstrue      % show comments
\title{Planarity and dimension II}

\author[Blake]{Heather S.~Blake}
\address[Blake]{Mathematics \& Computer Science Department, Davidson College, Davidson, North Carolina 28035}
\email{hsblake@davidson.edu}

\author[Hodor]{Jędrzej Hodor}
\address[Hodor]{Theoretical Computer Science Department, 
Faculty of Mathematics and Computer Science and  Doctoral School of Exact and Natural Sciences, Jagiellonian University, Krak\'ow, Poland}
\email{jedrzej.hodor@gmail.com}

\author[Micek]{Piotr Micek}
\address[Micek]{Theoretical Computer Science Department, 
Faculty of Mathematics and Computer Science, Jagiellonian University, Krak\'ow, Poland}
\email{piotr.micek@uj.edu.pl}

\author[Seweryn]{Michał T.~Seweryn}
\address[Seweryn]{Computer Science Institute, Charles University, Prague, Czech Republic}
\email{seweryn@iuuk.mff.cuni.cz}

\author[Trotter]{William T. Trotter}
\address[Trotter]{School of Mathematics, Georgia Institute of Technology, Atlanta, Georgia 30332}
\email{trotter@math.gatech.edu}

\thanks{H.~S.~Blake is supported by a grant from the Simons Foundation. 
J.~Hodor and P.~Micek are supported by the National Science Center of Poland under grant UMO-2022/47/B/ST6/02837 within the OPUS 24 program. 
W.~T.~Trotter is supported by a grant from the Simons Foundation.}

\begin{document}

\begin{abstract}
The dimension of a poset $P$ is the minimum positive integer $d$ such that $P$ is an induced subposet of $\mathbb{R}^d$ equipped with the product order.
We give a constant-factor polynomial-time approximation algorithm for computing dimension in the class of posets with a planar (Hasse) diagram.
While computing the dimension of a poset is \textsf{NP}-hard in general, the computational complexity of the problem for planar posets remains open. 
The algorithmic result is driven by a structural understanding of the canonical obstruction to small dimension: standard examples.
A longstanding problem, originating in the early 1980s, asked whether every poset with a planar diagram has dimension bounded by a function of the maximum order of a standard example that it contains. 
In the first paper of the series, we have resolved the problem in a more general setting of posets with planar cover graphs by establishing a polynomial bound. 
We prove a stronger bound in the original setting, namely, for every poset $P$ with a planar diagram $\dim(P) \le 96\se(P)+672$, where $\dim(P)$ denotes the dimension of $P$ and $\se(P)$ denotes the maximum order of a standard example contained in $P$. 
\end{abstract}

\maketitle

\newpage
\tableofcontents
\newpage

\section{Introduction}
\label{sec:intro}
In this paper, we study finite partially ordered sets, or \emph{posets} for short.
The \defin{dimension} of a poset $P$, denoted \defin{$\dim(P)$}, is the least positive integer $d$ 
such that
$P$ is isomorphic to an induced subposet of $\mathbb{R}^d$ equipped with the product order.\footnote{In the \defin{product order} of $\mathbb{R}^d$, for $(x_1,\dots,x_d),(y_1,\dots,y_d) \in \mathbb{R}^d$, we have $(x_1,\dots,x_d) \leq (y_1,\dots,y_d)$ if and only if $x_i \leq y_i$ for every $i \in [d]$. For a positive integer $d$, we write \defin{$[d]$} as a compact form of $\{1,\dots,d\}$.}  
Dimension is arguably the most extensively studied measure of complexity for posets.
It captures important concepts in graph theory such as planarity~\cite{S89} and nowhere denseness~\cite{JMOdMW19}.

The computational complexity of poset dimension has been studied for almost five decades. 
The problem of deciding whether a poset has dimension at most $d$ appears in the classical Garey-Johnson list of problems~\cite{GJ79}.
Soon afterwards, Yannakakis proved that computing dimension is \textsf{NP}-hard~\cite{Y82}.
Later, Felsner, Mustaţă, and Pergel showed that 
deciding whether a height-$2$ poset has dimension at most $3$ is already \textsf{NP}-hard~\cite{FMP17}.
%Furthermore, 
On the approximation side, 
Chalermsook, Laekhanukit, and Nanongkai proved that 
unless \textsf{ZPP \(\neq\) NP}, 
there is no polynomial-time $\Oh(n^{1-\varepsilon})$-factor approximation algorithm for computing the dimension~\cite{CLN13}. 
In contrast, we obtain one of the first positive results in the area.
Namely, we give a polynomial-time constant-factor approximation algorithm for the dimension of posets with planar (Hasse) diagrams.
An element $y$ in a poset $P$ \defin{covers} an element $x$ in~$P$ if $x < y$ in~$P$ and there is no $z$ in~$P$ with $x<z<y$ in~$P$.
The \defin{cover graph} of a poset $P$ is the graph whose vertices are the
elements of $P$ with two elements adjacent whenever one covers the other.
%A \defin{diagram} of $P$ is a drawing of the cover graph of $P$ in the plane such that whenever $xy$ is an edge in the cover graph and $x<y$ in~$P$, the relation is represented by a curve from $x$ to $y$ going monotonically upwards. 
A \defin{diagram} of $P$ is a drawing of its cover graph in the plane in which every cover relation $x<y$ is represented by a curve from $x$ to $y$ that is monotone in the vertical direction.”

\begin{theorem}\label{thm:approx-algo}
    There exists a polynomial-time algorithm that, given a poset $P$ together with a planar diagram of $P$, 
    computes an embedding of $P$ into $\mathbb{R}^d$ with the product order, where $d = 96\dim(P)+672$.
\end{theorem}

As stated, the algorithm in \Cref{thm:approx-algo} 
assumes that a planar diagram of a poset is given.
Although an explicit drawing may require substantially more space than its combinatorial description, the required geometric information admits a compact representation.
Di Battista and Tamassia~\cite{DiBattista88} proved that, 
given a combinatorial embedding of a planar diagram, 
one can construct an equivalent \q{small} planar diagram in polynomial time.
Consequently, it suffices to provide the input poset together with a combinatorial embedding of one of its planar diagrams. 
We discuss this issue in greater detail in \Cref{ssec:topology}. We remark that deciding whether a poset admits a planar diagram is itself \textsf{NP}-hard, as shown by Garg and Tamassia~\cite{upwardtesting}.
To the best of our knowledge, it remains open whether computing dimension is \textsf{NP}-hard for posets with planar diagrams.

The algorithm is a direct consequence of our main structural result: posets with a planar diagram and dimension $d$ contain a standard example of order linear in $d$. 
We next provide some context for this structural result.

In 1941, Dushnik and Miller~\cite{DM41} introduced dimension together with its canonical obstruction to small values: the family of standard examples.
For each integer $n \geq 2$,
the \defin{standard example} of order $n$, denoted by \defin{$S_n$}, is a poset on $2n$ elements 
$a_1,\dots,a_n,b_1,\dots,b_n$
such that 
$a_1,\dots,a_n$ are pairwise incomparable, 
$b_1,\dots,b_n$ are pairwise incomparable, and
for all $i,j \in [n]$, 
we have $a_i < b_j$ in $S_n$ if and only if $i\neq j$.
See~\cref{fig:se_kelly_wheel}.
It is one of the first exercises in dimension theory to show that $\dim(S_n)=n$. 
Since dimension is a monotone parameter, 
$\dim(P)\geq n$ whenever $P$ contains a subposet isomorphic to $S_n$.

% \begin{figure}[h]
%   \centering
%   \includegraphics{figs/se.pdf}
%   \caption{
%         The standard example of order $6$. 
%     }
%    \label{fig:se}
% \end{figure}

Large standard examples are not the only way to drive dimension up. 
Indeed, several natural families have unbounded dimension while excluding the standard example of a fixed order,
e.g., 
incidence posets of complete graphs (as proved by Dushnik and Miller~\cite{DM41}), interval orders (see a tight asymptotic bound on their dimension by 
Füredi, Hajnal, Rödl, and Trotter~\cite{FHRT91}), adjacency posets of triangle-free graphs with large chromatic number (as shown by Felsner and Trotter \cite{FT00}).
This motivates the following notion.
The \defin{standard example number} of a poset $P$, denoted \defin{$\se(P)$}, 
is set to be $1$ if $P$ does not contain a subposet isomorphic to a standard example; 
otherwise, 
$\se(P)$ is the maximum integer $n$ such that $P$ contains a subposet isomorphic to $S_{n}$.
Clearly, for every poset $P$, we have $\se(P) \leq \dim(P)$. 
A class of posets $\calC$ is \defin{$\dim$-bounded} if there is a function $f$ 
such that $\dim(P) \leq f(\se(P))$ for every $P$ in $\calC$. 
As we discussed, the class of all posets is not $\dim$-bounded.

\begin{figure}[tp]
  \centering
  \includegraphics{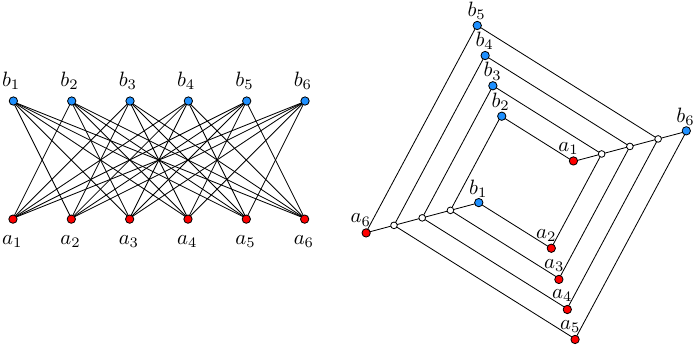}
  \caption{
    On the left, the standard example $S_6$ of order $6$. 
    On the right, the Kelly poset of order $6$: it has a planar diagram and contains a subposet isomorphic to $S_6$.
}
   \label{fig:se_kelly_wheel}
\end{figure}

Posets with planar diagrams can have arbitrarily large dimension, as proved by Kelly in 1981~\cite{K81}.  
%Posets with planar cover graphs can have arbitrarily large dimension, as proved by Trotter in 1978~\cite{T78}, see \cref{fig:se_kelly_wheel}.  
%In 1981, Kelly~\cite{K81} published a seminal construction of posets with planar diagrams and arbitrarily large dimension, see~\cref{fig:se_kelly_wheel} again. 
%A key remark is that all known constructions of planar posets with large dimension contain large standard examples. 
For each integer $n \geq 3$, the \defin{Kelly poset} of \defin{order} $n$ is a poset on $4n-4$ elements $a_1,\dots,a_n,b_1,\dots,b_n$ and $u_2,\dots,u_{n-1},v_2,\dots,v_{n-1}$ such that for all $i \in [n]$, $a_i$ is incomparable with $b_i$; for every $i \in \{2,\dots,n-2\}$, $a_i < u_i < b_{i-1}$ and $a_i < v_i < b_{i+1}$; $a_n \leq u_{n-1} \leq \dots \leq u_2 \leq b_1$; and $a_1 \leq v_2 \leq \dots \leq v_{n-1} \leq b_n$. 
Note that the elements $a_1,\dots,a_n,b_1,\dots,b_n$ induce a poset isomorphic to a standard example of order $n$.
Kelly posets admit planar diagrams, see \cref{fig:se_kelly_wheel}.

Since the early 1980's, it remained a challenge and perhaps the most important problem in poset theory to settle the following: 
\emph{Is it true that the class of posets with planar diagrams is $\dim$-bounded?}
%\later{For arxiv, say that we are authors of PD I :) Maybe use: In the first paper of this series~\cite{PD1}, we resolve the conjecture in the affirmative in full generality.}
We answer this question in the affirmative in the first paper of this series in the more general setting of posets with planar cover graphs~\cite{PD1}.
% The question was recently answered in the affirmative in the more general setting of posets with planar cover graphs by Blake, Hodor, Micek, Seweryn, and Trotter~\cite{PD1}.
Namely, we proved that $\dim(P) \in \Oh(\se(P)^8)$ for posets $P$ with planar cover graphs.
In the present paper, we study the original setting of planar diagrams and improve the bound from polynomial to linear.
% \begin{conjecture*}\hfill
%   \begin{enumerate}
%   \item The class of posets with planar diagrams is $\dim$-bounded.
%   \label{item:conj:diagram}
%   \item The class of posets with planar cover graphs is $\dim$-bounded.
%   \label{item:conj:cover-graph}
% \end{enumerate}
% \end{conjecture*}
% We believe that the first published reference to this question is
% an informal comment on page 119 in~\cite{Tro-book} published
% in 1992. However, the conjecture was circulating among researchers soon
% after the constructions illustrated in \cref{fig:se_kelly_wheel} appeared.
% In the first paper of this series~\cite{PD1}, we resolve the conjecture in the affirmative in full generality.
% Namely, we prove that $\dim(P) \in \Oh(\se(P)^8)$ for posets $P$ with planar cover graphs.
% In this paper, we study the more restrictive case of posets with planar diagrams, where we improve the mentioned bound to linear.

\begin{theorem}\label{thm:dim-boundedness}
  For every poset \(P\) with a planar diagram, 
  $\dim(P) \le 96 \se(P) + 672$.
\end{theorem}

% In the proof of~\Cref{thm:dim-boundedness}, given a poset $P$ and its planar diagram, we efficiently construct an embedding of $P$ into $\mathbb{R}^s$, where $s = 96\se(P) + 672$.
% Since $\se(P) \leq \dim(P)$ for every poset $P$, this way, we obtain~\Cref{thm:approx-algo}.
% Note that the result of~\cite{PD1} gives a polynomial-time polynomial approximation algorithm for dimension of posets with planar cover graphs.

Already in 1977, Trotter and Moore~\cite{TM77} showed that every poset with a planar diagram and a unique minimal element, called a zero, has dimension at most $3$.
In 1950, Dilworth~\cite{Dilworth1950} proved that for every poset~$P$, $\dim(P)$ is bounded by the width of $P$, i.e.\ the maximum size of an antichain in $P$. 
Standard examples show that no analogous bound holds in terms of height for arbitrary posets. 
However, posets with planar diagrams have dimension bounded by a linear function of their height, i.e.\ the maximum size of a chain in $P$, as proved by Joret, Micek, and Wiechert~\cite{JMW17}.

%Besides the general polynomial bound from~\cite{PD1}, there are two important cases with an improved bounds: (1) the every poset $P$ with a planar cover graph and a zero satisfies $\dim(P) \leq 2\se(P) + 2$, as proved by 
%Blake, Micek, and Trotter~\cite{BMT22}; 

Notably, classes of posets whose cover graphs have bounded treewidth (or even posets of bounded cliquewidth) 
are also dim-bounded, as proved by Joret, Micek, Pilipczuk, and Walczak~\cite{JMPW24}.
Since planar graphs exclude $K_5$ as a minor and graphs of treewidth less than $t$ exclude $K_{t+1}$ as a minor, a natural  generalization of the results of \cite{PD1} and \cite{JMPW24} would be the following statement.
\emph{For every positive integer $t$, the class of posets with cover graphs excluding $K_t$ as a minor is $\dim$-bounded.}
This remains open.

%\Cref{thm:dim-boundedness} improves the general bound of \cite{PD1}.
Another setting where a better bound than the one of \cite{PD1} is known are posets with planar cover graphs and a zero.
Blake, Micek, and Trotter~\cite{BMT22}, proved that for every such poset $P$, we have $\dim(P) \leq 2\se(P) + 2$.

The Kelly construction appears to play a central role in understanding dimension for posets with sparse cover graphs.
Joret, Micek, Pilipczuk, and Walczak~\cite{JMPW24} proved that, for every minor-closed class of graphs $\calC$, the class of posets with cover graphs in $\calC$ has bounded dimension if and only if $\calC$ excludes the cover graph of a Kelly poset.
For every poset $P$, we define the \defin{Kelly number} of $P$, denoted \defin{$\kelly(P)$}, to be $1$ if $P$ does not contain a subposet isomorphic to a Kelly poset; otherwise, 
$\kelly(P)$ is the maximum order of a Kelly poset isomorphic to a subposet of $P$.
Note that for every poset $P$, we have $\kelly(P) \leq \se(P)$.
% \later{For arxiv}
% In a manuscript in preparation, Planarity and dimension III, we show that a poset with a planar cover graph containing a large standard example must contain a large Kelly poset, and the function relating the standard example number and the Kelly number is linear.
% This implies that large dimensional posets with planar cover graphs contain large Kelly posets.
% In the case of planar diagrams, our proof-techniques give the Kelly posets \q{for free}.
In fact, we prove the following statement directly implying \cref{thm:dim-boundedness}.

\begin{theorem}\label{thm:dim-boundedness-Kelly}
  For every poset \(P\) with a planar diagram, 
  $\dim(P) \le 96 \kelly(P) + 672$.
\end{theorem}

\cref{thm:approx-algo,thm:dim-boundedness,thm:dim-boundedness-Kelly} follow from \cref{thm:algo-kelly}, i.e.\ the existence of an algorithm, which for a given poset with a planar diagram either finds a prescribed Kelly subposet or constructs an embedding of bounded dimension.

\begin{theorem}\label{thm:algo-kelly}
    There exists a polynomial-time algorithm that, given an integer $n \geq 3$, a poset $P$ together with a planar diagram of $P$, outputs either a subposet of $P$ isomorphic to the Kelly poset of order $n$ or an embedding of $P$ into $\mathbb{R}^d$ with the product order, where $d = 96n+576$.
\end{theorem}

For many years, there were no tools for forcing large standard examples in high-dimensional posets. 
This led Streib and Trotter~\cite{ST14} to propose a simpler problem of forcing long chains in highly dimensional posets,
or in other words, bounding dimension of posets in terms of their height. 
As already mentioned, standard examples witness that this is not possible in general.
This initiated a line of research establishing bounds on dimension in terms of height for several structurally sparse classes of posets~\cite{KMT19,GS21,JMMTWW16,W17,MW17,JMW18}.
Joret, Micek, and Wiechert~\cite{JMW17} showed that for every poset $P$ with a planar diagram, $\dim(P) \leq 192\hposet(P) + 96$.
\Cref{thm:dim-boundedness-Kelly} directly yields an improvement of the multiplicative constant from $192$ to $96$; a more careful analysis of the proof reduces it further to $48$.

% Finally, let us mention that in fact (similarly as in~\cite{PD1}) we prove a stronger result, which implies~\Cref{thm:dim-boundedness-Kelly} (see~\Cref{thm:technical} and the necessary definitions in~\Cref{sec:dim}).
% In particular, the stronger result implies that the class of posets that are subposets of posets with planar diagrams is dim-bounded (see~\cref{cor:subposets-dim-bounded}).
% \todo{Fill references}

\section{Outline of the paper}\label{sec:outline}

Our proof of~\Cref{thm:algo-kelly} is self-contained.   
However, we do not include proofs of the properties of poset dimension that are already covered in~\cite{PD1}: \cref{prop:dim-alternating-cycles,prop:dim_of_sum_incomparable_sets,prop:dim-components,prop:unfolding}.
In a stronger variant of~\Cref{thm:algo-kelly} that we prove, we also use the result of Di Battista and Tamassia~\cite{DiBattista88} on upward drawings of planar graphs.

The notation and objects introduced in the outline are introduced later again (perhaps more carefully).
However, the first part of the outline is crucial for understanding the setup of the proofs; it states the initial reductions and the remainder of the paper is devoted only to proving \cref{lem:singly-constrained-reduction,lem:singly-constrained-proof}, which imply all the main results.

We begin with an alternative perspective on the dimension of posets, presented in detail in~\cref{sec:dim}, which is more amenable to combinatorial arguments.
Let $P$ be a poset.
We denote by \defin{$\Inc(P)$} all incomparable pairs of elements of $P$.
For an integer $k$ with $k \geq 2$, a sequence $((a_1,b_1), \dots, (a_k,b_k))$ of pairs in $\Inc(P)$ is a \defin{strict alternating cycle} of size $k$ in~$P$ 
if $a_i\leq b_{j}$ in~$P$ for all $i,j\in[k]$ if and only if $j = i +1$, cyclically (that is, $(a_{k+1},b_{k+1}) = (a_1,b_1)$).
Let $I \subset \Inc(P)$.
We say that such a strict alternating cycle is \defin{contained} in $I$ if $(a_i,b_i) \in I$ for all $i \in [k]$. 
The subset $I$ is \defin{reversible} if $I$ does not contain a strict alternating cycle. 
A family \(\calS\) of subsets of \(\Inc(P)\) \defin{covers} $I$ if $I \subset \bigcup \calS$.
We define \defin{$\dim_P(I)$} as the minimum positive integer $d$ such that $I$ can be covered by $d$ reversible sets.
One can show that $\dim(P)=\dim_P(\Inc(P))$, see \Cref{prop:dim-alternating-cycles}. 
This alternative angle on poset dimension can be efficiently implemented. 
Indeed, given a reversible set $I$ of incomparable pairs in a poset $P$, 
one can efficiently compute a linear extension $L$ of $P$ reversing all the pairs in $P$, i.e., $b<a$ in $L$ for each $(a,b)\in I$; 
simply compute a topological ordering of $P\cup I^{-1}$, where $\defmath{I^{-1}}=\set{(b,a): (a,b)\in I}$ is the dual of $I$. 
Then, each such linear extension yields an ordering on a coordinate in the embedding into $\mathbb{R}^d$.
Therefore, all our algorithms output a covering of $I$ into $d$ reversible sets instead of an embedding into $\mathbb{R}^d$.

In fact, we work with a stronger notion of $\dim$-boundedness, where large $\dim_P(I)$ forces a large standard examples with incomparable pairs in $I$ and also Kelly posets \q{based} on $I$. 
Namely, for each integer $n \geq 3$, the \defin{Kelly subposet} of $P$ of \defin{order} $n$ \defin{based} on $I$ is a subposet of $P$ on $4n-4$ elements $a_1,\dots,a_n,b_1,\dots,b_n$ and $u_2,\dots,u_{n-1},v_2,\dots,v_{n-1}$ satisfying the same conditions as in the Kelly poset of order $n$ and additionally $(a_i,b_i) \in I$ for every $i \in [n]$. 

The algorithm in \cref{thm:algo-kelly} takes as input a diagram of a poset.
As discussed in the introduction, it suffices to take a combinatorial embedding of a diagram, see the definition and more details in \cref{ssec:topology}.

For a subset of incomparable pairs $I$ in a poset $P$, 
we define \defin{$\closure{I}$} to be the set of all the pairs $(a,b) \in \Inc(P)$ such that there exists $a'$ and $b'$ in $P$ with $(a',b) \in I$ and $(a,b') \in I$.

The following problem is central to our considerations.

\problemdef{\PDD}{
A tuple $(n,P,I,\calE)$ where $n$ is an integer with $n \geq 3$, $P$ is a poset, $I \subset \Inc(P)$, and $\calE$ is a combinatorial embedding of $P$ witnessing a planar diagram of $P$.
}
{
A Kelly subposet of $P$ of order $n$ based on $\closure{I}$ or a covering of $I$ by reversible sets in $P$.
}

For a non-decreasing function $f$, we say that an algorithm for \PDD is \defin{$f$-good} if for every input $(n,P,I,\calE)$, it outputs either a required Kelly subposet of $P$ or a covering consisting of at most $f(n)$ sets.
The main technical theorem that we prove in this paper, which implies \cref{thm:algo-kelly}, and so, also \cref{thm:approx-algo,thm:dim-boundedness,thm:dim-boundedness-Kelly} is the following.

\begin{theorem}\label{thm:main}
    Let $f(n) = 96n+576$ for every positive integer $n$.
    There exists an $f$-good polynomial-time algorithm solving \PDD\!\!.
\end{theorem}

The first step of our algorithm solving \PDD is a reduction to a so-called \q{singly constrained case}.
Given a poset \(P\) and an element \(x\) of \(P\), we say that a set \(I \subseteq \Inc(P)\) is \defin{singly constrained} by \(x\) in $P$ if for every \((a, b) \in I\) we have \(x \le b\) in \(P\).
The reduction exploits \q{unfoldings of posets}.
The idea was introduced by Streib and Trotter~\cite{ST14} in 2014, and it is inspired by the following classical application of layerings to graph colorings. 
A \defin{layering} of a graph $G$ is a family $(Z_i : i\in \NN)$ of pairwise disjoint subsets of $V(G)$ such that for every edge $uv$ of $G$, there exists $i \in \NN$ with $\{u,v\} \subset Z_i \cup Z_{i+1}$.
It is well-known that
given a graph $G$ and a
layering of $G$, 
if every layer induces a $k$-colorable graph, 
then $G$ is \(2k\)-colorable (one can use two disjoint
palettes of \(k\) colors each, one for even layers, and one for odd layers).
The counterpart of the above for posets is formulated in terms of the unfolding.
An \defin{unfolding} of a poset $P$ is a layering of the comparability graph of $P$.
If the union of any two consecutive sets in the unfolding induces a subposet of dimension at most \(d\), then the dimension of \(P\) is at most \(2d\), see~\cref{prop:unfolding}.
Unfoldings of $P$ that emerge from \textsc{bfs}-layerings of the comparability graph of $P$ are of particular interest. 

Streib and Trotter~\cite{ST14} used this idea to show that for every poset \(P\) with a planar cover graph, there exist a poset \(Q\) with a planar cover graph, an element, and a set \(J \subseteq \Inc(Q)\) that is singly constrained by $x$ in $Q$ with $\dim(P) \le 2\dim_Q(J)$.
See also~\cite[Subsection~3.6]{PD1}.
Later, Joret, Micek, and Wiechert~\cite{JMW17} with a more involved argument, showed that for every poset \(P\) with a planar diagram, there exist a poset \(Q\) with a planar diagram, an element \(x\) in $Q$, and a set \(J \subseteq \Inc(Q)\) that is singly constrained by $x$ in $Q$ with $\dim(P) \le 32\dim_Q(J)$.
In order to reduce \PDD to singly constrained setting, we adjust the setup, reprove, and algorithmize the reduction of Joret, Micek, and Wiechert.
Also, with a simple observation, we reduce the multiplicative constant from $32$ to $24$.
Here is a statement of the problem after this reduction:

\problemdef{\SCPDD}{
A tuple $(n,P,I,x_0,\calE)$ where $n$ is an integer with $n \geq 3$, $P$ is a poset, 
%$\calE$ is a combinatorial embedding of $P$, 
$x_0$ is an element of $P$, $I \subset \Inc(P)$ is singly constrained in $P$ by $x_0$, and 
$\calE$ is a combinatorial embedding of $P$ witnessing a planar diagram of $P$ with $x_0$ in the exterior face.
}
{
A Kelly subposet of $P$ of order $n$ based on $\closure{I}$ or a covering of $I$ by reversible sets in $P$.
}

Again, for a non-decreasing function $f$, we say that an algorithm for \SCPDD is \defin{$f$-good} if for every input $(n,P,I,x_0,\calE)$, it outputs either a required Kelly subposet or a covering consisting of at most $f(n)$ sets.
In \cref{sec:singly_constrained}, 
we describe the whole reduction which is encapsulated by the following lemma.
\begin{restatable}{lemma}{lemreduction}\label{lem:singly-constrained-reduction}
    Let $f$ be a non-decreasing function. 
    Suppose that there exists a polynomial-time $f$-good algorithm for \SCPDD\!\!.
    Then, there exists a polynomial-time $24f$-good algorithm for \PDD\!\!. 
\end{restatable}

Accordingly, to prove \cref{thm:main}, it suffices to design an algorithm for \SCPDD\!\!.

\begin{restatable}{lemma}{lemproof}\label{lem:singly-constrained-proof}
    Let $f(n) = 4n+24$ for every positive integer $n$.
    There exists an $f$-good polynomial-time algorithm solving \SCPDD\!\!.
\end{restatable}

We now discuss the algorithm witnessing \cref{lem:singly-constrained-proof} and all the necessary notions.
Let $(n,P,I,x_0,\calE)$ be an input tuple.
For every \((a, b) \in I\) we have \(x_0 \leq b\) in $P$.
Thus, for each $(a,b)\in I$, \(a \not \le x_0\) in $P$ 
as otherwise $a\leq x_0 \leq b$ in $P$.
Therefore, after removing all elements \(x\) with \(x < x_0\) in \(P\), we obtain
a convex subposet of $P$ which still satisfies all the requirements of the input.
Hence, without loss of generality, we may assume that $x_0$ is a minimal element of $P$.
By the discussion in \cref{ssec:topology}, in particular \Cref{thm:encoding-diagrams}, we may assume that we are given a diagram of $P$ encoded with number of bits polynomial in the size of $P$ such that every element of the poset is on a different vertical coordinate and $x_0$ is in the exterior face.
Note that the plane is always considered with the standard coordinate system, giving the notions of \defin{higher}, \defin{lower}, \defin{left of}, or \defin{right of}.
Let 
\[\defmath{B}=\{b \text{ in } P : x_0 \leq b \text{ in } P\} \ \text{ and } \ \defmath{A}=\{a \text{ in } P : a \notin B\}.\]
Thus, for every $(a,b) \in I$, we have $a \in A$ and $b \in B$.

In the first step, the algorithm identifies four \q{easy} reversible subsets of $\closure{I}$, see \cref{ssec:four-easy-sets}.
These sets are defined within a general framework illustrated by the following statement.

\begin{proposition}
\label{prop:reversible-from-order}
  %Let \(P\) be a poset, let $B$ be a subset of elements of $P$, and let $I \subset \Inc(P)$ such that for every $(a,b) \in I$, we have $b \in B$.  
  Let \(\preccurlyeq\) be a partial order on \(B\) such that 
  for every strict alternating cycle $((a_1,b_1),\ldots,(a_k,b_k))$ in $P$ contained in $I$ and for all $i,j\in[k]$ with $i\neq j$ we have
  either \(b_i \prec b_j\) or \(b_j \prec b_i\). Then, the set
  \[
    \ext(I,\preccurlyeq) = \{(a, b) \in I:
    \textrm{there is no \(b'\in B\)
    such that \(a \leq b'\) in $P$ and \(b' \prec b\)}\}
  \]
  is reversible in $P$.
\end{proposition}
\begin{proof}
    Suppose to the contrary that there is a strict alternating cycle \(((a_1, b_1), \ldots, (a_k, b_k))\) in $P$ with all the pairs in \(\ext(I,\preceq)\).
    Thus the elements in \(\set{b_1, \ldots, b_k}\) are linearly ordered by \(\preccurlyeq\).
    Without loss of generality, we assume that \(b_k \prec b_i\) for each \(i \in [k-1]\).
    In particular, we have \(b_{k} \prec b_{k-1}\) and \(a_{k-1} \le b_k\) in $P$,
    thus, \((a_{k-1}, b_{k-1}) \not \in \ext(I, \preceq)\), which is a contradiction.
\end{proof}

Let \defin{\(\preccurlyeq_{\uparrow}\)} be an ordering of the points in the plane such that for all \(q_1\) and \(q_2\) in $P$, if $q_1$ is not higher than $q_2$, then \(q_1 \preccurlyeq_{\uparrow} q_2\) and if $q_1$ is lower than $q_2$, then \(q_1 \prec_{\uparrow} q_2\). 
Note that $\preccurlyeq_\uparrow$ restricted to the elements of $P$ is a linear order since the elements in the fixed diagram have pairwise distinct vertical coordinates.
By \cref{prop:reversible-from-order}, the following sets are reversible in $P$:
\begin{align*}
    \defmath{J_1} &= \{(a,b) \in \closure{I} : \text{there is no $q' \in B$ with $a \leq q'$ in $P$ and $q' \prec_\uparrow b$}\},\\
    \defmath{J_2} &= \{(a,b) \in \closure{I} : \text{there is no $q'' \in B$ with $a \leq q''$ in $P$ and $b \prec_\uparrow q''$}\}.
\end{align*}

Let $\gamma_1$ and $\gamma_2$ be vertically monotone curves in the plane.
We say that \(\gamma_1\) is \defin{left} of \(\gamma_2\) if
\begin{enumerate}
    \item there exists a horizontal line $\ell$ intersecting both $\gamma_1$ and $\gamma_2$ such that the intersection of $\gamma_1$ with $\ell$ is left of the intersection of $\gamma_2$ with $\ell$, and
    \item for each horizontal line $\ell$ intersecting both $\gamma_1$ and $\gamma_2$, either $\gamma_1$ and $\gamma_2$ intersect $\ell$ at the same point or the intersection of $\gamma_1$ with $\ell$ is left of the intersection of $\gamma_2$ with $\ell$. 
\end{enumerate}
When \(\gamma_1\) is left of \(\gamma_2\), we also say that
\(\gamma_2\) is \defin{right} of \(\gamma_1\). 
See \cref{fig:curves-order}. 
Since we work with a fixed planar diagram of $P$, elements of $P$ are identified with points in the plane and edges of the cover graph of $P$ are identified with vertically monotone curves.
Points are also considered to be (degenerate) vertically monotone curves.
A \defin{witnessing path} in $P$ is a path in the cover graph of $P$ of ascending elements in $P$.
In particular, we also identify witnessing paths with vertically monotone curves.

Given two elements $u$ and $v$ in $P$, we may consider all the witnessing paths from $u$ to $v$ in $P$.
It is relatively simple to check (see \cref{claim:diagram:extreme_paths_exist}) that there exist \defin{leftmost} and \defin{rightmost} such paths, denoted by $\defmath{W_L(u,v)}$ and $\defmath{W_R(u,v)}$, respectively.
Namely, all witnessing paths from $u$ to $v$ in $P$ are either right of or equal to $W_L(u,v)$, and either left of or equal to $W_R(u,v)$.
For each $b \in B$, we abbreviate $\defmath{W_L(b)} = W_L(x_0,b)$ and $\defmath{W_R(b)} = W_R(x_0,b)$.
Let \defin{$\preccurlyeq_L$} and \defin{$\preccurlyeq_R$} be two orderings on $B$ such that for all $b_1,b_2 \in B$, we have $b_1 \preccurlyeq_L b_2$ if either $b_1 = b_2$ or $W_L(b_1)$ is left of $W_L(b_2)$, and $b_1 \preccurlyeq_R b_2$ if either $b_1 = b_2$ or $W_R(b_1)$ is left of $W_R(b_2)$.
We prove (\cref{claim:diagram:L_and_R_are_posets,claim:diagram:incomparable_b_are_comparable}) that $\preccurlyeq_L$ and $\preccurlyeq_R$ are partial orderings of $B$ such that two incomparable elements in $P$ are comparable in each of the orderings.
Since for a strict alternating cycle $((a_1,b_1),\dots,(a_k,b_k))$ in $P$, the elements $b_1,\dots,b_k$ form an antichain in $P$, we will be able to apply \cref{prop:reversible-from-order} to these orderings.
Along with a simple technical statement (\cref{prop:one-side-implies-the-other}), we get that the following sets are reversible in $P$:
\begin{align*}
    \defmath{J_3} &= \{(a,b) \in \closure{I} \setminus (J_1 \cup J_2) : \text{there is no $b' \in B$ with $a \leq b'$ in $P$, $b' \prec_L b$, and $b' \prec_R b$}\},\\
    \defmath{J_4} &= \{(a,b) \in \closure{I} \setminus (J_1 \cup J_2) : \text{there is no $b'' \in B$ with $a \leq b''$ in $P$, $b \prec_L b''$, and $b \prec_R b''$}\}.
\end{align*}

The next step of the algorithm is finding four reversible sets covering all the pairs $(a,b) \in \closure{I} \setminus (J_1 \cup J_2 \cup J_3 \cup J_4)$ where $a$ is a \q{bottom element}, see \cref{ssec:bottom}.
An element $a\in A$ is a \defin{bottom element} if there exists $b^*,b^{**} \in B$ and witnessing paths $W^*$ from $a$ to $b^*$ and $W^{**}$ from $a$ to $b^{**}$ such that $W^*$ is left of $x_0$ and $W^{**}$ is right of $x_0$.
See \cref{fig:bottom-element}.

Since $x_0$ is a minimal element of $P$, every edge of the cover graph of $P$ incident with \(x_0\) has its lowest point in \(x_0\).  
Let \(e_1, \ldots, e_t\) be the edges of the cover graph of \(P\) incident with \(x_0\) listed from left to right.
The vertex \(x_0\) is on the boundary
of the exterior face, so for some \(i \in [t]\) the edges \(e_{i-1}\) and \(e_i\) are consecutive edges in the facial walk along the exterior face (the indices are interpreted
cyclically, so \(e_0 = e_t\)).
We choose a non-trivial vertically monotone curve \defin{$\gamma_\infty$} in the exterior face of $P$ which has its lowest point in \(x_0\) and is otherwise disjoint from the diagram.
% The curve \(\gamma_\infty\) may be left of \(e_1\) or right of \(e_t\) or 
% between \(e_{i-1}\) and \(e_i\) for some \(i \in [t]\).
% In particular, $\gamma_\infty$ is consistent with every witnessing path from $x_0$.
% By \cref{prop:diagram:for_consistent_one_distinction_is_enough}, this yields that $\gamma_\infty$ is either left or right of every non-trivial witnessing path from $x_0$.

Let \defin{$B_L$} be the set of all $b \in B$ such that there is a witnessing path from $x_0$ to $b$ that is left of $\gamma_\infty$, and let \defin{$B_R$} be the set of all $b \in B$ such that there is a witnessing path from $x_0$ to $b$ that is right of $\gamma_\infty$.
By \cref{claim:diagram:e_infty}, $\set{B_L,B_R}$ is a partition of $B$.

Let \defin{\(\preccurlyeq_L^{\infty}\)} be a relation on $B$ defined as follows: 
for all $b_1,b_2 \in B$, 
$b_1 \preccurlyeq_L^{\infty} b_2$ if 
($b_1\in B_R$ and $b_2\in B_L$) or
($b_1,b_2 \in B_R$ and $b_1\preceq_L b_2$) or
($b_1,b_2 \in B_L$ and $b_1\preceq_L b_2$).
Note that \(\preccurlyeq_L^{\infty}\) partially orders $B$.
Moreover, for all $b_1,b_2 \in B$ if $b_1$ and $b_2$ are comparable in $\preccurlyeq_L$, then they are comparable in $\preccurlyeq_L^\infty$.
In particular, by \cref{claim:diagram:incomparable_b_are_comparable},
for all $b_1,b_2\in B$ with $b_1\parallel b_2$ in $P$, $b_1$ and $b_2$ are comparable in \(\preccurlyeq_L^{\infty}\). 
Symmetrically, 
let \defin{\(\preccurlyeq_R^{\infty}\)} be a relation on $B$ defined as follows: 
for all $b_1,b_2 \in B$, 
$b_1 \preccurlyeq_R^{\infty} b_2$ if 
($b_1\in B_R$ and $b_2\in B_L$) or
($b_1,b_2 \in B_R$ and $b_1\preceq_R b_2$) or
($b_1,b_2 \in B_L$ and $b_1\preceq_R b_2$).
Again, 
note that \(\preccurlyeq_R^{\infty}\) partially orders $B$ and 
for all $b_1,b_2\in B$ with $b_1\parallel b_2$ in $P$, $b_1$ and $b_2$ are comparable in \(\preccurlyeq_R^{\infty}\). 
We define
\begin{align*}
    \defmath{J_5} &= \set{(a,b) \in \closure{I} \setminus (J_1 \cup \dots \cup J_4) : \text{there is no $b' \in B$ such that $a \leq b'$ in $P$ and $b' \prec_L^\infty b$}},\\
    \defmath{J_6} &= \set{(a,b) \in \closure{I} \setminus (J_1 \cup \dots \cup J_4) : \text{there is no $b'' \in B$ such that $a \leq b''$ in $P$ and $b \prec_R^\infty b''$}},\\
    \defmath{J_7} &= \set{(a,b) \in \closure{I} \setminus (J_1 \cup \dots \cup J_6) : \text{$a$ is a bottom element and } b \in B_L},\\
    \defmath{J_8} &= \set{(a,b) \in  \closure{I} \setminus (J_1 \cup \dots \cup J_6) : \text{$a$ is a bottom element and } b \in B_R}.
\end{align*}
By \cref{prop:reversible-from-order}, $J_5$ and $J_6$ are reversible.
With a less trivial argument we also prove that $J_7$ and $J_8$ are reversible, see \cref{claim:diagram:dim_I_7_leq_2}.
Note that for $(a,b) \in \closure{I} \setminus (J_1\cup \dots \cup J_8)$, $a$ is not a bottom element.
From now on, we work with the set
    \[\defmath{I'} = \set{(a,b) \in  \closure{I} \setminus (J_1 \cup \dots \cup J_4) : \text{$a$ is not a bottom element}}.\]
As discussed, $I'\subseteq \closure{I}\setminus(J_1\cup\cdots\cup J_8)$. 

In the next step, the algorithm splits the remaining pairs into two groups: left and right.
Each group is processed separately.
One more reversible set is going to be removed from  each group.
See \cref{ssec:left-right}.

For each $(a,b) \in I'$, we say that $(a,b)$ is a \defin{left pair} if for every $d \in B$ with $a \leq d$ in $P$, all witnessing paths from $a$ to $d$ are left of $W_L(b)$.
Symmetrically, we say that $(a,b)$ is a \defin{right pair} if for every $d \in B$ with $a \leq d$ in $P$, all witnessing paths from $a$ to $d$ are right of $W_R(b)$.
\begin{align*}
    \defmath{I_L'} = \{(a,b) \in I' : \text{$(a,b)$ is a left pair}\} \ \text{ and } \ \defmath{I_R'} = \{(a,b) \in I' : \text{$(a,b)$ is a right pair}\}.
\end{align*}
We prove that $I_L'$ and $I_R'$ form a partition of $I'$, see \cref{claim:diagram:left_right_parititon}.
For every pair $(a,b) \in \Inc(P)$, we define:
\begin{align*}
    \defmath{B'(a,b)} &= \{ b' \in B : \text{$a \leq b'$ in $P$, $b' \prec_L b$, and $b' \prec_R b$}\}, \\
    \defmath{B''(a,b)} &= \{ b'' \in B : \text{$a \leq b''$ in $P$, $b \prec_L b''$, and $b \prec_R b''$}\},
\end{align*}
and
\begin{align*}
    \defmath{B_*'(a,b)} &= \{ b' \in B'(a,b) : \text{$W_L(a,b')$ is left of $W_L(b')$}\}, \\
    \defmath{B_*''(a,b)} &= \{ b'' \in B''(a,b) : \text{$W_R(a,b'')$ is right of $W_R(b'')$}\}.
\end{align*}
We prove that the following two sets are reversible, see \cref{clm:IRempty}
\begin{align*}
    \defmath{J_L} = \{(a,b) \in I_L' : B'_*(a,b) = \emptyset\} \ \text{ and } \ \defmath{J_R} = \{(a,b) \in I_R' : B''_*(a,b) = \emptyset\}.
\end{align*}
Accordingly, we define two sets that still need to be covered by reversible sets
\begin{align*}
    \defmath{I_L} = \{ (a,b) \in I_L' : B_*'(a,b) \neq \emptyset\} \ \text{ and } \ \defmath{I_R} = \{ (a,b) \in I_R' : B_*''(a,b) \neq \emptyset\}.
\end{align*}

Now, we are going to  study certain regions associated to pairs in $I_L \cup I_R$ and elements in $P$.
By symmetry, we study only \q{right} regions, see \cref{ssec:regions}.
Let $(a,b) \in I_R$, $b' \in B'(a,b)$, and $b'' \in B''_*(a,b)$.
The paths $W_R(b')$, $W_R(b'')$, $W_R(a,b'')$, and $W_R(a,b')$ give the region in a natural way, see \cref{fig:regionR} and see \cref{prop:tuple-is-valid} for a detailed definition and study of these regions, and in particular, the proof that they must \q{look like} in the figure.
We denote such a region as \defin{$\calR_R(a,b',b'')$} and the symmetric one for a pair $(a,b) \in I_L$, $b' \in B'_*(a,b)$, and $b'' \in B''(a,b)$ as \defin{$\calR_L(a,b',b'')$}.

    \begin{figure}
        \centering
        \includegraphics{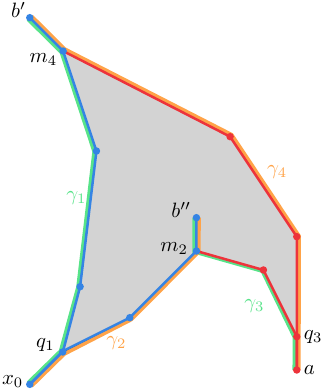}
        \caption{Illustration of a region given by elements $a$, $b'$, and $b''$ satisfying certain conditions.
        Important curves and points are named as in~\Cref{prop:tuple-is-valid}. 
        The region $\calR(a,b',b'')$ is the gray area.
        %The closed curve \(\gamma\) is the boundary of the gray area. 
        % \jedrzej{It's nice, I would just consider making it larger.} 
        % \piotr{Perhaps, we could add labels for $\gamma_1,\gamma_2,\gamma_3,\gamma_4$.}
        % \michal{Iterated on the figure.} \jedrzej{I like it. A very small comment that you can ignore is that we assume elements to have distinct $y$-coordinates.}\michal{Perturbed the elements.}
        } 
        \label{fig:regionR}
    \end{figure}

Next, for $(a,b) \in I_R$ and $b'' \in B''_*(a,b)$ we find $b' \in B'(a,b)$ such that the region is the smallest possible.
We set \defin{$z'(a,b)$} to be one of the rightmost (according to $\prec_R$) elements of $B'(a,b)$ with the lowest possible vertical coordinate.
See a more precise definition in \cref{ssec:smallest-region} and see \cref{fig:smallest-region}.
% Let $B_R'(a,b)$ be the set of all the maximal elements with respect to $\prec_R$ of the elements in $B'(a,b)$.
% It follows that $B_R'(a,b)$ is an antichain with respect to $\prec_R$.
% In other words for all $d,d' \in B_R'(a,b)$, either $W_R(d)$ is a subpath of $W_R(d')$ or $W_R(d')$ is a subpath of $W_R(d)$.
% In particular, the elements of $B_R'(a,b)$ can be labeled $b_1',\dots,b_m'$ so that $W_R(b_i')$ is a subpath of $W_R(b_{i+1}')$ for every $i \in [m-1]$.
% We set \defin{$z'(a,b)$} to be $b_1'$, i.e.\ the lowest element of $B_R'(a,b)$ in $P$.
We prove that $\calR_R(a,z'(a,b),b'')$ is indeed the smallest region in \cref{prop:smallest-region}.
The symmetric object is denoted by \defin{$z''(a,b)$}.

In \cref{ssec:regions,ssec:smallest-region,ssec:sacs} we derive many useful properties of regions and strict alternating cycles contained in $I_R$ or $I_L$.

In the last part (\cref{ssec:auxiliary}), we define four auxiliary oriented graphs on the sets $I_L$ and $I_R$.
First, we discuss $I_R$.
Let \defin{$H_R$} be an oriented graph with the vertex set $I_R$ such that $(a,b) \in I_R$ is connected by an edge oriented towards $(c,d) \in I_R$ if there exists $f \in B$ such that
\begin{enumerateNumHR}
    \item $d \preccurlyeq_R f \prec_R b$, 
    \item $(c,f) \in I_R$,
    \item $((a,b),(c,f))$ is a strict alternating cycle in $P$.
\end{enumerateNumHR}
See \Cref{fig:auxiliaryH}. 
% Note that item~\ref{item:def-HR-acyclic} guarantees that $H_R$ is acyclic. \jedrzej{It's not important}
The crucial property of $H_R$ is that every strict alternating cycle in $P$ contained in $I_R$ has two pairs connected by an edge in $H_R$, see \cref{cor:sacs}. 
Therefore, every proper coloring of vertices of $H_R$ into $\chi$ colors, gives a partition of $I_R$ into $\chi$ reversible sets in $P$.

Let \defin{$H_R'$} be an oriented graph with the vertex set $I_R$ such that $(a,b) \in I_R$ is connected by an edge oriented towards $(c,d) \in I_R$ if 
\begin{enumerateNumHRprim}
    \item $d \prec_R b$,
    \item $B_*''(a,b) \subset B_*''(c,d)$,
    \label{item:HR'-def-containment}
    \item for every $b'' \in B_*''(a,b)$, we have $\calR_R(a,z'(a,b),b'') \subset \calR_R(c,z'(c,d),b'')$.
\end{enumerateNumHRprim}
We say that an edge of $H'_R$ is a \defin{cycle edge} whenever it is a strict alternating cycle in $P$.
The first condition gives that $H'_R$ is acyclic.
Therefore, the following value for every $(a,b) \in I_R$ can be easily computed
\[\defmath{\ell_R(a,b)} = \text{the maximum number of cycle edges in a directed path in $H_R'$ starting in $(a,b)$.}\]
We prove that for every edge $((a,b),(c,d))$ in $H_R$, we have $\ell(a,b) \geq \ell(c,d) + 1$, see \cref{prop:ell-plus-1}.
This gives us an easy way to partition vertices of $H_R$ into independent sets:
\[
\defmath{J_{R,\alpha}} = \{(a,b) \in I_R : \ell_R(a,b) =\alpha\}\quad\textrm{for each $\alpha\in\mathbb{N}$}.
\]
As we discussed, by \cref{cor:sacs}, this gives a partition of $I_R$ into reversible sets in $P$.
This would complete the proof except one detail, 
we do not control the number of reversible sets in this partition; or, in other words, 
we do not have a bound on the number of cycle edges on a single path in $H_R'$. 
To circumvent this obstacle, we merge $J_{R,\alpha}$'s into larger sets. 
Our candidates for reversible sets partitioning $I_R$ are 
\[\defmath{I_{R,\alpha}} = \{(a,b) \in I_R : \ell_R(a,b) \equiv \alpha \bmod (2n+7) \}.\]
This time, we have at most $2n+7$ sets in the partition. 
It is not clear though if $\set{I_{R,\alpha}: \alpha\in[2n+7]}$ partition vertices of $H_R$ into independent sets. 
Recall that for each edge $((a,b),(c,d))$ in $H_R$, we know that $\ell(a,b)\geq \ell(c,d)+1$. 
Therefore, if one of the $I_{R,\alpha}$ contains an edge $((a,b),(c,d))$ of $H_R$, 
it must be that $\ell(a,b)\geq\ell(c,d)+2n+7$.
We conclude that if for each edge $((a,b),(c,d))$ of $H_R$, we have $\ell(a,b)<\ell(c,d)+2n+7$, then 
$\set{I_{R,\alpha} : \alpha \in [2n+7]}$ is a partition of $I_R$ into independent sets in $H_R$, and therefore, 
by \cref{cor:sacs}, $\set{I_{R,\alpha} : \alpha \in [2n+7]}$ is a partition of $I_R$ into reversible sets in $P$, as desired, see \cref{prop:IRalpha-reversible}.

It remains to deal with the case when there is an edge $((a,b),(c,d))$ of $H_R$ with $\ell_R(a,b)\geq\ell_R(c,d)+2n+7$.
The main technical part of the proof is showing that in this setting we can find a Kelly subposet of $P$ of order $n$ based on $I_R$.
This is proved in \Cref{lemma:shortcuts}.

The idea of the proof of \cref{lemma:shortcuts} is the following.
We fix an edge $((a,b),(c,d))$ of $H_R$ with $\ell_R(a,b)\geq\ell_R(c,d)+2n+7$. 
We fix a path in $H'_R$ starting in $(a,b)$ containing $\ell_R(a,b)$ cycle edges.
By the transitivity of $H'_R$ (\Cref{obs:transitive}), we may assume that this path is of the form 
\[(a,b)(a_1,b_1)(c_1,d_1)(a_2,b_2)(c_2,d_2)\dots(a_m,b_m)(c_m,d_m),\]
    where $((a_j,b_j),(c_j,d_j))$ is a cycle edge in $H'$, and either $(c_{j-1},d_{j-1}) = (a_j,b_j)$ or $((c_{j-1},d_{j-1}),(a_{j},b_{j}))$ is an edge in $H'$ 
     for each $j$, 
     (where $(c_0,d_0) = (a,b)$).
Then, using \cref{prop:sac-region-containment}, we construct many regions nested in each other.
These regions are \q{based} on some fixed element $b'' \in B_*''(a,b)\subseteq B_*''(a_j,b_j)\subseteq B_*''(c_j,d_j)$ for each $j$ (the containment holds by~\ref{item:HR'-def-containment}), and this gives the \q{left arm} of the sought Kelly poset.
The other \q{arm} is found by showing that a large difference of the value of $\ell_R$ between $(a,b)$ and $(c,d)$ forces some elements to \q{escape} the inner regions, giving the required comparabilities.
This is illustrated in \cref{fig:outline}.

The symmetric objects and statements can be obtained by mirroring the diagram.
However, for completeness of the algorithm, we restate the definitions for $I_L$.
Let \defin{$H_L$} be an oriented graph with the vertex set $I_L$ such that $(a,b) \in I_L$ is connected by an edge oriented towards $(c,d) \in I_L$ if there exists $f \in B$ such that
\begin{enumerateNumHL}
    \item $b \prec_L f \preccurlyeq_L d$,
    \item $(c,f) \in I_L$,
    \item $((a,b),(c,f))$ is a strict alternating cycle in $P$.
\end{enumerateNumHL}
Let \defin{$H_L'$} be an oriented graph with the vertex set $I_L$ such that $(a,b) \in I_L$ is connected by an edge oriented towards $(c,d) \in I_L$ if 
\begin{enumerateNumHLprim}
    \item $b \prec_L d$,
    \item $B_*'(a,b) \subset B_*'(c,d)$,
    \item for every $b' \in B_*'(a,b)$, we have $\calR_L(a,b',z''(a,b)) \subset \calR_L(c,b',z''(c,d))$.
\end{enumerateNumHLprim}
We say that an edge of $H'_L$ is a \defin{cycle edge} whenever it is a strict alternating cycle in $P$.
For every $(a,b) \in I_L$, we set
\[\defmath{\ell_L(a,b)} = \text{the maximum number of cycle edges in a directed path in $H_L'$ starting in $(a,b)$.}\]

We wrap up this outline with a concise description of our algorithm solving \SCPDD and witnessing \cref{lem:singly-constrained-proof}: \cref{algo}.

\begin{algorithm*}
        \begin{algorithmic}[1]
            \State $\backslash\backslash$ \Cref{lemma:shortcuts} and its symmetric version give the following two procedures
            \Function{FindKellyR}{$n,P,x_0,\calE,I_R,H'_R,(a,b),(c,d)$}
            \State $\backslash\backslash$ with a premise that $((a,b),(c,d)) \in E(H_R)$ and $\ell_R(a,b) \geq \ell_R(c,d) + (2n+7)$
                \State \Return the Kelly subposet of $P$ of order $n$ based on $I_R$
            \EndFunction
            \Function{FindKellyL}{$n,P,x_0,\calE,I_L,H_L',(a,b),(c,d)$}
            \State $\backslash\backslash$ with a premise that $((a,b),(c,d)) \in E(H_L)$ and $\ell_L(a,b) \geq \ell_L(c,d) + (2n+7)$
                \State \Return the Kelly subposet of $P$ of order $n$ based on $I_L$
            \EndFunction
            \State
            \Function{Solve}{$n,P,I,x_0,\calE$}
                \State compute $J_1$, $J_2$, $J_3$, $J_4$, $J_5$, $J_6$, $J_7$, $J_8$, $J_L$, $J_R$
                \State compute $I_L$ and $I_R$
                \State compute $H_L$, $H_L'$, $\ell_L$, $H_R$, $H_R'$, $\ell_R$
                \If {there exists $((a,b),(c,d))\in E(H_R)$ with $\ell_R(a,b) \geq \ell_R(c,d) + (2n+7)$ }
                \State $\backslash\backslash$ in this case, $\kelly_P(I_R) \geq n$
                \State \Return \textsc{FindKellyR}$(n,P,x_0,\calE,I_R,H'_R,(a,b),(c,d))$
                \ElsIf {there exists $((a,b),(c,d))\in E(H_L)$ with $\ell_L(a,b) \geq \ell_L(c,d) + (2n+7)$ }
                \State $\backslash\backslash$ in this case, $\kelly_P(I_L) \geq n$
                \State \Return \textsc{FindKellyL}$(n,P,x_0,\calE,I_L,H'_L,(a,b),(c,d))$
                \Else 
                \State $\backslash\backslash$ in this case, $\dim_P(I) \leq 4n+24$
                \State \Return $\{J_1,J_2,J_3,J_4,J_5,J_6,J_7,J_8,J_L,J_R\}$
                \State \ \ \ $\cup \{\{(a,b) \in I_{R} : \ell_R(a,b) \equiv \alpha \bmod (2n+7)\} : \alpha \in [2n+7]\} $
                \State \ \ \ $\cup \{\{(a,b) \in I_{L} : \ell_L(a,b) \equiv \alpha \bmod (2n+7)\} : \alpha \in [2n+7]\}$
                \EndIf 
            \EndFunction
        \end{algorithmic}
        \caption{}
        \label{algo}
    \end{algorithm*}

All the objects that the algorithm must compute can be easily found in polynomial time, and all the conditions that the algorithm must check can be easily verified in polynomial time.
To show the correctness of the algorithm, we must prove that the sets $J_1$, $J_2$, $J_3$, $J_4$, $J_5$, $J_6$, $J_7$, $J_8$, $J_L$, and $J_R$ are reversible (\cref{cor:J1J2,cor:J3J4,claim:diagram:dim_I_7_leq_2,clm:IRempty}), $I_L$ and $I_R$ are indeed the only remaining pairs (\cref{claim:diagram:left_right_parititon}), we must give the procedures \textsc{FindKellyR} and \textsc{FindKellyL} (\cref{lemma:shortcuts}), and we must prove that if the algorithm reaches \q{else} statements, then the sets $I_{R,\alpha}$ and $I_{L,\alpha}$ for each $\alpha \in [2n+7]$ are indeed reversible (\cref{prop:IRalpha-reversible}).
All these statements will complete the proof of \cref{lem:singly-constrained-proof}.

\begin{figure}[tp]
  \begin{center}
    \includegraphics{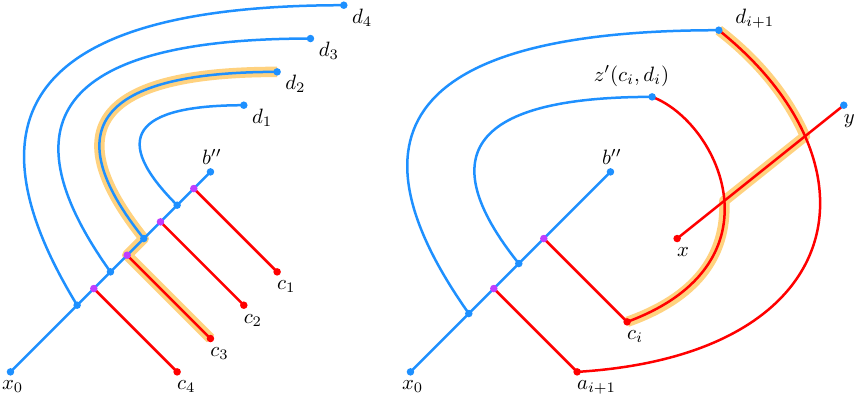}
  \end{center}
  \caption{
  An illustration of the proof of \cref{lemma:shortcuts}. 
  The fixed path contains cycle edges in $H'$ of the form $((a_i,b_i),(c_i,d_i))$ and $(c_i,d_i)$ and $(a_{i+1},b_{i+1})$ are either equal or connected by an edge in $H'$.
  The standard example (and Kelly poset) will be given by the family  of incomparable pairs $\set{(c_i,d_i) : i\in[n]}$. 
  On the left, we highlight how the relation $c_i<d_j$ in $P$ for all $i>j$ is forced in the sequence of nested regions.
  We prove that the elements $c_i$ and $d_i$ are \q{attached} to $W_R(b'')$ in the depicted order.
  On the right, we highlight how the relation $c_i<d_j$ in $P$ for all $i<j$ is forced.
  We prove that there is an element $x$ in one of the inner regions and an element $y$ outside one of the outer regions and $x < y$ in $P$ forces the required comparabilities.
  }
  \label{fig:outline}
\end{figure}

\section{Preliminaries}
\label{sec:preliminaries}
We denote by \defin{$\mathbb{R}$} the set of real numbers and by \defin{$\NN$} the set of nonnegative integers.
For a positive integer $k$, we write \defin{$[k]$} as a
compact form of $\{1,\dots,k\}$.
The plane is always considered with the standard coordinate system, giving the notions of \defin{higher}, \defin{lower}, and more to the \defin{left} or \defin{right}.

\subsection{Graphs}\label{ssec:graphs}
We consider simple and finite graphs.  
Unless we say otherwise, all graphs are undirected. 
The vertex set of a graph $G$ is denoted by \defin{\(V(G)\)} and the edge set of $G$ is denoted by \defin{\(E(G)\)}.
Let $G_1$ and $G_2$ be two graphs. 
The \defin{union} of $G_1$ and $G_2$, denoted by \defin{$G_1 \cup G_2$} is the graph with the vertex set $V(G_1)\cup V(G_2)$ and the edge set $E(G_1)\cup E(G_2)$.

Let $k$ be a nonnegative integer.
A \defin{path} is a graph 
with the vertex set $\set{v_0,\ldots,v_k}$ and the 
edges $\set{v_{i-1}v_{i} : i \in [k]}$, where $v_0,\dots,v_k$ are pairwise distinct. 
%The number of vertices of a path is its \defin{order}.
We often refer to a path by a natural sequence of its vertices, 
writing, say, $v_0\cdots v_k$ or equivalently $v_k \cdots v_0$.
 %\michal{Is there an ambiguity with the edge notation \(uv\)?}
 %\jedrzej{It's the same in Diestel, we think that we should just leave it as it is.}
 % \michal{I am fine with that, although we may consider making a small remark. I leave it up to your taste}
Writing a path as $v_0\cdots v_k$ fixes an underlying orientation of the path, where $v_0$ is the first vertex and $v_k$ is the last vertex.
We say that $v_0$ and $v_k$ are the \defin{endpoints} of $v_0 \cdots v_k$.
We also say that $v_0 \cdots v_k$ \defin{starts} in $v_0$ and \defin{ends} in $v_k$.
%When $W=v_0\cdots v_k$ is a path and $i,j \in \{0,\dots,k\}$ we write $v_i[W]v_j$ for the \defin{subpath} $v_i\cdots v_j$.
A \defin{cycle} is a graph with at least three vertices such that removing each of its edges gives a path. 
% \jedrzej{We changed the notation for paths. Now, we use "between" for paths and "from to" for witnessing path when we know the ordering. Keep this in mind and try to change succesively.}
A \defin{forest} is a graph with no cycles.
A \defin{tree} is a connected forest.
Let $T$ be a tree, and let $u,v$ be two vertices in $T$.
We write \defin{$u[T]v$} to denote the unique path in $T$ with endpoints $u$ and $v$ and with an orientation from $u$ to $v$. 
%We consider $u[T]v$ to have a fixed underlying orientation from $u$ to $v$.

We use the following convenient notation for manipulating paths in graphs. 
Let $G$ be a graph.
%For two paths $U$ and $V$ in $G$, we say that $U$ and $V$ are \defin{internally disjoint} if every vertex common to $U$ and $V$ is an endpoint of both $U$ and $V$.
Let \(v_0, \ldots, v_k\) be vertices of \(G\), and let \(T_1, \ldots, T_k\) be trees in \(G\) such that each \(T_i\) contains the vertices \(v_{i-1}\) and \(v_i\). Then we denote by \defin{\(v_0 [T_1] v_1 [T_2] \cdots [T_k] v_k\)} the union of the paths \(v_{i-1} [T_i] v_i\). See \cref{fig:concatenating_paths}.
%Whenever we use this notation, the resulting graph \(u_0 [U_1] u_1 [U_2] \cdots [U_k] u_k\) is either a path or a cycle. %\heather{In general, I don't think it's restricted to paths and cycles. Consider \(a[U]b[V]c[W]a[X]d\) where the four subpaths here are pairwise internally disjoint. But perhaps we are just saying that every time we use the notation, the result will be a path or a cycle.}
%\michal{Good point! this could be a bit confusing. I removed the mention that the paths \(u_{i-1}[U_i]u_i\) are pairwise internally disjoint.}
When $T_i = v_{i-1}v_i$, then we may omit \q{\([T_i]\)} in this notation. 
When the resulting graph is a path, we consider it with the orientation from $v_0$ to $v_k$.
For example, for a path \(W = v_0 \cdots v_5\), we have \(v_0[W]v_2 v_3 [W] v_5 = W\).

%\piotr{We discussed it with Michał. Michał has a green light to edit this paragraph. We don't know yet if we want $T_L$.}

% \begin{align*}
%     \maxsw(H,(a,b)) &= \text{the largest weight of a directed path in $H$ starting in $(a,b)$},\\
%     \maxew(H,(a,b)) &= \text{the largest weight of a directed path in $H$ ending in $(a,b)$}.
% \end{align*}

\subsection{Orders and posets}
% \jedrzej{We rewrote this section a little bit. Some decisions were made. 
%(1) We don't want to talk about $\prec$ at this point. 
%(2) We want to define "partially orders" and "poset" as close to each other as possible as they are the same things.
%(3) We removed some stuff that we believe is too much (a footnote).
%(4) We define incomparable elements.
%}
Let $X$ be a set. 
A \defin{partial order} on $X$ is a binary relation $R$ on \(X\)
that is \defin{reflexive} (for every \(x \in X\), \((x,x) \in R\)), \defin{antisymmetric} (for all \(x,y\in X\), if \((x,y) \in R\) and \((y,x) \in R\), then \(x=y\)), and \defin{transitive} (for all \(x,y,z \in X\), if \((x,y) \in R\) and \((y,z) \in R\), then \((x,z) \in R\)). 
A partial order $R$ is \defin{linear} if
for all $x,y \in X$, we have $(x,y) \in R$ or $(y,x) \in R$.
When \(R\) is a partial order (resp.\ linear order) on \(X\), we say that \(R\) \defin{partially orders} (resp.\ \defin{linearly orders}) \(X\).
When $R$ partially orders (resp.\ linearly orders) $X$, we say that $P = (X, R)$ is a \defin{partially ordered set}, or \defin{poset} for short (resp.\ \defin{linearly ordered set}).
% \footnote{Note that a poset \(P = (X, \le)\) carries the same information as the order relation \(\le\). More formally, we can restore the ground set \(X\) from the relation \(\le\) itself: by reflexivity we have \(X = \{x : x \le x\}\).
% However, in this paper posets will always have a planar cover graph. We use the symbol \(\preccurlyeq\) to denote an arbitrary partial order.}
We refer to the set \(X\) as the \defin{ground set} of \(P\), and to the members of \(X\) as the \defin{elements} of \(P\).
Note that $(X,R^{-1})$ is also a poset, we call it the \defin{dual} of~$P$ and denote it by~\defin{$P^{-1}$}.

Let $x$ and $y$ be elements of~$P$.
We say that $x$ and $y$ are \defin{comparable} in~$P$ (or in $R$)  if \((x,y) \in R\) or \((y, x) \in R\).
We use the notation \q{$x \leq y$ in~$P$} whenever $(x,y) \in R$, and \q{$x < y$ in~$P$} whenever $(x,y) \in R$ and $x \neq y$.
Furthermore, we say that $x$ and $y$ are \defin{incomparable} in~$P$ (or in $R$) if they are not comparable in $R$. 
In this case, we write \q{$x \parallel y$ in~$P$}.
Let \defin{\(\Inc(P)\)} denote the set of all ordered pairs \((x, y)\) of elements of~$P$ with \(x \parallel y\) in \(P\).
% For a subset $I \subset \Inc(P)$, we define
% \begin{align*}
%     \defmath{\pi_1(I)} &= \{a \in X : \text{ there exists $b'$ in $P$ with $(a,b') \in I$} \} \text{ and }\\
%     \defmath{\pi_2(I)} &= \{b \in X : \text{ there exists $a'$ in $P$ with $(a',b) \in I$} \}.
% \end{align*}
We define the \defin{closure} of $I$ as
\[\defmath{\closure{I}} = \{(a,b) \in \Inc(P) : \text{there exists $a'$ and $b'$ in $P$ such that $(a',b) \in I$ and $(a,b') \in I$}\}.\]

An element $x$ of~$P$ is \defin{minimal} in~$P$ if for every element $y$ of~$P$, $y \leq x$ in~$P$ implies $y = x$.
Symmetrically, an element $x$ of~$P$ is \defin{maximal} in~$P$ if for every element $y$ of~$P$, $x \leq y$ in~$P$ implies $y = x$.
A \defin{chain} in~$P$ is a set of elements in~$P$ such that every two elements are comparable in~$P$, 
and an \defin{antichain} in \(P\) is a set of elements in~$P$ such that every two distinct elements are incomparable in~$P$. 
%\jedrzej{I don't see why we need this.}
% A chain in \(P\) consists of elements \(x_1, \ldots, x_n\) with \(x_1 < \cdots < x_n\) in \(P\). 
% The elements \(x_1\) and \(x_n\) are respectively the \defin{minimal} and the \defin{maximal} element of the chain.

\begin{figure}[tp]
  \begin{center}
    \includegraphics{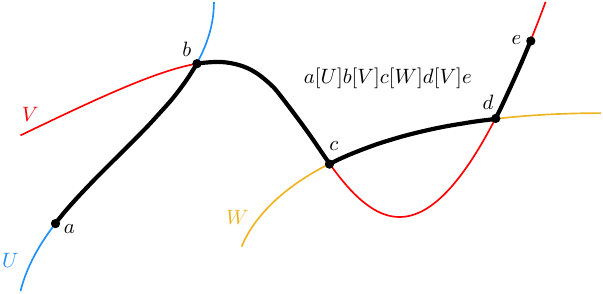}
  \end{center}
  \caption{
    A sample application of the notation for path concatenation.
  }
  \label{fig:concatenating_paths}
\end{figure}

%\jedrzej{Move this to the dimension subsection.}
% All posets in this paper are finite, except for \(\mathbb{R}^d\) which is always equipped with the product order: \((x_1, \ldots, x_d) \le (y_1, \ldots, y_d)\) in \(\mathbb{R}^d\) if and only if \(x_i \le y_i\) for each \(i \in [d]\).

The \defin{comparability graph} of \(P\) is a graph on the ground set of \(P\) in which two distinct vertices are adjacent if and only if they are comparable in \(P\). 
An element \(x\) of~$P$ is \defin{covered} by an element \(y\) of~$P$ (and \(y\) \defin{covers} \(x\)) in \(P\) if
\(x < y\) in~$P$ and there is no element \(z\) of~$P$ with \(x < z < y\) in \(P\).
The \defin{cover graph} of \(P\)
is a subgraph of the comparability graph in which two vertices are adjacent if either one of them covers the other.
Observe that the dual \(P^{-1}\) has the same comparability graph and cover graph as \(P\).

Two posets $P$ and $Q$ are \defin{isomorphic} if there exists a bijection $f$ from the ground set of~$P$ to the ground set of $Q$ such that for all elements $x$ and $y$ of~$P$, we have $x \leq y$ in~$P$ if and only if $f(x) \leq f(y)$ in $Q$.
For a subset $Y$ of the ground set of~$P$, the \defin{subposet induced} by $Y$ in~$P$ is the poset $Q$ with the ground set $Y$ such that for all $x,y \in Y$, we have $x \leq y$ in $Q$ if and only if $x \leq y$ in~$P$.
A poset $Q$ is a \defin{subposet} of a poset \(P\) if it is a subposet induced by some subset of the ground set of~$P$ in~$P$.
A subposet \(Q\) of a poset \(P\) is \defin{convex} if for all elements \(x,y,z\) with \(x < y < z\) in \(P\), if \(x\) and \(z\) are elements of \(Q\), then also \(y\) is an element of \(Q\). 
Note that if \(Q\) is a convex subposet of \(P\), then the cover graph of \(Q\) is a subgraph of the cover graph of \(P\).

% For an element \(x\) of \(P\), we denote by \defin{\(P - x\)}
% the subposet of \(P\) induced by the set of all elements of \(P\) except \(x\).

A \defin{finite} poset is a poset with a finite ground set.
All posets in this paper are finite, unless stated otherwise. 
A \defin{component} of a poset is a subposet induced by a connected component of its cover graph.
A poset is \defin{connected} if it has exactly one component.

\subsection{Dimension of posets}\label{sec:dim}
In the introduction, we presented a concise geometric definition of
dimension of posets. 
However, we (and most other researchers) work with a combinatorial equivalent. 
In this section, we discuss equivalent definitions of dimension and prove its basic properties.
Many statements are taken from~\cite{PD1} and adjusted to the algorithmic setting.

Let \(P\) be a poset.
A \defin{linear extension} of \(P\) is a linearly ordered set $L$ on the same ground set as $P$ such that $x\leq y$ in~$P$ implies $x\leq y$ in $L$ for all elements $x$ and $y$ of~$P$.
Every poset has a linear extension, and furthermore, for all incomparable elements \(x\) and \(y\) of~$P$ there exists a linear extension $L$ with \(x < y\) in $L$
%, and a linear extension $L'$ in which \(y < x\) in $L'$ 
(we prove an even stronger statement in \cref{prop:dim-alternating-cycles}). 
Therefore, if $\calL$ is the set of all linear extensions of~$P$ then for all elements $x$ and $y$ of~$P$, we have
\begin{equation}
\label{eq:all-linear-extensions}
\textrm{$x\leq y$ in~$P$ if and only if  $x\leq y$ in $L$ for each $L\in\calL$.} 
\end{equation}

Recall that the dimension of a poset \(P\) is the least positive integer \(d\) such that \(P\) is isomorphic to
a subposet of \(\mathbb{R}^d\) (ordered by the product order).
If \(Q\) is a finite subposet of \(\mathbb{R}^d\), then by slightly perturbing the coordinates of the points we can obtain an isomorphic subposet \(Q'\) of \(\mathbb{R}^d\) such that for each \(i \in [d]\), the \(i\)th coordinates of the elements of \(Q'\) are pairwise distinct. 
%For example, if \(Q\) is a subposet of \(\mathbb{R}^d\) such that for each \(i \in [d]\), the \(i\)th coordinates of the elements of \(Q\) are pairwise distinct, 
Then for each \(i \in [d]\), the linear order on the ground set of \(Q'\) given by the increasing \(i\)th coordinate yields a linear extension of \(Q'\).
This justifies an equivalent, combinatorial definition of dimension.

The dimension of \(P\) is the least positive integer \(d\) for which there exist \(d\) linear extensions
\(L_1, \ldots, L_d\) of \(P\) such that for all elements $x$ and $y$ of \(P\), we have
\begin{align*}
\textrm{$x\leq y$ in~$P$ if and only if  $x\leq y$ in $L_i$ for each $i\in[d]$.} 
\end{align*}
Indeed, a linear extension \(L_i\) can be seen as the order of elements in the \(i\)th coordinate in \(\mathbb{R}^d\),
and conversely, the order of elements in the \(i\)th coordinate in \(\mathbb{R}^d\) forms a linear extension. 
By~\eqref{eq:all-linear-extensions}, the dimension is well-defined.
In fact, Hiraguchi~\cite{H55} showed that an \(n\)-element poset with \(n \ge 4\), has dimension at most \( n/2 \). 
%\michal{check the reference}

A subset $I \subseteq \Inc(P)$ is \defin{reversible} in \(P\) if there exists a linear extension $L$ of~$P$ with
$y < x$ in $L$ for all $(x,y)\in I$.
A family \(\calS\) of subsets of \(\Inc(P)\) \defin{covers} $I$ if $I \subset \bigcup \calS$.
In this case, sometimes we say that $\calS$ is a \defin{covering} of $I$.
%\later{Sometimes coverings are multisets, technically it is a problem. If we have spare time, we should fix it (the same in PD1).}
We state arguably the most useful equivalent definition of dimension of posets.

% \jedrzej{We decided not to prove this.}

\begin{obs} \label{prop:redefine_dimension}
    For every poset \(P\), the dimension of~$P$ is the least positive integer $d$ for which 
          $\Inc(P)$ can be covered by $d$ reversible sets.
\end{obs}

% \begin{proposition} \label{prop:redefine_dimension}
%     Dimension of a poset $P$ is equal to the minimum positive integer $d$ such that $\Inc(P)$ can be covered by $d$ reversible sets.
% \end{proposition}
% \begin{proof}
%     TODO
%     \heather{Here's a first draft. Feel free to edit/trim:\\ 
%     Let $\{I_1, \ldots, I_d\}$ be a minimum size cover of $\Inc(P)$ by reversible sets. For each $i\in [d]$, let $\leq_i$ be the linear extension of~$P$ which reverses each pair in $I_i$. Let $x,y\in P$. If $x$ and $y$ are comparable in~$P$ with $x\leq_P y$ then $x\leq_i y$ for each $i\in [d]$ by the definition of linear extension, so $(x,y)\in \bigcap_{i=1}^d \leq_i$ as desired. Now suppose $x$ and $y$ are incomparable in~$P$. Then $(x,y), (y,x)\in \Inc(P)$. So there exists $j,k\in [d]$ with $(x,y)\in I_j$ and $(y,x)\in I_k$. So $y\leq_j x$ and $x\leq_k y$. So $(x,y)\not\in \bigcap_{i=1}^d \leq_i$. Therefore $\leq_P = \bigcap_{i=1}^d \leq_i$, so $\dim(P) \leq d$.
% \newline
%     Next, let $\{\leq_1, \ldots, \leq_{\dim(P)}\}$ be a smallest set of linear extensions of~$P$ such that $\leq_P = \bigcap_{i=1}^{\dim(P)} \leq_i$. For each $i \in [\dim(P)]$, set $I_i = \{(x,y)\in \Inc(P): y\leq_i x\}$. Let $(x,y)\in \Inc(P)$ be arbitrary. Then $x\not\leq_P y$, so there exists $j\in [\dim(P)]$ with $y\leq_j x$. So $(x,y)\in I_j$. Since this is holds for all pairs in $\Inc(P)$, $\{I_1, \ldots, I_{\dim(P)}\}$ is a cover of $\Inc(P)$. So $d\leq \dim(P)$.
%     }
% \end{proof}

Let $P$ be a poset.
Motivated by the above, for every \(I \subseteq \Inc(P)\), we define the \defin{dimension} of $I$ in~$P$, denoted by \defin{$\dim_P(I)$}, as the minimum positive integer $d$ such that $I$ can be covered by $d$ reversible sets.
%\piotr{Maybe $\dim(P,I)$?}
Note that by \cref{prop:redefine_dimension}, $\dim(P)=\dim_P(\Inc(P))$.

Similarly, we define a restricted version of the standard example number.
We say that a set $I\subseteq \Inc(P)$ \defin{induces} a standard example in~$P$ if $|I| \geq 2$ and for all distinct $(a,b), (a',b')\in I$, 
we have $a < b'$ and $a'<b$ in~$P$.  
For every $I\subseteq\Inc(P)$, 
let \defin{$\se_P(I)$} be defined as $1$ when there is no subset of $I$ inducing a standard example in~$P$; otherwise, \defin{$\se_P(I)$} is the maximum size of a subset of $I$ that induces a standard example in~$P$.
Note that $\se(P) = \se_P(\Inc(P))$.
For an integer $n \geq 3$ and $I \subset \Inc(P)$, the \defin{Kelly subposet} of $P$ of \defin{order} $n$ \defin{based} on $I$ is a subposet of $P$ induced by a set $\{c_1,\dots,c_n,d_1,\dots,d_n,u_2,\dots,u_{n-1},v_2,\dots,v_{n-1}\}$ satisfying
\begin{itemize}
    \item $(c_i,d_i) \in I$ for every $i \in [n]$,
    \item $c_n \leq u_{n-1} \leq \dots \leq u_2 \leq d_1$ in $P$,
    \item $c_1 \leq v_2 \leq \dots \leq v_{n-1} \leq d_n$ in $P$,
    \item $c_i < u_i < d_{i-1}$ in $P$ for every integer $2 \leq i \leq n-1$,
    \item $c_i < v_i < d_{i+1}$ in $P$ for every integer $2 \leq i \leq n-1$.
\end{itemize}
For each $I \subset \Inc(P)$, we define \defin{$\kelly_P(I)$} as the maximum order of a Kelly subposet in $P$ based on $I$.

It turns out that there is a relatively simple criterion to verify if a given $I \subset \Inc(P)$ is reversible.
For an integer $k$ with $k \geq 2$, a sequence $((a_1,b_1), \dots, (a_k,b_k))$ of pairs in $\Inc(P)$ is an \defin{alternating cycle} of size $k$ in~$P$ 
if $a_i\leq b_{i+1}$ in~$P$ for all $i\in[k]$, cyclically 
(that is, $b_{k+1} = b_1$).
We say that $I \subset \Inc(P)$ \defin{contains} an alternating cycle \(((a_1, b_1), \ldots, (a_k, b_k))\) if all pairs \((a_i, b_i)\) belong to  $I$.
An alternating cycle $((a_1,b_1),\dots,(a_k,b_k))$ is \defin{strict} if for all $i,j \in [k]$, we have
$a_i\le b_{j}$ in~$P$ if and only if $j=i+1$ (cyclically).
Note that in this case, $\{a_1, \dots, a_k\}$ and $\{b_1, \dots, b_k\}$ are $k$-element antichains in~$P$.
Note also that in alternating cycles, we allow that $a_i=b_{i+1}$ for some or even all values of $i$.
Trotter and Moore~\cite{TM77} made the following elementary observation that has proven over time to be far-reaching in nature.
\begin{proposition}[{\cite[Proposition~3]{PD1}}]\label{prop:dim-alternating-cycles}
Let $P$ be a poset and let $I \subset \Inc(P)$.
The following conditions are equivalent:
\begin{enumerate}
    \item\label{item:I-reversible} $I$ is reversible,
    \item\label{item:I-with-no-alternating-cycle} $I$ does not contain an alternating cycle,
    \item\label{item:I-with-no-strict-alternating-cycle} $I$ does not contain a strict alternating cycle.
\end{enumerate}
\end{proposition}

The property of being reversible can be checked algorithmically.

\begin{proposition}\label{prop:detect-reversibility}
    There exists a polynomial-time algorithm that takes a poset $P$ and $I \subset \Inc(P)$ and decides if $I$ is reversible in $P$.
\end{proposition}
\begin{proof}
    For a given poset $P$ and $I \subset \Inc(P)$, the algorithm constructs an auxiliary oriented graph $H$ on the vertex set $I$.
    There is an oriented edge from $(a,b) \in I$ to $(c,d) \in I$ if $a \leq d$ in $P$.
    Now, $I$ contains an alternating cycle in $P$ if and only if $H$ contains a directed cycle.
    Thus, by~\Cref{prop:dim-alternating-cycles}, $I$ is reversible in $P$ if and only if $H$ is acyclic, which can be verified in polynomial time.
\end{proof}

Next, we develop some easy properties of dimension.
For every poset $P$ and for each $I \subset \Inc(P)$, we set $\defmath{I^{-1}} = \{(b,a) : (a,b) \in I\}$.
Note that if $I \subset \Inc(P)$, then $I^{-1} \subset \Inc(P^{-1})$.
The following observation is straightforward.

\begin{obs}
    For every poset $P$ and $I \subset \Inc(P)$, we have $\dim_{P}(I) = \dim_{P^{-1}}(I^{-1})$.
\end{obs}

Let $P$ be a poset, $I\subseteq \Inc(P)$, 
and let $\{I_i : i \in [s]\}$ be a covering of $I$. 
The following inequalities are trivial but still very useful
\[
\textstyle\dim_P(I) \leq \dim_P\left(\bigcup_{i\in[s]} I_i \right) \leq \sum_{i\in[s]} \dim_P(I_i).
\]
The next proposition describes a situation when we get a stronger bound, i.e.\ $\dim_P(I)\leq \max\{\dim_P(I_i): i\in[s]\}$ (so somehow the problem becomes local with respect to the covering).
For a poset $P$ and a set $I \subseteq \Inc(P)$, a subsets $I_1,\dots,I_s \subseteq I$ are \defin{splitting partition} of $I$ in $P$ if they are pairwise disjoint, their union is $I$, and for every strict alternating cycle in~$P$ contained in $I$, there is $i \in [s]$ such that all the pairs in the strict alternating cycle are in $I_i$.

\begin{proposition}[{\cite[Proposition 5]{PD1}}]\label{prop:dim_of_sum_incomparable_sets}
    Let $P$ be a poset, let $I \subseteq \Inc(P)$, and let $I_1,\dots,I_s \subseteq I$ be a splitting partition of $I$ in $P$.
    Then, $\dim_P(I) = \max\{ \dim_P(I_i): i \in [s]\}$.
    Moreover, if there exists a positive integer $d$ and for every $i \in [s]$, there exists a covering $\{I_{i,j} : j \in [d]\}$ of $I_i$ by reversible sets in $P$, then, $\{\bigcup_{i \in [s]}I_{i,j} : j \in [d]\}$ is a covering of $I$ by reversible sets in $P$.
\end{proposition}

Back in the 1950s, Hiraguchi~\cite{H55} proved that if a poset \(P\) is the union of
disjoint chains, then \(\dim(P) \le 2\), and otherwise
$\dim(P)$ equals the maximum dimension of a component of~$P$. 
We will use a slightly refined version of this statement devised for subsets of incomparable pairs.

\begin{proposition}[{\cite[Proposition 6]{PD1}}]\label{prop:dim-components}
Let $P$ be a poset, and let $I\subseteq \Inc(P)$.
Let $C_1,\dots,C_s$ be components of~$P$, let $I_{i} = I \cap \Inc(C_i)$, and suppose that there exists an integer $d \geq 2$ such that $\{I_{i,j}: j \in [d]\}$ is a covering of $I_i$ by reversible sets in $P$.
Also, define $J_1 = \{(a,b) \in I : a \text{ in } C_i, \ b \text{ in } C_j, \ \text{and} \ i < j\}$ and $J_2 = \{(a,b) \in I : a \text{ in } C_i, b \text{ in } C_j, \ \text{and} \ i > j\}$.
For each $j \in [d]$, let
    \[
        J_j' = 
        \begin{cases}
            \bigcup_{i=1}^s I_{i,j} \cup J_j \ &\text{if} \ j \in \{1,2\},\\
            \bigcup_{i=1}^s I_{i,j} \ &\text{otherwise}.
        \end{cases}
    \]
Then, $\{J_j' : j \in [d]\}$ is a covering of $I$ by reversible sets in $P$.
In particular, 
\[\dim_P(I) \leq \max (\{\dim_{C_i}(I_i) : i \in [s]\} \cup \{2\}).\]
\end{proposition}

\subsection{Topology and planarity}\label{ssec:topology}
A plane is the set $\mathbb{R}^2$ equipped with the standard topology.
For a set $S$ of points in the plane, we write \defin{$\partial S$} for the topological boundary of $S$ and \defin{$\Int S$} for the topological interior of $S$.
Additionally, we define the \defin{exterior} of $S$ as the interior of the complement of $S$ in the plane.

A \defin{simple curve} $\gamma$ in the plane is the image of an injective continuous map of a closed segment into the plane. 
% \michal{in the case of a closed curve it is an image of any circle, and in the case of a non-closed curve it is the image of the specific interval \([0, 1]\). Also, why is it a `map' rather than a simple `function'?}
% \jedrzej{(1) map is traditionally used for functions going to the plane (2) Now it is any circle and any segment.}
In this case, the \defin{endpoints} of $\gamma$ are the images of the endpoints of the segment, while the \defin{interior} of $\gamma$ is the set of non-endpoint points in \(\gamma\).
We say that $\gamma$ \defin{connects} its endpoints.
%\michal{A small issue: the interior of a segment in the plane is empty. Alternative: ``the \defin{interior} of $\gamma$ is the set of non-endpoint points in \(\gamma\) (the interior of a curve is a different concept than the topological interior).''}
%\jedrzej{Ok, but I would omit the bracket as we use emph, so it is clear that this is a new definition.}
A simple curve $\gamma$ is \defin{vertically monotone} if for every real number $y_0$, the set $\{(x,y_0) : x \in \mathbb{R}\}$ intersects $\gamma$ in at most one point.
A \defin{simple closed curve} in the plane is the image of an injective continuous map of a circle into the plane. 
Since all combinatorial objects considered in this paper are finite, we always assume that curves are finite unions of segments in the plane.
%\michal{I don't understand this sentence. What is the connection between ``combinatorial objects'' and curves on the plane?}
%\michal{Is it ever actually useful to assume that curves are piecewise linear?}

Let $\gamma$ be a simple closed curve.
The Jordan Curve Theorem\footnote{Since in this paper we only consider curves that are finite unions of segments, we only need the Jordan Curve Theorem for Polygons~\cite[Chapter~4]{Diestel-book}.} states that the complement of~$\gamma$ in the plane consists of two arc-connected components, one bounded $B$ and one unbounded~$U$. 
Moreover, the boundary of each of the components is equal to $\gamma$.
%\michal{arc-connected components undefined.}
%\jedrzej{Continous functions undefined... I think that there is some limit :)}
We define the \defin{region} of $\gamma$ as the union of $\gamma$ and $B$. %, and the \defin{exterior} of the region of $\gamma$ as $U$.
%Note that again by Jordan Curve Theorem, the boundary of the region of $\gamma$ is exactly $\gamma$. 
% \michal{I don't see how JCT implies it.} 
% \jedrzej{It doesn't imply that, it states that. See above.}
%\heather{I believe we are defining ``boundary'' here, so should we use emph?} \michal{I would rather not have a notion of boundary different than the topological boundary=frontier. And if we assume that the reader knows what is a continuous function, it is safe to assume that they also know what is boundary.} 
%\michal{Do we use the fact that one of the regions is homeomorphich to a closed disk? If yes, then the Jordan-Schoenflies Theorem is more appropriate}
%\jedrzej{I don't think so}
In particular, if $\calR$ is the region of $\gamma$, then $\Int \calR = B$, $\partial\calR = \gamma$, and the exterior of $\calR$ is $U$.

\begin{proposition}\label{obs:region_containment}
    Let $\gamma_1$ and $\gamma_2$ be simple closed curves and let $\Gamma_1$ and $\Gamma_2$ be the regions of $\gamma_1$ and $\gamma_2$ respectively.
    If $\partial \Gamma_1 \subset \Gamma_2$, then $\Gamma_1 \subset \Gamma_2$.
\end{proposition}
\begin{proof}
    Since $\partial \Gamma_1 \subset \Gamma_2$, the exterior of \(\Gamma_2\) is disjoint from \(\partial \Gamma_1\). Since the exterior of \(\Gamma_2\) is connected, it must be contained in some component of the complement of $\partial \Gamma_1$ in the plane. 
    Since the exterior of \(\Gamma_2\) is unbounded, it must be contained in the exterior of \(\Gamma_1\). 
    Thus, \(\Gamma_1 \subseteq \Gamma_2\).
\end{proof}

A \defin{drawing} of a graph $G$ is a function $f$ assigning a subset of the plane to each vertex and each edge of $G$ such that $f(u)$ is a point in the plane for every vertex $u$ in $G$ and no two vertices are assigned to the same point; for every edge $uv$ in $G$, $f(uv)$ is a curve in the plane connecting $f(u)$ and $f(v)$ disjoint from $f(w)$ for every vertex $w$ of $G$ distinct from $u$ and $v$.
We say that a drawing $f$ of a graph $G$ is \defin{planar} if the interiors of the images of edges of $G$ under $f$ are pairwise disjoint.
A graph \(G\) is \defin{planar} if it admits a planar drawing.
A \defin{plane graph} is a planar graph with a fixed planar drawing.
In plane graphs, we usually identify vertices with the assigned points in the plane and edges with the assigned curves in the plane.
More generally, we identify subgraphs of plane graphs with unions of respective vertices and edges in the plane.
For example, every path in a plane graph is a simple curve, and every cycle in a plane graph is a simple closed curve.
The complement of a plane graph $G$ in the plane is a union of arc-connected components.
The topological closure of such a component is a
\defin{face} of $G$.
The only unbounded face is called the \defin{exterior face} of $G$.

Let $P$ be a poset and let $G$ be the cover graph of $P$.
A drawing $f$ of $G$ is a \defin{diagram} of $P$ if for every edge $uv$ of $G$, $f(uv)$ is a vertically monotone curve and additionally, if $u < v$ in $P$, then $u$ is \q{lower} in the drawing than $v$, meaning that if $f(u) = (x_u,y_u)$ and $f(v) = (x_v,y_v)$, then $y_u < y_v$.
In turn, a diagram of a poset $P$ is identified with a plane graph $G$, which is isomorphic to the cover graph of $P$, and its drawing is a diagram of $P$.

Only for algorithmic purposes, we introduce the notion of the combinatorial embedding. 
For an oriented graph $G$, a family $\set{(\mathcal{E}_v,<_v)}_{v\in V(G)}$ is a \defin{combinatorial embedding} of $G$ if 
$\mathcal{E}_v$ is a set of all edges incident to $v$ ordered cyclically by $<_v$, for each $v\in V(G)$. 
Given a drawing of a planar graph $G$, 
the edges incident to each vertex of $G$ are naturally ordered (clockwise) in the drawing. 
This gives rise to the notion of combinatorial embedding of $G$ \defin{witnessed by the drawing}.
A combinatorial embedding of a poset is a combinatorial embedding of its cover graph.

Our algorithm from \cref{thm:approx-algo} takes as input a poset and a combinatorial embedding witnessing a planar diagram of this poset.
Sometimes in the proof, we need an actual drawing of the poset.
However, %to keep the time complexity of the algorithm tame, 
we do not want the time complexity of the algorithm to depend on the arbitrarily complex drawing. 
%i.e.\ we need to work with a drawing want it to be representable by a small number of bits.
Di~Battista and Tamassia~\cite{DiBattista88} proved that every upward planar acyclic digraph on $n$ vertices admits an upward planar drawing where vertices are mapped into points of $n \times (2n-4)$-grid and edges are represented by lines with at most two \q{bends} where the bending points are also mapped into the grid.
Say that such a drawing of an acyclic digraph is \defin{efficient}.
Moreover, Di~Battista and Tamassia gave a linear-time algorithm that given a combinatorial embedding of an acyclic digraph witnessing an upward-planar drawing, outputs an efficient drawing.
Note that, in an efficient drawing of an $n$-vertex digraph, the mapping of each vertex or edge can be encoded in $\Oh(\log n)$ bits.
See also the result of Gronemann~\cite{Gronemann16}, who reduced the number of required bends to one.

It will also be convenient for us to assume that every element of a poset is represented in the drawing with a distinct vertical coordinate.
We say that such a drawing is \defin{distinguishing}.
It is easy to transform an efficient drawing on a grid as described above to an efficient distinguishing drawing by doubling the width and height of the grid and by introducing at most one more bend to each of the edges.

We need a slightly more precise statement.
An \defin{st-graph} is an acyclic directed graph with exactly one source and exactly one sink, and additionally, the source and the sink are connected by an edge.
A directed graph admits an upward planar drawing if and only if it is a subgraph of a planar st-graph as proved by Di~Battista and Tamassia~\cite[Theorem~4.3]{DiBattista88}.
On the other hand, testing such a property is an \textsf{NP}-hard problem as proved by Garg and Tamassia~\cite{upwardtesting}.
In the first step of the algorithm of Di~Battista and Tamassia for finding an efficient drawing of a digraph, they refine the digraph to an st-graph by adding edges using the given combinatorial embedding.
Sometimes we want to preserve a certain minimal element of our poset on the exterior face of the drawing.
Note that this is possible as we have a lot of freedom in the refinement step.
We summarize what we need in the next statement.

\begin{theorem}[\cite{DiBattista88}]\label{thm:encoding-diagrams}
    %\piotr{minimal or maximal?}
    There exists a linear-time algorithm that given an $n$-element poset $P$, its combinatorial embedding witnessing a planar diagram of $P$, and a minimal element $x_0$ of $P$ that is in the exterior face in the witnessed drawing, returns a distinguishing efficient planar diagram of $P$ with $x_0$ in the exterior face.
    % There exists a linear-time algorithm that, given a poset $P$ and its combinatorial embedding witnessing a planar diagram of $P$, returns an efficient planar diagram of $P$.
\end{theorem}

\section{Reduction to a singly constrained instance: unfolding}
\label{sec:singly_constrained}
Given a poset \(P\) and an element \(x\) of \(P\), we say that a set \(I\subseteq \Inc(P)\) is \defin{singly constrained} by \(x\) in $P$ if for every \((a, b) \in I\) we have \(x \le b\) in \(P\).
In this section, we prove the following lemma that we restate for convenience.
This is an adjustment, algorithmization, and a slight improvement of the result of Joret, Micek, and Wiechert~\cite{JMW17}.

\lemreduction*

% The key tool needed to prove \Cref{lem:singly-constrained-reduction} is \q{unfolding} introduced by Streib and Trotter \cite{ST14}.
% They showed that for every poset \(P\) with a planar cover graph, there exist a poset \(Q\) with a planar cover graph, an element \(x\) in $Q$ such that \(Q - x\) is a subposet of \(P\) or \(P^{-1}\), and a set \(J \subseteq \Inc(Q)\) that is singly constrained by $x$ in $Q$ with $\dim(P) \le 2\dim_Q(J)$.
% See also~\cite[Subsection~3.6]{PD1}.
% Later, Joret, Micek, and Wiechert~\cite{JMW17} with a much more involved proof, showed that for every poset \(P\) with a planar diagram, there exist a convex subposet \(Q\) of $P$ or $P^{-1}$, an element \(x\) in the exterior face of $Q$, and a set \(J \subseteq \Inc(Q)\) that is singly constrained by $x$ in $Q$ with $\dim(P) \le 32\dim_Q(J)$.
% In this subsection, we introduce all the necessary tools, and we reprove (with an adjusted setup) and algorithmize the reduction of Joret, Micek, and Wiechert.
% Also, with a simple observation, we manage to reduce the multiplicative constant from $32$ to $24$.

Let $P$ be a poset. 
An \defin{upset} in \(P\) is a subset \(U\) of elements of \(P\) such that for
all elements \(x, y\) with \(x \le y\) in \(P\), if \(x \in U\), then \(y \in U\).
For each element \(x\) of \(P\), we denote by \defin{\(U_P[x]\)} the upset in \(P\) consisting
of all elements \(y\) such that \(x \le y\) in \(P\).
Dually, a \defin{downset} in \(P\) is a subset \(D\) of elements of \(P\) such that for
all elements \(x, y\) with \(x \le y\) in \(P\), if \(y \in D\), then \(x \in D\).
For each element \(y\) of \(P\), we denote by \defin{\(D_P[y]\)} the downset in \(P\) consisting
of all elements \(x\) such that \(x \le y\) in \(P\).
For a set of elements \(Z\) in \(P\), we denote by \defin{\(D_P[Z]\)} and \defin{\(U_P[Z]\)} the
unions \(\bigcup_{z \in Z} D_P[z]\) and \(\bigcup_{z \in Z} U_P[z]\), respectively.
Note that every upset and downset in a poset induces a convex subposet.

An \defin{unfolding} of a poset \(P\) is a family \((Z_i : i \in \NN)\) of pairwise disjoint subsets of the ground set of \(P\) such that for all nonnegative integers $i,j$ and elements \(x \in Z_i\), \(y \in Z_j\) with \(x \le y\) in \(P\), either
\(i = j\), or \(|i-j| = 1\) and \(i\) is even (and thus \(j\) is odd).
%For convenience, for every nonnegative integer $k$, we write $Z_{\geq k} = \bigcup_{i \geq k} Z_i$ and $Z_{< k} = \bigcup_{0 \leq i < k} Z_i$ where the unions go over integers.
Note that for each $k$, the set \(Z_k\) is a downset if \(k\) is even, and an upset if \(k\) is odd. 
%Moreover, $Z_{\geq k}$ and $Z_{< k+1}$ are downsets in~$P$ whenever $k$ is even, and they are upsets in~$P$ whenever $k$ is odd.
%In particular, $Z_{\geq k}$ and $Z_{< k+1}$ induce convex subposets of~$P$.
In particular, $Z_k \cup Z_{k+1}$ includes a convex subposet of $P$.

An unfolding of a poset is somewhat a parallel of a layering of a graph.\footnote{A \defin{layering} of a graph $G$ is a family $(Z_i : i\in \NN)$ of pairwise disjoint subsets of $V(G)$ such that for every edge $uv$ of $G$, there exists $i \in \NN$ such that $\{u,v\} \subset Z_i \cup Z_{i+1}$.
Given a connected graph \(G\) and \(u,v \in V(G)\), the \defin{distance} between $u$ and $v$ is the minimum number of edges in a path in \(G\) with endpoints \(u\)
and \(v\).
Now, given a connected graph $G$ and a vertex $v$ of $G$, the \bfs\defin{-layering} of $G$ from $v$ is the sequence
\((Z_i : i \in \NN)\), where \(Z_i\)
is the set of all vertices at distance \(i\) from \(v\) in \(G\) for every nonnegative integer $i$.
} 
Next, we discuss a parallel of \textsc{bfs}-layerings for posets.
Let \(P\) be a connected poset, and let \(z_0\) be a minimal
element of \(P\). 
Let \((Z_i : i\in \NN)\) be the \textsc{bfs}-layering of the comparability 
graph of \(P\) from \(z_0\), that is, \(Z_0 = \{z_0\}\), and for each positive integer $k$,
\[
Z_k =
\begin{cases}
  U_P[Z_{k-1}] \setminus (Z_0 \cup \cdots \cup Z_{k-1})&\textrm{if \(k\) is odd,}\\
  D_P[Z_{k-1}] \setminus (Z_0 \cup \cdots \cup Z_{k-1})&\textrm{if \(k\) is even.}  
\end{cases}
\]
Hence, \((Z_i : i\in \NN)\) is an unfolding of~$P$, and we call it the unfolding of \(P\) \defin{from} \(z_0\).
See an example of an unfolding from a minimal element in \cref{fig:unfolding}.

\begin{figure}[tp]
    \centering 
    \includegraphics[scale=1]{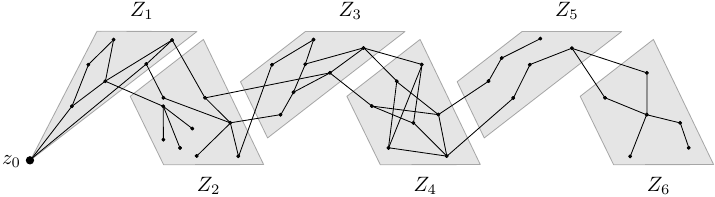} 
    \caption{The unfolding of a poset from $z_0$.
    The drawing is a diagram, that is, for two elements connected by an edge, the lower is less than the higher in the poset.
    In all the following figures, when a segment has no direction and it is not stated otherwise, the comparability is diagram-like (go upwards).
    } 
    \label{fig:unfolding} 
\end{figure} 

Let \((Z_i : i \in \NN)\) be an unfolding of~$P$.
A set \(I \subseteq \Inc(P)\) is
\defin{supported} by
\((Z_i : i \in \NN)\) if there exists a nonnegative integer \(k\) such that 
\[
    \bigcup_{(a,b) \in I}\set{a,b} \subseteq Z_{\geq k}\ \ \ \textrm{and}\ \ \ \begin{cases}\set{b : (a,b) \in I} \subseteq Z_k & \ \ \text{if} \ k \ \text{is odd},\\
        \set{a : (a,b) \in I} \subseteq Z_k & \ \ \text{if} \ k \ \text{is even}.\end{cases}
\]
% \[
%     \bigcup_{(a,b) \in I}\set{a,b} \subseteq Z_{k} \cup Z_{k+1}\ \ \ \textrm{and}\ \ \ \begin{cases}\set{b : (a,b) \in I} \subseteq Z_k & \ \ \text{if} \ k \ \text{is odd},\\
%         \set{a : (a,b) \in I} \subseteq Z_k & \ \ \text{if} \ k \ \text{is even}.\end{cases}
% \]
If $k$ is odd, then we say that \(I\) is \defin{supported
from below} by $(Z_i : i \in \NN)$, if $k$ is even, then we say that \(I\) is \defin{supported
from above} by $(Z_i : i \in \NN)$.
In both cases we say that \(k\) \defin{witnesses} that \(I\) is supported by $(Z_i : i \in \NN)$.
Similarly, we say that a pair $(a,b) \in \Inc(P)$ is supported by an unfolding (from below or above) witnessed by $k$ if $\{(a,b)\}$ is.
%See the set $J'$ in \Cref{fig:unfolding-contracting}, which is supported from below, witnessed by $k = 3$.

It is well-known that
given a graph $G$ and a
layering of $G$, 
if every layer induces a $k$-colorable graph, 
then $G$ is \(2k\)-colorable (one can use two disjoint
palettes of \(k\) colors each, one for even layers, and one for odd layers).
The counterpart of the above for posets is formulated in terms of the unfolding.
Given a connected poset \(P\) and an unfolding of \(P\),
if the union of any two consecutive sets in the unfolding induces a subposet of dimension at most \(d\),
then the dimension of \(P\) is at most \(2d\).
We use this fact to reduce the problem of bounding \(\dim_P(I)\) for any \(I \subseteq \Inc(P)\)
to bounding \(\dim_P(I')\) for some \(I' \subseteq I\) which is supported by the unfolding of \(P\).

\begin{proposition}[{\cite[Proposition 11]{PD1}}]\label{prop:unfolding}
    Let \(P\) be a poset, let $I \subset \Inc(P)$, and let \((Z_i : i \in \NN)\) be an unfolding of \(P\).
    For all positive integers $i$ and $j$, let
      \begin{align*}
          I_{0,i}&=\{(a, b) \in I:\textrm{\(a\in Z_i,b\in Z_k\) with \(i < k\) or (\(i=k\) and \(i\) is even)}\},\\
          I_{1,j}&=\{(a, b) \in I:\textrm{\(a\in Z_k,b\in Z_j\) with \(k > j\) or (\(k=j\) and \(j\) is odd)}\}.
      \end{align*}
    Then,
    \[\{I_{0,i} : i \in \NN\} \cup \{I_{1,j} : j \in \NN\}\]
    is a splitting partition of $I$ in $P$.
\end{proposition}

In the case of planar cover graph, the reduction to a singly constrained instance is simple: we just construct a prefix of an unfolding to a single vertex, see~\cite[Lemma~12]{PD1}.
In the case of posets with planar diagrams, contracting a prefix of an unfolding may violate
planarity of the diagram.

Let \(P\) be a connected poset with a cover graph \(G\) and a minimal element
\(z_0\). Using the unfolding \((Z_i : i \in \NN)\) of \(P\) from 
\(z_0\), we will define a particular type of a spanning tree of \(G\) rooted at \(z_0\).
Such a tree \(T\) will be called a \defin{zig-zag tree} rooted at \(z_0\) in \(P\), and will be defined inductively. If \(z_0\) is the only element of \(P\), then the tree
\(T\) is the trivial tree only consisting of \(z_0\). Otherwise, let \(k\) be the greatest positive integer such that \(Z_k \neq \emptyset\) (and thus \(Z_{\ge k+1} = \emptyset\)).
Let \(x\in Z_k\) such that if \(k\) is odd, then \(x\) is maximal in \(P\), and if \(k\) is even, then \(x\) is minimal in \(P\).
Note that the unfolding of \(P - x\) from \(z_0\) is \((Z_0, \ldots, Z_{k-1}, Z_k \setminus\{x\}, \emptyset, \emptyset, \ldots)\). Let \(T'\) be a zig-zag tree rooted at \(z_0\) in \(P - x\).
Then, we extend \(T'\) to \(T\) by picking as the parent of \(x\) any neighbour \(x'\) of \(x\) in \(G\).
Observe that \(x' \in Z_{k-1} \cup Z_k\), and by our choice of \(x\),
if \(k\) is odd, then \(x' < x\) in \(P\), and
if \(k\) is even, then \(x < x'\) in \(P\).

\begin{obs}\label{obs:parent_function_gives_tree}
    Let $P$ be a connected poset, let $z_0$ be an element of $P$, let $(Z_i : i \in \NN)$ be the unfolding of $P$ from $z_0$, and let $T$ be a zigzag tree rooted at \(z_0\) in \(P\).
    Let $y$ be an element of $P$, and let $k$ be an integer with $y \in Z_k$.
    Let $z_0[T]y = v_0\cdots v_m$. 
    There exist integers $\ell_0,\dots,\ell_k$ such that for each $j \in [k]$, the vertex set of the intersection of $z_0[T]y$ and $Z_j$ is $\set{v_{\ell_{j-1}+1},\dots,v_{\ell_{j}}}$, and 
        \begin{align*}
            v_{\ell_{j-1}}&<\dots<v_{\ell_j} \textrm{ in $P$ \ if $j$ is odd,}\\
            v_{\ell_{j-1}}&>\dots>v_{\ell_j} \textrm{ in $P$ \ if $j$ is even.}
        \end{align*}
    In particular, $v_{\ell_{j-1}}$ is a minimal (resp.\ maximal) element of the vertex set of the intersection of $z_0[T]y$ and $Z_{j-1}\cup Z_{j}$ if $j$ is odd (resp.\ even).
\end{obs}

\begin{figure}[tp]
    \centering 
    \includegraphics[scale=1]{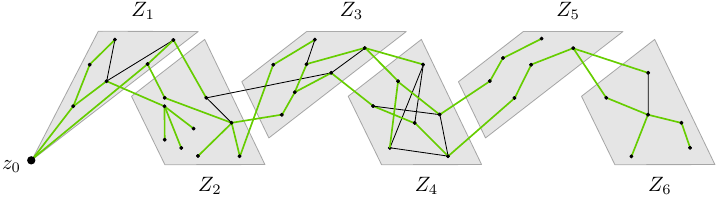} 
    \caption{
    Consider the same unfolding as in \cref{fig:unfolding}.
    For each element, we arbitrarily choose the parent according to the rules. 
    We mark the edges between elements and their parents in green.
    As a result, we obtained a spanning tree of the cover graph of a poset.
    } 
    \label{fig:zigzag-tree} 
\end{figure}

\begin{proposition}\label{lem:planar-subposet}
    Let \(P\) be an $n$-element poset with a fixed efficient planar diagram. 
    Let \(Q\) be a convex subposet of \(P\) and let $H$ be the cover graph of $Q$.
    Let \(x\) be an element of \(P\) which lies in the interior of the exterior face of the drawing of \(H\) inherited from the diagram of \(P\).
    Then the subposet of \(P\) induced by all elements of \(Q\) and \(x\) has a planar diagram with \(x\) on the exterior face. 
    Moreover, the combinatorial embedding of such a diagram can be computed in polynomial time.
\end{proposition}
\begin{proof}
        Since $Q$ is a convex subposet of $P$, $x$ can not be simultaneously below some elements of $Q$ and above some other elements of $Q$.
        We assume that $x$ is not above elements of $Q$. 
        %(In fact, we will use the claim only in this case.)
        The proof for the other case is symmetric.
        
        Let $M$ be the set of minimal elements among all $q$ in $Q$ with $x < q$ in $P$.
        Note that $M$ induces an antichain in $P$.
        Moreover, all the elements in $M$ lie on the boundary of the exterior face of $Q$.
        Consider a witnessing path between $x$ and an element of $M$ in $P$.
        Each such witnessing path is contained in the exterior face of $Q$ and is disjoint from $Q$ except its last element in $M$. 
        By a simple iterative process, we construct a family $T$ of such paths that contains one path for each $q\in M$ and that the union of all paths in $T$ gives a tree as a subgraph of the cover graph of $P$.    
        In order to obtain a planar diagram of the subposet of $P$ induced by $x$ and all the elements of $Q$, we start with the diagram of $Q$ inherited from the diagram of $P$.
        We put $x$ in the same point as in the diagram of $P$.
        Finally, we draw diagram edges between $x$ and $M$ in a non-crossing way. We can achieve this by picking a sufficiently small \(\eps > 0\) and drawing the edge between \(x\) and a vertex \(y \in M\) in the strip consisting of all points at most \(\eps\) left or right of the path between \(x\) and \(y\) in \(T\). 
        See \Cref{fig:drawing-along-tree}.
        Note that since we add only edges incident to $x$, the element $x$ remains on the exterior face. This completes the proof of the existence of a planar diagram of the subposet of $P$ induced by elements of $Q$ and $x$ with $x$ on the exterior face. 
        
        In order to compute the combinatorial embedding of such a diagram, we take the combinatorial embedding of the efficient planar diagram of $P$ (given on the input) and compute the inherited combinatorial embedding of the diagram of $Q$. Moreover tracing the paths in the family $T$ in the diagram of $P$, we obtain a total left-to-right order on them which gives us the order of edges incident to $x$ in the new diagram. 
            \begin{figure}[tp]
                \centering 
                \includegraphics[scale=1]{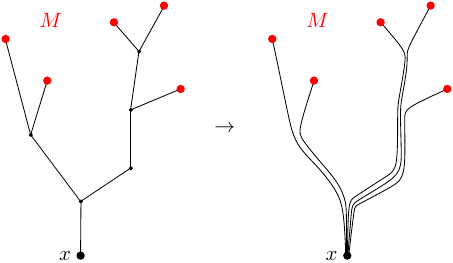} 
                \caption{The tree of edges in the cover graph, where $x$ is the root, and elements in $M$ are the leaves can be transformed into vertically monotone edges connecting $x$ with each element of $M$ preserving planarity of the drawing.} 
                \label{fig:drawing-along-tree} 
            \end{figure} 
    \end{proof}

% \begin{lemma}\label{PlanarDiagramReduction}
%     Let $P$ be a poset with a planar diagram. 
%     There exists a poset $Q$, a set $J\subseteq \Inc(Q)$, and an element $x$ in $Q$ such that
%     \begin{enumerate}
%         \item $Q$ has a planar diagram with $x$ on the boundary of the exterior face, \label{item:planar_diagram_reduction:drawing}
%         \item $Q$ is a subposet of $P$ or a subposet of the dual of $P$, \label{item:planar_diagram_reduction:subposet}
%         \item $J$ is singly constrained by $x$ in $Q$, and        \label{item:planar_diagram_reduction:constrained}
%         \item $\dim(P) \leq 24\dim_Q(J)$.
%         \label{item:planar_diagram_reduction:dimension_bound}
%         %\item \michal{We also need:} \(\kelly(Q) \le \kelly(P)\).
%         %\label{item:planar_diagram_reduction:kelly_bound}
%     \end{enumerate}
% \end{lemma}
\begin{proof}[Proof of \Cref{lem:singly-constrained-reduction}]
    % Assume that we have already designed the required algorithm for connected posets.
    % By \Cref{prop:dim-components}, we may extend this algorithm to disconnected posets.
    Let $(n,P,I,\calE)$ be an input to the algorithm.
    Using the algorithm from \Cref{thm:encoding-diagrams}, we fix an efficient planar diagram of $P$.
    Several times during the proof, we apply the given $f$-good algorithm $\algo{A}$ solving \PDD problem.
    Recall that it outputs either a covering by reversible sets of a given subset of incomparable pairs or a Kelly subposet of order $n$ based on the given subset.
    Each time, we apply $\algo{A}$ to the same $n$, a subposet of $P$, and a subset of $I$.
    Thus, if the output of $\algo{A}$ is a Kelly subposet, we may also output it.
    Therefore, for simplicity, we assume otherwise, i.e.\ that the output from $\algo{A}$ is always a covering.

    We describe a branching procedure that the algorithm executes, producing many new instances of the problem in a richer setting.
    The number of these instances will be polynomial in $|P|$.
    The substantial part of the proof of the lemma deals with solving the problem for such instances.
    Finally, we use \Cref{prop:dim-components,prop:unfolding} to combine the results returned for the instances.
    We use the convention that if we do not quantify objects in \q{max}, then we maximize over all possible valid choices of the algorithm.
    
    %For each \(s \in [3]\), the algorithm chooses a pair \((P_s, I_s)\) such that \(P_s\) is a connected convex subposet of \(P_{s-1}\) and \(I_s \subseteq I_{s-1} \cap \Inc(P_s)\).
    %We also ensure that for all valid choices (according to the rules below) of the algorithm \piotr{flagged!} of $(P_s,I_s)$, we have either $\dim_P(I) \leq 16$ or \(\dim(P) \le 2^s \cdot \max_{(P_s,I_s)} \dim_{P_s}(I_s)\).

    First, the algorithm chooses $P_0$ to be a component of $P$ and sets $I_0 = \Inc(P_0) \cap I$.
    Let \(s \in [3]\), and suppose that the algorithm has already chosen \(P_{s-1}\) and \(I_{s-1}\).
    Let \(y_s\) denote an element of \(P_{s-1}\) which is lowest in the diagram of \(P\).
    Let \((Z_i^s : i\in \NN)\) be the unfolding of \(P_{s-1}\) from \(y_s\).
    The algorithm chooses a nonnegative integer $k_s$ with $Z_{k_s}^s \neq \emptyset$.
    Let $I_s'$ be the set of all pairs in $I_{s-1}$ supported by $(Z_i^s : i \in \NN)$, which is witnessed by $k_s$.
    Let $P_{s}'$ be the subposet of $P_{s-1}$ induced by \(Z_{\ge k_s}^s\).
    The algorithm chooses a component $P_s$ of $P_{s}'$, and let $I_s = I_s' \cap \Inc(P_s)$.
    Note that
    \begin{align*}
        \dim_{P_{s-1}}(I_{s-1}) &\leq 2\max \{\dim_{P_s'}(I_s') : (P_s',I_s') \} && \text{by \Cref{prop:unfolding}}\\
        &\leq 2 \max (\{ \dim_{P_s}(I_s) : (P_s,I_s) \} \cup \{2\}) && \text{by \Cref{prop:dim-components},}
    \end{align*}
    where the first max goes over all valid choices of $k_s$, and the second max goes over all valid choices of $(k_s,P_s)$.
    A more precise application of \Cref{prop:unfolding} yields that, given a covering of $I_s'$ by reversible sets in $P_s'$ of size at most $d$ for each possible $(I_s',P_s')$ given by a choice of $k_s$ by the algorithm, we can efficiently compute a covering of $I_{s-1}$ by reversible sets in $P_{s-1}$ with at most $2d$ sets.
    Subsequently, a more precise application of \Cref{prop:dim-components}, yields that, given a covering by reversible sets of $I_s$ in $P_s$ of size at most $d$ for each possible $(I_s,P_s)$ given by a choice of $P_s$ by the algorithm, we can efficiently compute a covering of $I_{s}'$ by reversible sets in $P_{s}'$ with at most $\max\{d,2\}$ sets.
    
    Note that in each step, the algorithm has polynomial choices in terms of $|P|$.
    Combining and simplifying the inequalities for each $s \in [3]$, we obtain,
    \[\dim_P(I) \leq 2\cdot 2 \cdot 2 \cdot \max  (\{ \dim_{P_3}(I_3) : (P_3,I_3) \} \cup \{2\}).\]
    Again, we will need a more precise statement for the sake of the algorithm.
    Simply combining the steps for each $s \in [3]$ and applying~\Cref{prop:dim-components} to go from $(P,I)$ to $(P_0,I_0)$, we obtain the following property.
    Given a covering of $I_3$ by reversible sets in $P_3$ of size at most $d$ for each possible $(P_3,I_3)$ given by each branch $(P_0,k_1,P_1,k_2,P_2,k_3,P_3)$ of choices of the algorithm, we can efficiently compute a covering of $I$ by reversible sets in $P$ with at most $\max\{8d,16\}$ sets. 
    
    With a simple trick of grouping some of the branches, we will replace \q{$8$} with \q{$6$}.
    Consider a possible branch of the algorithm: $(P_0,k_1,P_1,k_2,P_2,k_3,P_3)$ and respective sets $I_1$, $I_2$, and $I_3$ defined there.
    We record a 3-element bitstring $(\gamma_1,\gamma_2,\gamma_3)$ where for every $s \in [3]$, $\gamma_s = 0$ if $I_s$ is supported from below in $(Z_i^s : i \in \NN)$ ($k_s$ is even) and $\gamma_s = 1$ if $I_s$ is supported from above in $(Z_i^s : i \in \NN)$ ($k_s$ is odd).
    We assign $(\gamma_1)$ to $I_1$, $(\gamma_1,\gamma_2)$ to $I_2$, and $(\gamma_1,\gamma_2,\gamma_3)$ to $I_3$.
    
    For a bitstring $\gamma$ of length at most $3$, let $\calI_\gamma$ be the family of all $I_{|\gamma|}$ produced by the algorithm, which is assigned $\gamma$.
    Let $\Gamma = \{(0,0),(1,1),(0,1,0),(0,1,1),(1,0,0),(1,0,1)\}$.
    This set has two important properties.
    First, for every $3$-element bitstring, there exists $\gamma \in \Gamma$ such that $\gamma$ is a prefix of this bitstring.
    Second, each $\gamma \in \Gamma$ contains two equal entries.
    The first property implies that the following statement is true.
    Given a covering of $J \in \calI_\gamma$ for every $\gamma \in \Gamma$ by reversible sets in $P$ of size at most $d$ of choices of the algorithm, we can efficiently compute a covering of $I$ by reversible sets in $P$ with at most $\max\{|\Gamma|\cdot d,16\} = \max\{6d,16\}$ sets.

    In the next step, for each $J \in \calI_\gamma$ for every $\gamma \in \Gamma$, and for every corresponding poset $Q$ (if $J = I_2$ in some branch of the algorithm, then $Q = P_2$ and if $J = I_3$, then $Q = P_3$), we will construct a covering of $J$ by reversible sets in $Q$ of size at most $4f(n)$.
    %Note that by the monotonicity of $f$, $f(Q,J) \leq f(n)$.
    As announced, this will give a covering of $I$ by reversible sets in $P$ with at most $\max\{24f(n),16\}$ sets.
    Combinatorially, we obtain
        \[\dim_P(I) \leq 24 f(n).\]
    
    % After completing the branching process, the algorithm inspects each choice of $(P_3,I_3)$ and tests whether $\dim_{P_3}(I_3)=1$ (or in other words, whether $I_3$ is reversible). 
    % By \Cref{prop:detect-reversibility} this can be done efficiently. % and if it is true the covering of $I_3$ by reversible sets in $P$ is just $\{I_3\}$.
    % There are two cases.
    % If for all valid choices of $(P_3,I_3)$, we have that $I_3$ reversible, then using \Cref{prop:dim-components,prop:unfolding} the algorithm returns a covering of $I$ by at most $16$ reversible sets.
    %Since $f(n)\geq1$, we have $16 \leq 24f(n) \leq 24$.
    
    Fix a branch of the algorithm $(P_0,k_1,P_1,k_2,P_2,k_3,P_3)$. 
    Recall that the algorithm has defined in this branch the following objects: $y_S$, $(Z_i^s : i \in \NN)$, $I_s'$, $P_s'$, and $I_s$ for each $s \in [3]$. 
    We will first state some general observations.
    %Note that in some of the branches we never use $P_3$ or $I_3$.
    For every $s \in [3]$, let \(G_s\) be the cover graph of \(P_s\), let \(T_s\) be a zig-zag tree rooted at \(y_s\) in \(P_{s-1}\), and let $Z_s = Z_{k_s}^s$.
By construction, we have
  \begin{enumerateNumeo}
      \setcounter{enumi}{-1}
      \item $y_s$ lies in the exterior face of the diagram of $P_{s-1}$ inherited from the diagram of $P$. \label{item:eo0}
      %\item for every $(a,b) \in I_s$, we have $a,b \in Z_{\geq k_s}^s$. \label{item:eo1}
  \end{enumerateNumeo}  
  Moreover by \cref{obs:parent_function_gives_tree}, if $k_s$ is odd
  \begin{enumerateNumo}
      \setcounter{enumi}{0}
      \item for every $u \in Z_s$, the graph $y_s[T_s]u \cap G_s$ is a witnessing path in $P$ contained in $Z_s$ with the greatest element being $u$;\label{item:k_s_odd_first}
      \item the subpath $W$ of $y_s[T_s]u$ restricted to 
      $D_{P_{s-1}}[Z_s]$ 
      %$Z_{k_s-1}^s \cup Z_{k_s}^s$ 
      contains an element $w$ such that for every $v$ in $W$, we have $w \leq v$ in $P$;
      \item for every $(a,b) \in I_s$, we have $b \in Z_s$; and\label{item:k_s_odd_last}
  \end{enumerateNumo}
  if $k_s$ is even
  \begin{enumerateNume}
      \setcounter{enumi}{0}
      \item for every $u \in Z_s$, the graph $y_s[T_s]u \cap G_s$ is a witnessing path in $P$ contained in $Z_s$ with the least element being $u$;\label{item:k_s_even_first}
      \item the subpath $W$ of $y_s[T_s]u$ restricted to 
      $U_{P_{s-1}}[Z_s]$
      %$Z_{k_s-1}^s \cup Z_{k_s}^s$ 
      contains an element $w$ such that for every $v$ in $W$, we have $w \geq v$ in $P$;
      \item for every $(a,b) \in I_s$, we have $a \in Z_s$.\label{item:k_s_even_last}
  \end{enumerateNume}

  For every \(s \in [3]\), the set \(I_s\)
  is supported by the unfolding \((Z_i^s : i \in \NN)\)
  either from below ($k_s$ is odd and \ref{item:k_s_odd_first}-\ref{item:k_s_odd_last} hold) or from above ($k_s$ is even and \ref{item:k_s_even_first}-\ref{item:k_s_even_last} hold).
  By the pigeonhole principle, there
  exist \(\alpha,\beta\in[3]\) with \(\alpha < \beta\) such that
  the sets \(I_\alpha\) and \(I_\beta\) are supported from the same side by
  the corresponding unfoldings.
  Fix such \(\alpha\) and \(\beta\) with $\beta$ minimal.
  In other words, the set $I_\beta$ is a member of $\calI_\gamma$ for some $\gamma \in \Gamma$. 
  For each \(s \in \{0, 1,2, 3\}\), let \((P_s^*,I_s^*)=(P_s,I_s)\)
  if the sets \(I_\alpha\) and \(I_\beta\) are supported from below, and
  let \((P_s^*,I_s^*)=(P_s^{-1},I_s^{-1})\)
  if the sets \(I_\alpha\) and \(I_\beta\) are supported from above.
  Note that $G_s$ is the cover graph of $P_s^*$.
  Let also $P^* = P$ if the sets are supported from below and let $P^* = P^{-1}$ otherwise.
  We fix the diagram of $P^*$: if $P^*=P$ then we take the fixed diagram of $P$, 
  and if $P^*=P^{-1}$ then we take the fixed diagram of $P$ and rotate it by $180^\circ$ in the plane.
  For each $s \in [3]$, we consider $P_s^*$ with a fixed diagram inherited from the diagram of $P^*$.
  % \michal{I have been consistently using ``inherited drawing/diagram'', we should stick to one version (I am fine with induced)}
  % \jedrzej{Ok.}
  By \ref{item:eo0}, by \ref{item:k_s_odd_first}-\ref{item:k_s_odd_last} in the former case and by \ref{item:k_s_even_first}-\ref{item:k_s_even_last} in the latter case, for each $s \in \{\alpha,\beta\}$, we obtain
  \begin{enumerateNumU}
      \setcounter{enumi}{-1}
      \item $y_s$ lies in the exterior face of the diagram of $P^*_{s-1}$;  \label{items:extracted_from-unfoldings:y_s_on_exterior_face_u0}
      %\item for every $(a,b) \in I_s^*$, we have $a,b \in Z_{\geq k_s}^s$;\label{items:extracted_from-unfoldings:a_b_in_suffix}
      \item for every $u \in Z_s$, the graph $y_s[T_s]u \cap G_s$ is a witnessing path in $P^*$ contained in $Z_s$ with the greatest element being $u$;\label{items:extracted_from-unfoldings:vertical_paths}
      \item the subpath $W$ of $y_s[T_s]u$ restricted to 
      $D_{P^*}[Z_s]$
      %$Z_{k_s-1}^s \cup Z_{k_s}^s$ 
      contains an element $w$ such that for every $v$ in $W$, we have $w \leq v$ in $P^*$;\label{items:extracted_from-unfoldings:V_paths}
      \item for every $(a,b) \in I_s^*$, we have $b \in Z_s$.\label{items:extracted_from-unfoldings:bs_in_right_cell}
  \end{enumerateNumU}
  Let $J = I_\beta^*$, $Q = P_\beta^*$, 
  and $G_Q = G_\beta$.
  See an example in \cref{fig:unfoldings-alpha-beta}.
  Moreover, combining \ref{items:extracted_from-unfoldings:y_s_on_exterior_face_u0} and 
 \ref{items:extracted_from-unfoldings:bs_in_right_cell} for both $s=\alpha$ and $s=\beta$, we have
  \begin{enumerateNumU}
      \setcounter{enumi}{3}
      \item $y_\alpha$ and $y_\beta$ lie in the exterior face of the diagram of $P^*_{\beta-1}$; \label{items:extracted_from-unfoldings:y_s_on_exterior_face}
      \item for every $(a,b) \in J$, we have $b \in Z_\alpha \cap Z_\beta$.\label{items:extracted_from-unfoldings:bs_in_right_cell_combined}
  \end{enumerateNumU}

    The algorithm will construct four subsets $J_1,J_2,J_3,J_4$ of $J$ covering $J$, four subposets $Q_1,Q_2,Q_3,Q_4$ of $P^*$, four combinatorial embeddings $\calE_1,\calE_2,\calE_3,\calE_4$, and four elements $x_1,x_2,x_3,x_4$ of $P^*$.
    Moreover, for each $\eps \in \{3,4\}$, the algorithm will construct a splitting partition $\{J_\eps^i : i \in [m]\}$ of $J_\eps$ and elements $\{x_\eps^i : i \in [m]\}$ of $Q_\eps$.
    The required properties are that for every tuple $(n,Q',J',x',\calE')$ of the form either $(n,Q_\eps,J_\eps,x_\eps,\calE_\eps)$ for $\eps \in \{1,2\}$ or $(n,Q_\eps,J_\eps^i,x_\eps^i,\calE_\eps)$ for $\eps \in \{3,4\}$ and $i\in[m]$ we have,
    \begin{enumerate}
        \item  $J' \subset \Inc(Q')$ and $x'$ is in $Q'$,
        \item $\calE'$ is a combinatorial embedding of $Q'$,
        \item there exists a planar diagram of $Q'$ with $x'$ on the boundary of the exterior face witnessed by $\calE'$,
        \item $J'$ is singly constrained by $x'$ in $Q'$.
    \end{enumerate}
    Having all these objects, the algorithm calls a polynomial-time $f$-good algorithm solving \PDD given by the assumption for each instance of the form $(n,Q',J',x',\calE')$ as above.
    The results can be efficiently combined to a covering of $J$ by $4f(n)$ reversible sets in $Q$.
    Note that some of the sets have to be \q{reversed} in the final covering of $I$ by reversible sets in $P$.
    This will complete the proof and the construction of the algorithm.

    \begin{figure}[tp]
        \centering 
        \includegraphics[scale=1]{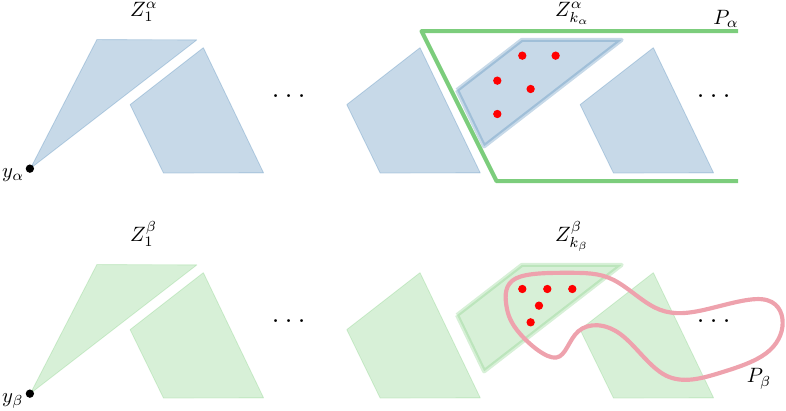} 
        \caption{
            Red elements depict elements $b$ in pairs $(a,b) \in J$.
            The top unfolding is an unfolding of $P_{\alpha-1}$ and the bottom unfolding is an unfolding of $P_{\beta-1}$.
            We show the case, where the sets $I_\alpha$ and $I_\beta$ are supported from below (in particular, $P^* = P$). 
            The poset $Q$ is a convex subposet of $P$.
        } 
        \label{fig:unfoldings-alpha-beta} 
    \end{figure} 
    
  For each $s \in \set{\alpha,\beta}$, define a subtree \(T_s'\) of \(T_s\) rooted at $y_s$ as
    \[T_s' = \bigcup_{(a, b) \in J} y_s [T_s] b.\]
  % Moreover, by \ref{T1},
  % \[V(T_s') \cap V(G_s) = Z_{k_s}^s.\]
  %\michal{the current form of \ref{T1} doe not imply this. gotta fix.}
  Let \(H = T_{\beta}' \cup G_Q\).
  Since $Q$ is a convex subposet of $P_{\beta-1}^*$, by \ref{items:extracted_from-unfoldings:y_s_on_exterior_face}, the element $y_\beta$ lies in the exterior face of the diagram of $Q$.
  Recall that $G_Q$ is the cover graph of $Q$ and the leaves of $T_\beta'$ are the elements of $Q$ by \ref{items:extracted_from-unfoldings:bs_in_right_cell_combined}.
  Therefore, there is a facial walk \(\calQ\) along the exterior face of $G_Q$ in the fixed diagram of $Q$ of the form
    \[\calQ = v_0 e_0 \cdots e_{t-1} v_t\]
  so that there is \(\zeta \in [t]\) with the facial walk along the exterior face of \(H\)
  being the concatenation of \(y_\beta [T_\beta'] v_\zeta\), \(v_\zeta \calQ v_t\) and
  \(v_t [T_\beta'] y_\beta\). See \Cref{fig:unfoldingH}
  Fix such $\calQ$ and $\zeta$.
  Recall that by definition of a facial walk $v_0 = v_t$. 
  %\fig.
  \begin{figure}
      \centering
      \includegraphics{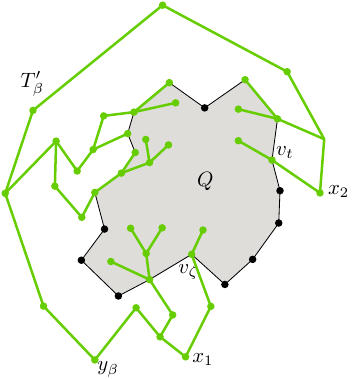}
      \caption{The facial walk along the exterior face of graph \(H = T_{\beta}' \cup G_Q\) is the concatenation of walks \(y_\beta[T_{\beta}']v_\zeta\), \(v_\zeta \mathcal{Q} v_t\) and \(v_t[T_\beta']y_\beta\). Only the elements in \(T_\beta'\) and \(\mathcal{Q}\) are drawn on the figure.}
      \label{fig:unfoldingH}
  \end{figure}

  Let $s \in \set{\alpha,\beta}$.
  For each element $q$ in $Q$, let $v^s(q)$ be the first vertex of $y_s[T_s]q$ in $Q$.
  For emphasis, we state the following straightforward claim.
  \begin{claim}\label{claim:planar_diagram_reduction:green_entries}
      For every element $q$ in $Q$ such that $q$ is a vertex of $T_\beta'$, we have $v^\beta(q) \in \{v_0,\dots,v_\zeta\}$.
  \end{claim}

    By \ref{items:extracted_from-unfoldings:V_paths}, 
    we fix an element $x_1$ of $y_\beta[T_\beta']v_\zeta$ such that for all $v$ in the restriction of $y_\beta[T_\beta']v_\zeta$ to $D_{P^*}[Z_\beta]$, 
    we have $x_1 \leq v$ in $P^*$.
    Similarly, we fix an element $x_2$ of $y_\beta[T_\beta']v_t$ such that for all $v$ in the restriction of $y_\beta[T_\beta']v_t$ to $D_{P^*}[Z_\beta]$, 
    we have $x_2 \leq v$ in $P^*$.
    %See \fig.

    Let $\eps \in \{1,2\}$, let $Q_\eps$ be the subposet of $P^*$ induced by all the elements of $Q$ and $x_\eps$, and let
        \[
            J_\eps = \{(a, b) \in I : \textrm{\(x_{\eps} \le b\) in \(P^*\)}\}.
        \]
    Let $\calE_\eps$ be a combinatorial embedding obtained from \Cref{lem:planar-subposet} witnessing a planar diagram of $Q_\eps$ with $x_\eps$ in the exterior face.
    We have \(J_\eps \subseteq \Inc(Q_\eps)\), $x_\eps$ is in $Q_\eps$, and $J_\eps$ is singly constrained by $x_\eps$ in $Q_\eps$ (by definition), as desired.

  Let $B'$ be the set of all the elements $q$ in $Q$ that are vertices of $V(T_\alpha') \cap V(T_\beta')$ but $x_1 \not\leq q$~in~$P^*$ and $x_2 \not\leq q$~in~$P^*$.
    %\[B' = V(G_\beta) \cap V(T_\alpha') \cap V(T_\beta') \setminus U_{P^*}[\{x_1, x_2\}].\]
  Note that $B' \subset Z_\alpha \cap Z_\beta$, and observe that if  \((a, b) \in J \setminus (J_1 \cup J_2)\), then \(b \in B'\).
  %See \fig.
  
  For each \(q \in B'\), let $m(q)$ be a minimal element of the subposet of $P^*$ induced by  
    \[ \{u \in B' : u \leq q \text{ in } Q\}. \]
  %choose an element $m(b) \in B'$ that is a minimal element in the subposet of $P^*$ induced by $B'$ such that \(m(b) \le b\) in \(P^*\).
  %For each $b \in B'$ and \(s \in \{\alpha, \beta\}\), let \(v^s(b)\) denote the first element on \(y_s [T_s] m(b)\) in \(Q\), and let \(M_s(b) = v^s(b) [T_s] m(b)\).
  For each $q \in B'$ and $s \in \{\alpha,\beta\}$, let \(M_s(q) = v^s(m(q)) [T_s'] m(q)\).
  Since $Q$ is a convex subposet of $P_\alpha^*$, and by \ref{items:extracted_from-unfoldings:vertical_paths}, for each \(q \in B'\) and \(s \in \{\alpha, \beta\}\), the path \(M_s(q)\) is a witnessing path in \(Q\) contained in $Z_s$ with the greatest element $m(q)$.
  %See \fig.
  For emphasis, we state the following easy observation.
  
  \begin{claim}\label{IntersectionMsMs}
      Let \(q,q'\in B'\) and \(s\in\{\alpha,\beta\}\). 
      If $M_s(q)$ and $M_s(q')$ intersect, and $u$ is their greatest common element, then $v^s(m(q))[M_s(q)]u = v^s(m(q'))[M_s(q')]u$.
      In particular, $v^s(m(q)) = v^s(m(q'))$.
  \end{claim}

  % \begin{claim}
  %     For each \(b \in B'\) and \(s \in \{\alpha, \beta\}\), the path \(M_s(b)\) is a witnessing path from \(v^s(b)\) to \(m(b)\) in \(Q\).
  % \end{claim}
  % \begin{proof}
  % If \(s = \beta\), the claim follows
  % directly from \ref{T1}.
  % Suppose then that \(s = \alpha\). By
  % \ref{T1},
  % \(y_\alpha [T_\alpha'] m(b) \cap G_\alpha\) is a witnessing path in \(P_\alpha\) whose greatest element is \(m(b)\).
  % Since \(P_\beta^*\) is a convex subposet of
  % \(P_\alpha^*n\) which contains \(m(b)\),
  % the intersection \(y_\alpha [T_\alpha'] m(b) \cap G_\beta\)
  % must be a witnessing path in \(P_\beta\)
  % whose greatest element is \(m(b)\).
  % But \(y_\alpha(b) [T_\alpha'] m(b) \cap G_\beta
  % = v^\alpha(b) [T_\alpha'] m(b)\), so the claim
  % holds.    
  % \end{proof}

  \begin{claim}\label{IntersectionMalphaMbeta}
      Let \(q,q'\in B'\).
      If $M_\alpha(q)$ and $M_\beta(q')$ intersect, then $m(q) = m(q')$ and no other element lies in the intersection of the two paths.
      %\(V(M_\alpha(b)) \cap V(M_\beta(b')) \neq \emptyset\), then \(V(M_\alpha(b)) \cap V(M_\beta(b')) = \{m(b)\} = \{m(b')\}\).
  \end{claim}
  \begin{proofclaim}
    Assume that $M_\alpha(q)$ and $M_\beta(q')$ intersect in $u$.
    It follows that $u$ is a vertex in $V(T_\alpha') \cap V(T_\beta')$.
    Since $u \leq m(q) \leq q$ in $P^*$ and $q \in B'$, we have $x_1 \not\leq q$ in $P^*$ and $x_2 \not\leq q$ in $P^*$, therefore, $x_1 \not\leq u$ in $P^*$ and $x_2 \not\leq u$ in $P^*$, and hence $u \in B'$.
    Since $u \leq m(q) \leq q$ in $Q$ and $u \leq m(q') \leq q'$ in $Q$, we obtain $u = m(q) = m(q')$ by definition of $m(q)$ and $m(q')$.
  \end{proofclaim}

  \cref{IntersectionMalphaMbeta} implies that for every \(q \in B'\) the paths \(M_\alpha(q)\) and \(M_{\beta}(q)\) intersect only in \(m(q)\).
  Therefore, \(M_\alpha(q)\cup M_{\beta}(q)\)
  is a path between \(v^\alpha(m(q))\) and \(v^\beta(m(q))\) in $Q$, and we denote it by
  \(M(q)\).
  %Observe that for every \(b \in B'\), we have \(m(m(b)) = m(b)\), so \(M(m(b)) = M(b)\).

  The path $M(q)$ has two endpoints: $v^{\alpha}(m(q))$ and $v^{\beta}(m(q))$. 
  \cref{claim:planar_diagram_reduction:green_entries} locates $v^{\beta}(m(q))$ 
  within the section $\{v_0, \ldots, v_\zeta\}$ of the walk $\mathcal{Q}$ around the exterior face of $Q$.
  %one of the endpoints of $M(q)$ on the walk $\calQ$ for every $q \in B'$.
  In the next claim, we locate the other endpoint in the complementary section of $\mathcal{Q}$.
  %See \fig.

  % Recall that \(H = T_\beta' \cup G_\beta\),
  % \(\calQ = v_0 e_0 \cdots e_{t-1} v_t\)
  % is the facial walk along the exterior face
  % of \(G_\beta\), and the facial walk along the exterior face of \(H\)
  % is the concatenation of \(y_\beta[T_\beta']v_\zeta\), \(v_\zeta [\calQ] v_t\) and
  % \(v_t[T_\beta'] y_\beta\).
  % every vertex which lies on the boundary of the exterior face of \(H\)
  % but not on the boundary of the exterior face of \(G_\beta\)
  % lies on one of the paths \(y_\beta[T_\beta'] v_\zeta\) or \(y_\beta [T_\beta'] v_t\)
  % the only edges of \(\calQ\) on the exterior face of \(H\)
  % are \(e_0, \ldots, e_{\zeta-1}\).
  % Therefore,
  % \[\{v^\beta(b): b \in B'\} \subseteq \{v_0, \ldots, v_\zeta\}.\]
  
  \begin{claim}
  \label{claim:planar_diagram_reduction:green_and_blue_entries}
    For every \(q \in B'\), we have \(v^\beta(m(q)) \in \{v_0, \ldots, v_\zeta\}\) and
    \(v^\alpha(m(q)) \in \{v_{\zeta+1}, \ldots, v_{t-1}\}\).
  \end{claim}
  \begin{proofclaim}
      The first part follows directly from \cref{claim:planar_diagram_reduction:green_entries}, thus, we proceed with the proof of the second part.
      First, note that since $H$ is a subgraph of the cover graph of $P_{\beta-1}^*$, by \ref{items:extracted_from-unfoldings:y_s_on_exterior_face}, the element \(y_\alpha\) lies in the exterior face of $H$.
      Suppose to the contrary that for some \(q \in B'\), we have
      \(v^{\alpha}(m(q)) \not \in \{v_{\zeta+1}, \ldots, v_{t-1}\}\).
      It follows that the path \(y_\alpha [T_\alpha] v^\alpha(m(q))\) intersects the boundary of the exterior face of $H$ in an element of \(y_\beta [T_\beta'] v_\zeta\) or
      \(y_\beta [T_\beta'] v_t\), call this element~$v$.
      Note that $v$ is an element of $P_{\beta-1}^*$, hence, it is an element of $P_\alpha^*$, and so, $v$ is a vertex of $y_\alpha[T_\alpha] v^\alpha(m(q)) \cap G_\alpha$.
      Therefore, by \ref{items:extracted_from-unfoldings:vertical_paths}, we have $v \leq v^\alpha(m(q))$ in $P^*$, and so, $v \leq q$ in $P^*$.
      On the other hand, $q \in B'$, hence, $q \in Z_\beta$, thus, $v \in D_{P^*}[Z_\beta]$.
      It follows that by \ref{items:extracted_from-unfoldings:V_paths}, $x_1 \leq v$ in $P^*$ or $x_2 \leq v$ in $P^*$. 
      Thus $x_1 \leq q$ in $P^*$ or $x_2 \leq q$ in $P^*$, which contradicts $q \in B'$ and ends the proof.
  \end{proofclaim}

    Let $M$ be a path in $Q$ with endpoints $v,v'$ such that $v \in \{v_0,\dots,v_\zeta\}$ and $v' \in \{v_{\zeta+1},\dots,v_{t-1}\}$.
    We define a region associated with $M$.
    Let $i$ be the greatest integer in $\{0, \ldots, \zeta\}$ with \(v_{i} = v\), and let $j$ be the least integer in $\{\zeta+1, \ldots, t-1\}$ with \(v_{j} = v'\).
    Let \(\calQ' = v_{i} e_{i} \cdots e_{j-1} v_{j}\) be a subwalk of \(\calQ\), and let $\calQ''$ be the closed walk obtained by concatenating $\calQ'$ with the walk along $M$ starting in $v'$ and ending in~$v$.
    Finally, let $R(M)$ be the region bounded by the closed walk $\calQ''$, and let $I(M) = \{i,\dots,j\}$.
    For convenience, let us divide all the elements of $Q$ into three pairwise disjoint parts.
    The first part is all the elements of $M$.
    Among the remaining elements, let $R^{+}(M)$ be all the elements in $R(M)$, and let $R^{-}(M)$ be all the elements not in $R(M)$.
    Note that it can be checked efficiently if a given element is in a given region in the plane of the form as above.
    %See \fig.
    
    We state the following straightforward topological observations for emphasis.
    Let $M$ be a path in $Q$ with endpoint in $\{v_0,\dots,v_\zeta\}$ and the other in $\{v_{\zeta+1},\dots,v_{t-1}\}$.
    Let $W$ be a path in $Q$, let $w,w'$ be the endpoints of $W$, and let $v$ be an element of $Q$ that lies in the exterior face of $Q$.
    \begin{enumerateNumt}
        %\item If $w \in R^{+}(M)$ and $w' \in R^{-}(M)$, then $W$ intersects $M$.
        %    \label{items:topological_observations:intersect_M}
        \item If $W$ is disjoint from $M$, then either $w,w' \in R^{+}(M)$ or $w,w' \in R^{-}(M)$.
            \label{items:topological_observations:dont_intersect_M}
        \item If $v \in R^{+}(M)$, then $v = v_\eta$ for some $\eta \in I(M)$.
        %there is $\eta \in \{i_q,\dots,j_q\}$ such that $v = v_\eta$.
            \label{items:topological_observations:disjoint_with_R-}
        \item If $v \in R^{-}(M)$, then $v \neq v_\eta$ for every $\eta \in I(M)$.
        %for every integer $\eta$ with $v = v_\eta$, we have $\eta \in \{0,\dots,i_q\} \cup \{j_q,t-1\}$.
            \label{items:topological_observations:disjoint_with_R+}
        %\item If $\{i_{q'},\dots,j_{q'}\} \subset \{i_{q},\dots,j_{q}\}$ and $M(q') \subset R(q)$, then $R(q') \subset R(q)$.
        %    \label{items:topological_observations:force_inclusion_on_Rs}
    \end{enumerateNumt}

    By \Cref{claim:planar_diagram_reduction:green_and_blue_entries}, for each $q \in B'$, the region $R(M(q))$ and the interval $I(M(q))$ are well-defined.
    Since we will use the above notions only in such a situation, let $R(q) = R(M(q))$ and $I(q) = I(M(q))$.
    Note that $R(q)$ and $I(q)$ depend only on $m(q)$ and not on $q$ itself.
    Next, we discuss how the regions $R(q)$ and the intervals $I(q)$ relate to each other.

    \begin{claim}\label{claim:diagram_reduction:regions_ordered}
      Let \(q, q' \in B'\). If \(q' \in R^{+}(q)\), then \(I(q') \subseteq I(q)\) and if \(q' \in R^{-}(q)\), then \(I(q) \subseteq I(q')\).
      % \begin{enumerate}
      %     \item\label{b'inr} if \(q' \in R(q)\), then \(R(q') \subseteq R(q)\) and
      %     \item\label{b'notinr} if \(q \not \in R(q)\), then \(R(q) \subseteq R(q')\).
      % \end{enumerate}
  \end{claim}
  \begin{proofclaim}
    If $m(q) = m(q')$, then $I(q) = I(q')$, and so, the claim holds.
    Thus, we assume that $m(q) \neq m(q')$.
    Note that by definition, $m(q) \parallel m(q')$ in $Q$, and therefore, every witnessing path from $m(q')$ to $q'$ is disjoint with $M(q)$.
    In particular, by \ref{items:topological_observations:dont_intersect_M}, if $q' \in R^{+}(q)$, then $m(q') \in R^{+}(q)$, and if $q' \in R^{-}(q)$, then $m(q') \in R^{-}(q)$.

    Next, we prove that for each $\sigma \in \{+,-\}$ if $m(q') \in R^{\sigma}(q)$, then for each $s \in \{\alpha,\beta\}$, either $v^s(m(q')) = v^s(m(q))$ or $v^s(m(q')) \in R^{\sigma}(q)$.
    Indeed, if $M_s(q')$ is disjoint from $M(q)$, then by~\ref{items:topological_observations:dont_intersect_M}, $v^s(m(q')) \in R^{\sigma}(q)$.
    On the other hand, if $M_s(q')$ intersects $M(q)$, then by \Cref{IntersectionMalphaMbeta}, $M_s(q')$ intersects $M_s(q)$, and by \Cref{IntersectionMsMs}, $v^s(m(q')) = v^s(m(q))$.

    This completes the proof of the claim by~\ref{items:topological_observations:disjoint_with_R-} in the case, where $\sigma = +$, and by~\ref{items:topological_observations:disjoint_with_R+} in the case, where $\sigma = -$.
  \end{proofclaim}

    For each $q \in B'$, the elements of $Q$ are partitioned into three sets $R^{+}(q)$, $R^{-}(q)$, and elements of $M(q)$.
    All elements of $M(q)$ are comparable with $q$ in $Q$.
    Consider a pair $(a,b) \in J$.
    It follows that $a \in R^{+}(b) \cup R^{-}(b)$.
    Thus, we define 
        \[J^+ = \{(a,b) \in J : b \in B' \text{ and } a \in R^{+}(b) \} \text{ and } J^- = \{(a,b) \in J : b \in B' \text{ and } a \in R^{-}(b)\}.\]
    Recall that $b \in B'$ if and only if $(a,b) \notin J_1 \cup J_2$.
    Therefore, $J^{+} \cup J^{-} = J \setminus (J_1 \cup J_2)$.

    % Consider an alternating cycle $((a_1,b_1),\dots,(a_k,b_k))$ in $Q$ fully contained in one of the sets $I^+$ or $I^-$.
    % We claim that $m(b_1) = \dots = m(b_k)$.

    \begin{claim}\label{claim:diagram_reduction_alternating_cycle}
        Let $((a_1,b_1),\dots,(a_k,b_k))$ be alternating cycle in $Q$ with all pairs in $J^+$ or in all sets in $J^-$.
        Then, $I(b_1) = \dots = I(b_k)$.
        In particular, $v^\alpha(b_1) = \dots = v^\alpha(b_k)$.
    \end{claim}
    \begin{proofclaim}
        Let $i \in [k]$, and let $\sigma \in \set{+,-}$.
        Assume that all the pairs in the alternating cycle are in $J^\sigma$.
        For every element $u$ in $M(b_i)$, we have $u \leq b_i$ in $Q$.
        Therefore, every witnessing path from $a_i$ to $b_{i+1}$ (we consider indices cyclically) is disjoint from $M(b_i)$.
        Since $a_i \in R^\sigma(b_i)$, by \ref{items:topological_observations:dont_intersect_M}, we have $b_{i+1} \in R^\sigma(b_i)$.
        By \cref{claim:diagram_reduction:regions_ordered}, this yields $I(b_{i+1}) \subset I(b_{i})$ if $\sigma = +$ and $I(b_{i}) \subset I(b_{i+1})$ if $\sigma = -$.
        This holds for each $i \in [k]$, thus, we obtain $I(b_1) = \dots = I(b_k)$.
    \end{proofclaim}

    For each $\sigma \in \{+,-\}$ and $v$ in the exterior face of $Q$, let $J^\sigma_v = \{(a,b) \in J^\sigma : \ v^\alpha(m(b)) = v\}$.
    Note that the choice of $\alpha$ is arbitrary and we could choose $\beta$.
    By \Cref{claim:diagram_reduction_alternating_cycle}, the collection of all $J^\sigma_v$ over $v$ in the exterior face of $Q$ is a splitting partition of $J^\sigma$.
    Moreover, $J^\sigma_v$ is singly constrained by $v$ in $Q$, and $v$ lies in the exterior face of $Q$.
    Let $\calE$ be a combinatorial embedding witnessing a diagram of $Q$ inherited from $P$. 
    Let
    \[J_3 = J^+, \ \ J_4 = J^-, \ \ Q_3 = Q_4 = Q, \ \ \calE_3 = \calE_4 = \calE.\]
    This completes the description of the algorithm and the proof as discussed before.
\end{proof}

\section{Navigating in a planar diagram}
\label{sec:navigation}

In this section, we discuss topological properties of vertically monotone curves and regions that they may form.
Some of the definitions and statements are specific for posets, thus, for the whole section, we fix a poset $P$ with a fixed planar diagram.
Recall that in a fixed planar diagram, we identify elements with points in the plane and cover relations with vertically monotone curves.
Furthermore, witnessing paths are also identified with vertically monotone curves.

%We introduce a notion of left and right specific to the planar diagram.
%For each point \(p \in \mathbb{R}^2\), which may or may not be an element of $P$, we denote its coordinates by \((x(p), y(p))\) \jedrzej{Do we really need named coordinates?}
%\michal{It is easier and more natural this way}.
Let $\gamma_1$ and $\gamma_2$ be vertically monotone curves in the plane.
We say that \(\gamma_1\) is \defin{left} of \(\gamma_2\) if
%(1) there exists at least one
%pair of points \(p_1 \in \gamma_1\) and \(p_2 \in \gamma_2\)
%such that \(y(p_1) = y(p_2)\), and 
%(2) for any such a pair \((p_1, p_2)\),
%we have \(x(p_1) \le x(p_2)\).
% (1) there exists a horizontal line intersecting both $\gamma_1$ and $\gamma_2$,
% (2) for each such line, either $\gamma_1$ and $\gamma_2$ intersect the line at the same point or the intersection of $\gamma_1$ with the line is left of the intersection of $\gamma_2$ with the line.
\begin{enumerate}
    \item there exists a horizontal line $\ell$ intersecting both $\gamma_1$ and $\gamma_2$ such that the intersection of $\gamma_1$ with $\ell$ is left of the intersection of $\gamma_2$ with $\ell$,
    \item for each horizontal line $\ell$ intersecting both $\gamma_1$ and $\gamma_2$, either $\gamma_1$ and $\gamma_2$ intersect $\ell$ at the same point or the intersection of $\gamma_1$ with $\ell$ is left of the intersection of $\gamma_2$ with $\ell$.\label{item:diagram:definition_left_curves_all_lines}
\end{enumerate}
%the $x$-coordinate of the intersection with $\gamma_1$ is not greater than the $x$-coordinate of the intersection with $\gamma_2$.
When \(\gamma_1\) is left of \(\gamma_2\), we also say that
\(\gamma_2\) is \defin{right} of \(\gamma_1\).
% If \(\gamma_1\) is left of \(\gamma_2\) and \(\gamma_1 \cap \gamma_2 = \emptyset\), we say that \defin{\(\gamma_1\) is strictly left of \(\gamma_2\)} and \defin{\(\gamma_2\) is strictly right of \(\gamma_1\)}.
See \cref{fig:curves-order} for some examples.

\begin{figure}[tp]
    \centering
    \includegraphics{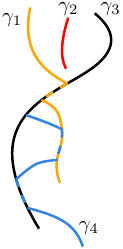}
    \caption{The relation of ``being left of'' for vertically monotone curves is not transitive: here, \(\gamma_i\) is left of \(\gamma_j\) if and only if \(j\equiv i+1 \pmod 4\).}
    \label{fig:curves-order}
\end{figure}

    % \begin{figure}[tp]
    %     \centering 
    %     \includegraphics[scale=1]{figs/left-right-curves.pdf} 
    %     \caption{
    %         Examples of relations between vertically monotone curves.
    %         \michal{
    %         Would be a nice to have a figure with a pair of consistent curves. Also, in my opinion,
    %         4 subfigures would look nicer than 1 figure split into 4 parts.}
    %     } 
    %     \label{fig:left-right-curves} 
    % \end{figure}

We identify each point of the plane with a curve consisting only of that point, so
we can say that a point is left/right of a vertically monotone curve.
Clearly, if \(\gamma\) and \(\gamma'\) are two vertically monotone curves in the plane such that
\(\gamma'\) contains a point left of \(\gamma\) and a point right of \(\gamma\),
then \(\gamma \cap \gamma' \neq \emptyset\).
Furthermore, we identify each witnessing path in $P$ 
with the vertically monotone curve representing it in the diagram.
This way, we can say that \(\gamma_1\) is left/right of \(\gamma_2\)
whenever each of \(\gamma_1\) and \(\gamma_2\) is a
vertically monotone curve in the plane, or a point in the plane
(possibly an element of $P$), or a witnessing path in $P$.

The following topological observation will be very useful.

\begin{obs}
\label{obs:curves-region}
    Let $\gamma$ be a simple closed curve that is the union of finitely many vertically monotone curves.
    Then, a point $p$ not in $\gamma$ lies in the region of $\gamma$ if and only if $p$ is right of exactly an odd number of the curves.
\end{obs}

%For every element $u$ in $P$, let $e_u^{\leftarrow}$ be a curve leaving $u$ horizontally to the left short enough such that the only common point of $e_u^{\leftarrow}$ with the diagram is $u$.
%Note that $e_u^{\leftarrow}$ is a half-edge incident to $u$.
%For all elements \(u, v \) in $P$ with \(u \le v\) in $P$, let $\WL(u,v)=\WL(e_u^{\leftarrow},u,v)$ and $\WR(u,v) = \WR(e_u^{\leftarrow},u,v)$.
%For every $u$ in $P$ with $x_0 \leq u$ in $P$, let $\WL(u) = \WL(e_{x_0}^{\leftarrow},x_0,u)$ and $\WR(u) = \WR(e_{x_0}^{\leftarrow},x_0,u)$.

We introduce similar notation for concatenating curves as for paths in \Cref{ssec:graphs}.
For a vertically monotone curve $\gamma$ and two points $u$ and $v$ in $\gamma$, let $u[\gamma]v$ be the section of $\gamma$ between $u$ and $v$.
Let \(p_0, \ldots, p_k\) be points in the plane, and let \(\gamma_1, \ldots, \gamma_k\) be curves in the plane such that each \(\gamma_i\) contains \(p_{i-1}\) and \(p_i\). 
Then we denote by \defin{\(p_0 [\gamma_1] p_1 [\gamma_2] \cdots [\gamma_k] p_k\)} the union of the curves \(p_{i-1} [\gamma_i] p_i\).

The following straightforward observation will be useful.

\begin{obs}\label{obs:where-curves-end}
    Let $\gamma$ and $\gamma'$ be two disjoint vertically monotone curves such that the lowest point of $\gamma'$ is right (resp.\ left) of $\gamma$.
    Then, the highest point of $\gamma'$ is either right (resp.\ left) of $\gamma$ or higher than any point of $\gamma$.
    Moreover, in the latter case, the highest point of $\gamma$ is left (resp.\ right) of $\gamma'$.
\end{obs}

We say that two vertically monotone curves $\gamma$ and $\gamma'$ are \defin{bottom-consistent} if they share the lowest point $u$, and there exists a point $v$ in the plane such that $u[\gamma]v = u[\gamma']v$ and $\gamma - u[\gamma]v$ is disjoint from $\gamma' - u[\gamma']v$.
We say that two vertically monotone curves $\gamma$ and $\gamma'$ are \defin{top-consistent} if they share the highest point $u$, and there exists a point $v$ in the plane such that $v[\gamma]u = v[\gamma']u$ and $\gamma - v[\gamma]u$ is disjoint from $\gamma' - v[\gamma']u$.
We say that two vertically monotone curves $\gamma$ and $\gamma'$ are \defin{consistent} if they are bottom-consistent, top-consistent, or disjoint.
Since each witnessing path in $P$ is identified with a vertically monotone curve, the notions of bottom-consistency, top-consistency, and consistency also apply to witnessing paths.
%See \fig.

Next, we prove that checking the relation of being left/right between two vertically monotone curves is much easier for consistent curves as it suffices to find only one point witnessing the relation.

% Let $U=u_0\ldots u_\ell$, $V = v_0\dots v_{m}$ be two witnessing paths in $P$ with $u_0 \leq u_\ell$ in $P$ and $v_0 \leq v_m$ in $P$.
% We say that 
% Two paths $U=u_0\ldots u_\ell$, $V = v_0\dots v_{m}$ in $P$ are \defin{consistent} if $u_0=v_0$ and there exist non-negative integer $i$ such that $u_0Uu_i = v_0Vv_i$ and the sets $\{u_{i+1},\dots,u_\ell\}$ and $\{v_{i+1},\dots,v_m\}$ are disjoint.

% For two curves satisfying a certain condition, it suffices to find one horizontal line as in \ref{item:diagram:definition_left_curves_all_lines} in order to verify that one is left of the other.
% The condition is a generalization of the one stated in \cref{sec:preliminaries} for witnessing paths.
% First, for a curve $\gamma$ and two points $u,v$ in $\gamma$, let $u[\gamma]v$ be the section of gamma between $u$ and $v$.
% We say that two vertically monotone curves $\gamma$ and $\gamma'$ are \defin{consistent} if they share the lowest point $u$, and there exists a point $v$ in the plane such that $u[\gamma]v = u[\gamma']v$ and $\gamma - u[\gamma]v$ is disjoint from $\gamma' - u[\gamma']v$.

\begin{proposition}\label{prop:diagram:for_consistent_one_distinction_is_enough}
    % Let $W,W'$ be consistent witnessing paths in $P$. 
    % Then, $W$ is left of $W'$ if and only if there is a horizontal line $\ell$ such that $\ell$ intersects both $W$ and $W'$ and the intersection of $W$ with $\ell$ is left of the intersection of $W'$ with $\ell$.
    Let $\gamma$ and $\gamma'$ be two vertically monotone curves that are
    consistent or disjoint.
    %consistent
    Then, $\gamma$ is left of $\gamma'$ if and only if there is a horizontal line $\ell$ such that $\ell$ intersects both $\gamma$ and $\gamma'$ and the intersection of $\gamma$ with $\ell$ is left of the intersection of $\gamma'$ with $\ell$.
\end{proposition}
\begin{proof}
    We need to prove an equivalence and the forward implication is clear by definition of $\gamma$ being left of $\gamma'$. 
    In order to prove the backward implication, suppose to the contrary that 
    there is a horizontal line $\ell$ such that 
    the intersection of $\gamma$ with $\ell$ is left of the intersection of $\gamma'$ with $\ell$, and 
    $\gamma$ is \emph{not} left of $\gamma'$. 
    This implies that there is a horizontal line $\ell'$ such that 
    the intersection of $\gamma$ with $\ell'$ is right of the intersection of $\gamma'$ with $\ell'$. 
    Clearly, $\ell\neq\ell'$. 
    Since $\gamma$ and $\gamma'$ are both vertically monotone, the intersection of $\gamma$ and $\gamma'$ differ on both $\ell$ and $\ell'$, and they are in the opposite order, by Darboux property, $\gamma$ and $\gamma'$ intersect in the horizontal strip between $\ell$ and $\ell'$.
    This contradicts the consistency or disjointness 
    of $\gamma$ and $\gamma'$.
\end{proof}

Given two comparable elements in $P$, we want to find the extreme witnessing paths between them, that is, the leftmost path and the rightmost path.

\begin{proposition}\label{claim:diagram:extreme_paths_exist}
    Let $u$ and $v$ be elements of $P$ with $u \leq v$ in $P$.
    There exist witnessing paths $W_L$ and $W_R$ from $u$ to $v$ such that for each witnessing path $W$ from $u$ to $v$, either $W = W_L$ or $W_L$ is left of $W$ and either $W = W_R$ or $W_R$ is right of $W$.
    Moreover, the paths $W_L$ and $W_R$ can be computed in polynomial time.
\end{proposition}
\begin{proof}
    We will find the path $W_L$, the proof for existence of $W_R$ is symmetric. 
    We claim that for each two witnessing paths $W_1$ and $W_2$ from $u$ to $v$ there is a witnessing path $W$ from $u$ to $v$ such that $W$ is left of both $W_1$ and $W_2$.
    For each horizontal line that intersects \(W_1\) and \(W_2\), consider the leftmost point of the intersection with either of these curves.
    These points form a vertically monotone curve $\gamma$ between $u$ and $v$.
    By construction, $\gamma$ is left of $W_1$ and $\gamma$ is left of $W_2$.
    It suffices to show that $\gamma$ corresponds to a witnessing path from $u$ to $v$. 
    When we traverse $\gamma$ from $u$ to $v$, 
    we move monotone vertically, and whenever we traverse a point corresponding to an element of $P$ we enter an edge $e$ in the union of $W_1$ and $W_2$. 
    Since the drawing is planar, no two edges cross, and so, $\gamma$ contains the whole $e$, therefore, also its other endpoint. 
    This shows that $\gamma$ corresponds to a witnessing path from $u$ to $v$ in $P$ as desired.
\end{proof}

Let $u$ and $v$ be elements in $P$ with $u \leq v$ in $P$. 
% \michal{Why not \(u \le v\)?}
% \jedrzej{Ok}
By \cref{claim:diagram:extreme_paths_exist}, there exist the \defin{leftmost} witnessing path from $u$ to $v$ and the \defin{rightmost} witnessing path from $u$ to $v$, denoted by \defin{$W_L(u,v)$} and \defin{$W_R(u,v)$}, respectively.
See \Cref{fig:WLWR}.
\begin{figure}
    \centering
    \includegraphics{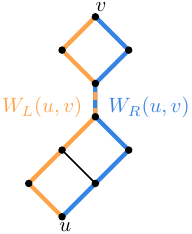}
    \caption{The leftmost and rightmost witnessing paths from \(u\) to \(v\).}
    \label{fig:WLWR}
\end{figure}
% \michal{Add: Observe that for any elements \(u, v\) with \(u \le v\) in $P$ and any elements \(u', v' \in V(W_L(u, v))\) with \(u' \le v'\) in $P$, the path \(W_L(u', v')\) is a subpath of \(W_L(u, v)\). Similarly, for any elements \(u', v' \in V(W_R(u, v))\), the path \(W_R(u', v')\) is
% a subpath of \(W_R(u, v)\). This implies that for any \(b_1, b_2 \in B\), the paths \(W_L(b_1)\) and \(W_L(b_2)\) are consistent, and likewise are \(W_R(b_1)\) and \(W_R(b_2)\).}
% \jedrzej{You want to remove the next claim?}
% \michal{Yes.} 
% \piotr{We disagree.}
Next, we derive some properties of the extreme witnessing paths.

\begin{proposition}\label{claim:diagram:extremal_paths_are_consistent}
    % Let $b,b' \in B$.
    % The paths $W_L(b)$ and $W_L(b')$ are consistent, and the paths $W_R(b)$ and $W_R(b')$ are consistent.    
    Let $D \in \set{L,R}$.
    Let $a,b,c,d$ be elements of $P$ with $a \leq b$, $a \leq d$, and $c \leq b$ in $P$. 
    Then, the paths $W_D(a,b)$ and $W_D(a,d)$ are bottom-consistent, and the paths $W_D(a,b)$ and $W_D(c,b)$ are top-consistent.
\end{proposition}
\begin{proof}
    We will prove the statement for $D = L$ and for the paths $W_L(a,b)$, $W_L(a,d)$ as the proofs of other variants are symmetric. 
    In order to get a contradiction, suppose that $W_L(a,b)$ and $W_L(a,d)$ are not consistent. 
    Therefore, we have two elements $u$ and $v$ in $P$ such that 
    both of them lie in $W_L(a,b)$ and $W_L(a,d)$, and 
    the paths $u[W_L(a,b)]v$ and $u[W_L(a,d)]v$ intersect only in $\set{u,v}$. 
    Since both $u[W_L(a,b)]v$ and $u[W_L(a,d)]v$ have the same endpoints, they are internally disjoint, and they are both vertically monotone, 
    one of them must be left of the other, say $u[W_L(a,d)]v$. 
    Consider a witnessing path $W=x_0 [W_L(a,b)] u [W_L(a,d)] v [W_L(a,b)] b$.
    We claim that $W$ is left of $W_L(a,b)$, which will be a contradiction. 
    Indeed, for every horizontal line $\ell$ in the strip between $u$ and $v$, 
    the intersection of $W$ and $\ell$ is left of the intersection of $W_L(a,b)$ and $\ell$. 
    Also, for every other horizontal line $\ell$ either 
    the intersection of $W$ with $\ell$ and the intersection of $W_L(a,b)$ and $\ell$ coincide, or 
    $\ell$ is disjoint from both $W$ and $W_L(a,b)$.
    This completes the proof.
    % We will prove the statement for $W_L(b)$ and $W_L(b')$, the proof of the statement for the rightmost paths is symmetric.
    % In order to get a contradiction, suppose that $W_L(b)$ and $W_L(b')$ are not consistent. 
    % Therefore, we have two elements $u$ and $v$ in $P$ such that 
    % both of them lie in $W_L(b)$ and $W_L(b')$, and 
    % the paths $u[W_L(b)]v$ and $u[W_L(b')]v$ intersect only in $\set{u,v}$. 
    % Since both $u[W_L(b)]v$ and $u[W_L(b')]v$ have the same endpoints, they are internally disjoint, and they are both vertically monotone, 
    % one of them must be left of the other, say $u[W_L(b')]v$. 
    % Consider a witnessing path $W=x_0 [W_L(b)] u [W_L(b')] v [W_L(b)] b$.
    % We claim that $W$ is left of $W_L(b)$, which will be a contradiction. 
    % Indeed, for every horizontal line $\ell$ in the strip between $u$ and $v$, 
    % the intersection of $W$ and $\ell$ is left of the intersection of $W_L(b)$ and $\ell$. 
    % Also, for every other horizontal line $\ell$ either 
    % the intersection of $W$ with $\ell$ and the intersection of $W_L(b)$ and $\ell$ coincide, or 
    % $\ell$ is disjoint from both $W$ and $W_L(b)$.
    % This completes the proof.
\end{proof}

\cref{claim:diagram:extremal_paths_are_consistent} in particular, implies the following natural property.

\begin{proposition}\label{claim:extreme_subpaths}
    Let $D \in \set{L,R}$. 
    Let $a$ and $b$ be elements of $P$ with $a \leq b$ in $P$, and let $u$ and $v$ be elements of $W_D(a,b)$ with $u \leq v$ in $P$.
    The path $W_D(u,v)$ is a subpath of $W_D(a,b)$.
\end{proposition}
\begin{proof}
    By \cref{claim:diagram:extremal_paths_are_consistent}, the path $W_D(a,v)$ is bottom-consistent with $W_D(a,b)$.
    Since $v$ is an element of $W_D(a,b)$, this yields that $W_D(a,v)$ is a subpath of $W_D(a,b)$.
    By \cref{claim:diagram:extremal_paths_are_consistent}, the path $W_D(u,v)$ is top-consistent with $W_D(a,v)$.
    Since $u$ is an element of $W_D(a,v)$, this yields that $W_D(u,v)$ is a subpath of $W_D(a,v)$.
    The path $W_D(u,v)$ is a subpath of $W_D(a,v)$ and $W_D(a,v)$ is a subpath of $W_D(a,b)$, hence, $W_D(u,v)$ is a subpath of $W_D(a,b)$.
\end{proof}

For two vertically monotone curves $\gamma$ and $\gamma'$ with a common lowest point $w$, we define \defin{$\gcpe(\gamma,\gamma')$} to be the element $w'$ in $\gamma \cap \gamma'$, where $w[\gamma]w' = w[\gamma']w'$ and this prefix is the longest possible.
Symmetrically, for two vertically monotone curves $\gamma$ and $\gamma'$ with a common highest point $w$, we define \defin{$\lcse(\gamma,\gamma')$} to be the element $w'$ in $\gamma \cap \gamma'$, where $w'[\gamma]w = w'[\gamma']w$ and this suffix is the longest possible.
The abbreviations stand for \q{greatest common prefix-element} and \q{least common suffix-element}.
\begin{figure}
    \centering
    \includegraphics{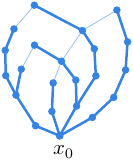}
    \caption{The paths \(W_R(b)\) for \(b \in B\) are pairwise bottom-consistent, so their union is a tree.
    %\jedrzej{I removed the reference to using it implicitly as in fact we try to be explicit always. :)}
    }
    \label{fig:placeholder}
\end{figure}

\begin{proposition}\label{prop:sandwich}
    Let $\gamma_1$, $\gamma_2$, and $\gamma_3$ be vertically monotone curves sharing the lowest (resp.\ highest) element where each pair of the curves is bottom-consistent (resp.\ top-consistent).
    Assume that $\gamma_1$ is left of $\gamma_2$ and $\gamma_2$ is left of $\gamma_3$.
    Let $q = \gcpe(\gamma_1,\gamma_3)$ (resp.\ $m = \lcse(\gamma_1,\gamma_3)$).
    Then, $q$ (resp.\ $m$) lies in $\gamma_2 \setminus \{b\}$ where $b$ is the highest (resp.\ lowest) point of $\gamma_2$.
    %\michal{We should define the interior of a curve (topologically, the interior of a 1-dimensional curve in the 2-dimensional plane is empty.)}
\end{proposition}
\begin{proof}
    Let $x$ be the common lowest point of $\gamma_1$, $\gamma_2$, and $\gamma_3$. 
    Let $b_2$ be the highest point of $\gamma_2$.
    Consider a horizontal line $\ell$ with the vertical coordinate at least as high as $x$ and at most as high as $q$. 
    %\michal{Between can be understood as `connecting`.} \jedrzej{True, better now?} \michal{Better. Changed at least as \(x\) by }
    Thus, $\ell$ intersects all three curves $\set{\gamma_i}_{i\in[3]}$. 
    Say that $\ell$ intersects $\gamma_i$ in the point $p_i$ for each $i \in [3]$. 
    By the assumption, $p_1$ is left of or equal to $p_2$ and $p_2$ is left of or equal to $p_3$.
    Since the line is not higher than $q$, we have $p_1 = p_3$ by the definition of $q$.
    It follows that $p_1 = p_2 = p_3$.
    Since $\gamma_2$ is not contained in $\gamma_1$ (nor in $\gamma_3$), we conclude that $q$ lies in $\gamma_2$ and $q \neq b_2$.
    The proof for the top-consistent case is symmetric.
\end{proof}

\section{The proof}
\label{sec:proof}
In this section, we give an algorithm that solves \SCPDD problem.
Namely, we prove \cref{lem:singly-constrained-proof} that we restate below for convenience.
Along with \Cref{lem:singly-constrained-reduction} proved in \cref{sec:singly_constrained}, this gives \cref{thm:main}.

\lemproof*

The proof follows the lines given in \cref{sec:outline}.
We fix an input to \SCPDD, i.e., a tuple $(n,P,I,x_0,\calE)$ where $n$ is an integer with $n\geq 3$, $P$ is a poset, $x_0$ is an element of $P$, $I \subset \Inc(P)$ is singly constrained in $P$ by $x_0$, and $\calE$ is a combinatorial embedding of $P$ witnessing a planar diagram of $P$ with $x_0$ in the exterior face.

For every \((a, b) \in I\) we have \(x_0 \leq b\) in $P$.
Thus, for each $(a,b)\in I$, \(a \not \le x_0\) in $P$ 
as otherwise $a\leq x_0 \leq b$ in $P$.
Therefore, after removing all elements \(x\) with \(x < x_0\) in \(P\), we obtain
a convex subposet which still satisfies all the requirements of the input.
Hence, without loss of generality, we may assume that $x_0$ is a minimal element of $P$.
Applying a procedure from \Cref{thm:encoding-diagrams}, the algorithm fixes a distinguishing efficient planar diagram of $P$ with $x_0$ in the exterior face.
%Let $G$ be the emerging plane graph isomorphic to the cover graph of $P$.
%\jedrzej{Do we use $G$?}

Let $\defin{B}=U_P[x_0]$ and let $\defin{A}$ be the set of all the elements of $P$ not in $U_P[x_0]$.
Note that for every $(a,b) \in I$, we have $a \in A$ and $b \in B$.

Let \defin{\(\preccurlyeq_{\uparrow}\)} be an ordering of the points in the plane such that for all \(q_1\) and \(q_2\) in $P$, if $q_1$ is not higher than $q_2$, then \(q_1 \preccurlyeq_{\uparrow} q_2\) and if $q_1$ is lower than $q_2$, then \(q_1 \prec_{\uparrow} q_2\). 
Note that $\preccurlyeq_\uparrow$ restricted to the elements of $P$ is a linear order since the fixed diagram is distinguishing.
%\later{In most of the cases when we write $\preccurlyeq_\uparrow$, we could just wrtie $\prec_\uparrow$.}

For each $b \in B$, we abbreviate $\defmath{W_L(b)} = W_L(x_0,b)$ and $\defmath{W_R(b)} = W_R(x_0,b)$.
Let \defin{$\preccurlyeq_L$} and \defin{$\preccurlyeq_R$} be two orderings on $B$ such that for all $b_1,b_2 \in B$, we have $b_1 \preccurlyeq_L b_2$ if either $b_1 = b_2$ or $W_L(b_1)$ is left of $W_L(b_2)$, and $b_1 \preccurlyeq_R b_2$ if either $b_1 = b_2$ or $W_R(b_1)$ is left of $W_R(b_2)$.

\begin{proposition}\label{claim:diagram:L_and_R_are_posets}
    The set $B$ is partially ordered by $\preccurlyeq_L$ and $\preccurlyeq_R$.
\end{proposition}
\begin{proof}
    We will prove the statement for $\preccurlyeq_L$.
    The proof for $\preccurlyeq_R$ is symmetric.
    It suffices to show that for all $b_1,b_2,b_3 \in B$, if $b_1 \prec_L b_2$ and $b_2 \prec_L b_3$, then $b_1 \prec_L b_3$.
    Since $b_1 \prec_L b_2$, there is a horizontal line $\ell$ such that the intersection of $W_L(b_1)$ with $\ell$ is left of the intersection of $W_L(b_2)$ with $\ell$.
%    \michal{Add: by Claim 25? Or make Claim 25 an implication?}
%    \piotr{We don't care.}
    Since $b_2 \prec_L b_3$, there is a horizontal line $\ell'$ such that the intersection of $W_L(b_2)$ with $\ell'$ is left of the intersection of $W_L(b_3)$ with $\ell'$.
    Without loss of generality, assume that $\ell$ is not higher in the plane than $\ell'$.
    It follows that $\ell$ intersects $W_L(b_3)$, and the intersection of $W_L(b_2)$ with $\ell$ is either the same as the intersection of $W_L(b_3)$ with $\ell$ or left of it.
    In particular, the intersection of $W_L(b_1)$ with $\ell$ is left of the intersection of $W_L(b_3)$ with $\ell$.
    By \cref{claim:diagram:extremal_paths_are_consistent}, $W_L(b_1)$ and $W_L(b_3)$ are consistent, and so, by \cref{prop:diagram:for_consistent_one_distinction_is_enough}, $W_L(b_1)$ is left of $W_L(b_3)$ (i.e.\ $b_1 \prec_L b_3$), which ends the proof.
\end{proof}

\begin{proposition}\label{claim:diagram:incomparable_b_are_comparable}
    Let $b,b' \in B$ be incomparable in $P$.
    Then, $b$ and $b'$ are comparable in $\preccurlyeq_L$ and $\preccurlyeq_R$.
\end{proposition}
\begin{proof}
    We will prove the statement for $\preccurlyeq_L$.
    The proof for $\preccurlyeq_R$ is symmetric.
    Assume without loss of generality that $b$ is lower in the diagram than $b'$, and consider the horizontal line $\ell$ intersecting $b$.
    In particular, $\ell$ intersects $W_L(b')$.
    Since $b\not\leq b'$ in $P$ are incomparable, the intersection point of $W_L(b')$ with $\ell$ and $b$ are distinct.
    Therefore, by \cref{prop:diagram:for_consistent_one_distinction_is_enough}, either $W_L(b)$ is left of $W_L(b')$ or $W_L(b')$ is left of $W_L(b)$, which yields that $b$ and $b'$ are comparable in $\preccurlyeq_L$.
\end{proof}

We will need the following simple property of extreme paths.

\begin{proposition}\label{prop:exposed-paths}
    Let $a \in A$ and $b \in B$ with $a \leq b$ in $P$. 
    %\later{Find applications}
    \begin{enumerate}
        \item If $W_L(a,b)$ is left of $W_L(b)$, then for $m = \lcse(W_L(a,b),W_L(b))$, all the elements of $a[W_L(a,b)]m\setminus\{m\}$ are in $A$. \label{prop:exposed-paths:left}
        \item If $W_R(a,b)$ is right of $W_R(b)$, then for $m = \lcse(W_R(b),W_R(a,b))$, all the elements of $a[W_R(a,b)]m\setminus\{m\}$ are in $A$. \label{prop:exposed-paths:right}
    \end{enumerate}
\end{proposition}
\begin{proof}
    We prove only \ref{prop:exposed-paths:right} as the proof of \ref{prop:exposed-paths:left} is symmetric.
    Assume that $W_R(a,b)$ is right of $W_R(b)$.
    Suppose to the contrary that $a[W_R(a,b)]m\setminus\{m\}$ contains an element $u$ in $B$.
    It follows that $x_0 \prec_\uparrow u \prec_\uparrow b$, and so, $u$ is right of $W_R(b)$.
    Thus, $x_0[W_R(u)]u[W_R(a,b)]b$ is a witnessing path from $x_0$ to $b$ in $P$ containing an element right of $W_R(b)$, which contradicts \cref{claim:diagram:extreme_paths_exist}.
\end{proof}

\subsection{Four easy sets}\label{ssec:four-easy-sets}

The first four reversible subsets $J_1,J_2,J_3,J_4 \subset \closure{I}$ will be obtained by applying
\cref{prop:reversible-from-order} to appropriate partial
orders on \(B\).
The following properties may or may not be satisfied for a pair $(a,b) \in \Inc(P)$
    \begin{enumerateNumI}
        \item there is \(q' \in B\) with $a \leq q'$ in $P$ and \(q' \prec_{\uparrow} b\),  \label{items:diagram_pairs:b'_vertical} \label{I-first}
        \item there is \(q'' \in B\) with $a \leq q''$ in $P$ and \(b \prec_{\uparrow} q''\).\label{items:diagram_pairs:b''_vertical}
    \end{enumerateNumI}
Let 
    \begin{align*}
        % \defmath{J_1} = \ext(\closure{I}, \preccurlyeq_{\uparrow}\vert_B) \qquad \text{and} \qquad \defmath{J_2} = \ext(\closure{I}, \succcurlyeq_\uparrow\vert_B).
        %J_1 &= \{(a, b) \in I:\textrm{there is no \(b' \in B\) with $a \leq b'$ in $P$ and \(b' \prec_{\uparrow} b\)}\} \ \text{and}\\
        %J_2 &= \{(a, b) \in I:\textrm{there is no \(b'' \in B\) with $a \leq b''$ in $P$ and \(b \prec_{\uparrow} b''\)}\}.
        \defmath{J_1} &= \{(a,b) \in \closure{I} : (a,b) \text{ does not satisfy \ref{items:diagram_pairs:b'_vertical}}\},\\
        \defmath{J_2} &= \{(a,b) \in \closure{I} : (a,b) \text{ does not satisfy \ref{items:diagram_pairs:b''_vertical}}\},\\
        \defmath{I_1} &= \{(a,b) \in \closure{I} : (a,b) \text{ satisfies \ref{items:diagram_pairs:b'_vertical} and \ref{items:diagram_pairs:b''_vertical}}\}.
    \end{align*}
Also note that $\closure{I} \setminus (J_1 \cup J_2) = I_1$.
\cref{prop:reversible-from-order} implies that the sets $J_1$ and $J_2$ are reversible.
\begin{corollary}\label{cor:J1J2}
    $J_1$ and $J_2$ are reversible in $P$.
\end{corollary}

\begin{proposition}\label{claim:diagram:enclosed:pairs}
  Each \((a, b) \in I_1\) satisfies
    \begin{enumerateNumI}
        \setcounter{enumi}{2}
        \item $a \prec_{\uparrow} b$, \label{items:diagram_pairs:a_lower_b}
        \item \(b \neq x_0\),\label{items:b_in_B}
        %\item there is no $r$ in $P$ with $a \leq r$ in $P$ such that \(r\) is both right of \(\WL(b)\) and left of \(\WR(b)\).\label{items:diagram_pairs:a_not_enclosed_by_bs}
        \item $a$ is not both right of \(\WL(b)\) and left of \(\WR(b)\), and \label{items:diagram_pairs:a_not_enclosed_by_bs}
        %\item for each $q$ in $P$ with $a\leq q$ in $P$, either $q$ is left of $\WL(b)$ or $q$ is right of $\WR(b)$. \label{items:diagram_pairs:a_not_enclosed_by_bs}
        \item for every $d \in B$ with $a \leq d$ in $P$, a witnessing path from $a$ to $d$ is either left of $W_L(b)$ or right of $W_R(b)$.\label{items:diagram_pairs:a-d_paths}
    \end{enumerateNumI}
\end{proposition}
\begin{proof}
    Let $(a,b) \in I_1$.
    By \ref{items:diagram_pairs:b'_vertical}, there is \(q' \in B\) with $a \leq q'$ in $P$ and \(q' \prec_{\uparrow} b\).
    It follows that $a \preccurlyeq_{\uparrow} q' \prec_\uparrow b$, hence, $a \prec_{\uparrow} b$, which proves \ref{items:diagram_pairs:a_lower_b}.
    Since \(x_0\) is the lowest element in \(B\), we have \(x_0 \preccurlyeq_\uparrow q' \prec_{\uparrow} b\), so \ref{items:b_in_B} holds.

    For the proof of \ref{items:diagram_pairs:a_not_enclosed_by_bs} suppose to the contrary that $a$ is right of $\WL(b)$ and $a$ is left of $\WR(b)$. 
    Consider the horizontal line $\ell$ containing $a$. 
    Let $p$ be the intersection of $\WL(b)$ with $\ell$.
    Let $q$ be the intersection of $\WR(b)$ with $\ell$.
    By assumptions, $p$ is left of $a$, which is left of $q$.
    Let $x$ be the lowest common element of $W_L(b)$ and $W_R(b)$ above $\ell$ and let $y$ be the highest common element of $W_L(b)$ and $W_R(b)$ below $\ell$.
    Let $\calR$ be the region of the simple closed curve $x[W_R(b)]y[W_L(b)]x$.
    By \Cref{obs:curves-region}, $a \in \calR$ and $b \notin \Int \calR$.
    Note also that every element $s$ on the boundary of $\calR$ satisfies $s\leq b$ in $P$. 
    By \ref{items:diagram_pairs:b''_vertical}, there is \(q''\) in $P$ with $a \leq q''$ in $P$ and \(b \prec_{\uparrow} q''\).
    Again, by \Cref{obs:curves-region}, $q'' \notin \Int \calR$.
    Let $W''$ be a witnessing path from $a$ to $q''$ in $P$.
    Note that every element $s$ in $W''$ satisfies $a \leq s$ in $P$.
    Therefore, $W''$ is disjoint from the boundary of $\calR$, which is a contradiction.

    For the proof of \ref{items:diagram_pairs:a-d_paths}, 
    let $d\in B$ with $a\leq d$ in $P$, and 
    let $W$ be a witnessing path from $a$ to $d$ in $P$.
    Recall that by~\ref{items:diagram_pairs:a_lower_b} $a \prec_{\uparrow} b$. 
    We distinguish two cases $a\prec_{\uparrow} x_0$ and $x_0 \prec_{\uparrow} a \prec_{\uparrow} b$.
    In the latter case, the horizontal line through $a$ intersects both $\WL(b)$ and $\WR(b)$, so by~\ref{items:diagram_pairs:a_not_enclosed_by_bs}, $a$ is left of $\WL(b)$ or $a$ is right of $\WR(b)$. 
    Since $W$ and $\WL(b)$ (or $\WR(b)$) are disjoint, 
    \cref{prop:diagram:for_consistent_one_distinction_is_enough} implies that $W$ is left of $\WL(b)$ or $W$ is right of $\WR(b)$, as desired. 
    In the case that $a\prec_{\uparrow} x_0$, the path $W$ must intersect the horizontal line going through $x_0$ since $d \in B$, and clearly the intersection point is distinct from $x_0$. 
    Again, \cref{prop:diagram:for_consistent_one_distinction_is_enough} implies that $W$ is left of $\WL(b)$ or $W$ is right of $\WR(b)$. 
    This completes the proof of \ref{items:diagram_pairs:a-d_paths}.
\end{proof}

\begin{proposition}\label{prop:one-side-implies-the-other}
    Let $(a,b) \in I_1$ and let $d \in B$ with $a \leq d$ in $P$.
    If $d \prec_R b$, then $d \prec_L b$. 
    Also, if $b \prec_L d$, then $b \prec_R d$.
\end{proposition}
\begin{proof}
    We prove only the first statement as the proof of the second is symmetric.
    Suppose that $d \prec_R b$.
    First, assume that $d \prec_\uparrow b$.
    By \ref{items:diagram_pairs:a-d_paths}, every witnessing path from $a$ to $d$ is either left of $W_L(b)$ or right of $W_R(b)$.
    In particular, since $d \prec_\uparrow b$, either $d$ is left of $W_L(b)$ or $d$ is right of $W_R(b)$.
    Since \(d \prec_{R} b\), we have $d$ left of $W_R(b)$, and we conclude that $d$ is left of $W_L(b)$.
    By \cref{prop:diagram:for_consistent_one_distinction_is_enough}, this yields that $W_L(d)$ is left of $W_L(b)$, and so, $d \prec_L b$.
    Next, we assume that $b \prec_\uparrow d$.
    Since \(d \prec_{R} b\) and by \cref{claim:diagram:extremal_paths_are_consistent}, we have $W_R(d)$ left of $b$.
    It follows that $W_L(d)$ is left of $b$.
    Again by \cref{prop:diagram:for_consistent_one_distinction_is_enough}, we conclude that $W_L(d)$ is left of $W_L(b)$, and so, $d \prec_L b$.
\end{proof}

The following properties may or may not be satisfied for a pair $(a,b) \in \Inc(P)$ 
    \begin{enumerateNumI}
        \setcounter{enumi}{6}
        \item there is \(b' \in B\) with $a \leq b'$ in $P$, \(b' \prec_{L} b\), and \(b' \prec_{R} b\), \label{items:diagram_pairs:b'}
        %(in other words $B'(a,b) \neq \emptyset$),  
        \item there is \(b'' \in B\) with $a \leq b''$ in $P$, \(b \prec_{L} b''\), and \(b \prec_{R} b''\). \label{items:diagram_pairs:b''}
        %(in other words $B''(a, b) \neq \emptyset$).
    \end{enumerateNumI}
Let
\begin{align*}
    \defmath{J_3} &= \{(a,b) \in I_1 : (a,b) \text{ does not satisfy \ref{items:diagram_pairs:b'}}\},\\
    \defmath{J_4} &= \{(a,b) \in I_1 : (a,b) \text{ does not satisfy \ref{items:diagram_pairs:b''}}\},\\
    \defmath{I_2} &= \{(a,b) \in I_1 : (a,b) \text{ satisfies \ref{items:diagram_pairs:b'} and \ref{items:diagram_pairs:b''}}\}\\
    &=\{(a,b) \in \closure{I} : (a,b) \text{ satisfies \ref{items:diagram_pairs:b'_vertical}, \ref{items:diagram_pairs:b''_vertical}, \ref{items:diagram_pairs:b'}, and \ref{items:diagram_pairs:b''}}\}.
\end{align*}
Since $\preccurlyeq_L$ and $\preccurlyeq_R$ partially order $B$ 
(by \cref{claim:diagram:L_and_R_are_posets}).\, 
% Let
% % For each $D \in \{L,R\}$, let \(\preccurlyeq_{D}\) be a partial order on \(B\) such that for all $b_1,b_2 \in B$, we have \(b_1 \preccurlyeq_{D} b_2\) whenever either $b_1 = b_2$ or \(W_D(b_1)\) is left of \(W_D(b_2)\) in the diagram and none of the paths \(W_D(b_1)\) and
% % \(W_D(b_2)\) is a subpath of the other.
%     \begin{align*}
%         \defmath{J_3} = \ext(I_1, \preccurlyeq_{R}) \ \ \ \ \text{and} \ \ \ \ \defmath{J_4} = \ext(I_1, \succcurlyeq_L).
%         %J_3 &= \{(a, b) \in I:\textrm{there is no \(b' \in B\) with $a < b'$ in $P$ and \(b' \prec_{R} b\)}\} \ \text{and}\\
%         %J_4 &= \{(a, b) \in I:\textrm{there is no \(b'' \in B\) with $a < b''$ in $P$ and \(b \prec_{L} b''\)}\}.
%     \end{align*}
By \cref{claim:diagram:incomparable_b_are_comparable}, for each $b,b' \in B$ with $b \parallel b'$ in $P$, $b$ and $b'$ are comparable in $\preccurlyeq_L$ and $\preccurlyeq_R$, hence, by \cref{prop:reversible-from-order}, the sets $J_3$ and $J_4$ are reversible.
Additionally, by \Cref{prop:one-side-implies-the-other}, $\closure{I} \setminus (J_1 \cup J_2 \cup J_3 \cup J_4) = I_2$.

\begin{corollary}\label{cor:J3J4}
    $J_3$ and $J_4$ are reversible in $P$.
\end{corollary}

\begin{proposition}\label{clm:ainA}
    Each \((a, b) \in I_2\) satisfies
    \begin{enumerateNumI}
        \setcounter{enumi}{8}
        \item \(a \in A\). \label{items:diagram_pairs:a_in_A}
    \end{enumerateNumI}
\end{proposition}
\begin{proof}
    Recall that $a \neq x_0$.
    Suppose to the contrary that \(a \not \in A\), hence, $a \in B \setminus\{x_0\}$.
    Since \(a \parallel b\) in $P$, by \cref{claim:diagram:incomparable_b_are_comparable}, the elements \(a\) and \(b\) are
    comparable in \(\preccurlyeq_L\) and \(\preccurlyeq_R\).
    By~\ref{items:diagram_pairs:a_lower_b}, we have $x_0 \prec_\uparrow a \prec_\uparrow b$.
    In particular, $a$ is either left or right of each of $\WL(b)$ and $\WR(b)$.
    By~\ref{items:diagram_pairs:a_not_enclosed_by_bs}, the element \(a\) is not
    simultaneously right of \(\WL(b)\) and left of \(\WR(b)\).
    It follows that either $a$ is left of $\WL(b)$ or $a$ is right of $\WR(b)$.
    The two cases are symmetric, so without loss of generality, assume that $a$ is right of $\WR(b)$.
    % Since $\WR(b)$ and $\WR(a)$ are consistent by \cref{claim:diagram:extremal_paths_are_consistent}, by \cref{prop:diagram:for_consistent_one_distinction_is_enough}, $\WR(a)$ is right of $\WL(b)$, and so, $b \prec_R a$.
    %Hence, \(a \prec_L b\) or \(b \prec_R a\).
    %The two cases are symmetric, so without loss of generality, we assume that \(b \prec_R a\).
    By~\ref{items:diagram_pairs:b'}, there exists \(b' \in B\) with $a \leq b'$ in $P$, \(b' \prec_{L} b\), and \(b' \prec_{R} b\).
    Choose a witnessing path $W$ from $a$ to $b'$, and let 
    The path $W' = x_0[W_R(a)]a[W]b'$ is a witnessing path from $a$ to $b'$ in $P$ containing an element (specifically $a$) right of $W_R(b)$, which is right of $W_R(b')$, which is a contradiction with \cref{claim:diagram:extreme_paths_exist}.
    % Consider the horizontal line $\ell$ containing $a$.
    % The intersection of $W_R(b)$ with $\ell$ is left of $a$, which is equal to the intersection of $W'$ with $\ell$.
    % Since $b' \prec_R b$, the intersection of $W_R(b')$ with $\ell$ is not right of the intersection of $W_R(b)$ with $\ell$.
    % Summarizing, the intersection of $W_R(b')$ with $\ell$ is left of the intersection of $W'$ with $\ell$, which yields that $W_R(b')$ is not right of $W'$, which contradicts \cref{claim:diagram:extreme_paths_exist}.
\end{proof}

\begin{proposition}\label{prop:strict-cycle-ordering-bs}
    Let $((a_1,b_1), \dots, (a_k,b_k))$ be a strict alternating cycle in $P$ contained in $I_2$.
    For all distinct $i,j \in [k]$, either $b_i \prec_L b_j$ and $b_i \prec_R b_j$ or $b_j \prec_L b_i$ and $b_j \prec_R b_i$.
\end{proposition}
\begin{proof}
    Suppose to the contrary that there exist distinct $i,j \in [k]$ with $b_i \prec_L b_j$ and $b_j \prec_R b_i$.
    Note that since the cycle is strict, $b_i$ and $b_j$ are incomparable in $P$.
    By renumbering the cycle, we may assume that $i = 1$ and we set $j$ to be minimal such that $b_1 \prec_L b_j$ and $b_j \prec_R b_1$.

    Note that if for an element $d \in B$, we have $b_1 \prec_L d$ and $d \prec_R b_1$, then $d \prec_\uparrow b_1$.
    Indeed, otherwise, consider the horizontal line $\ell$ containing $b_1$.
    Both paths $W_L(b_1)$ and $W_R(b_1)$ intersect $\ell$ in $b_1$, thus, it is not possible that the intersections with $\ell$ are ordered $W_L(b_1)$, $W_L(d)$, $W_R(d)$, and $W_R(b_1)$ from left to right.
    It follows that indeed, $d \prec_\uparrow b_1$.
    In particular, this applies to $d = b_j$, 
    so $b_j \prec_\uparrow b_1$.

    Let $W$ be a witnessing path from $a_{j-1}$ to $b_j$ in $P$.
    Since the cycle is strict, $W$ is disjoint from $W_L(b_1)$ and $W_R(b_1)$.
    Since $b_j \prec_\uparrow b_1$, $b_j$ is right of $W_L(b_1)$ and left of $W_R(b_1)$.
    In particular, by \Cref{prop:diagram:for_consistent_one_distinction_is_enough}, $W$ is right of $W_L(b_1)$ and left of $W_R(b_1)$.
    It also follows that $x_0 \prec_{\uparrow} a_{j-1}$, and $a_{j-1}$ is right of $W_L(b_1)$ and left of $W_R(b_1)$.
    Note that if $j = 2$, then this contradicts \ref{items:diagram_pairs:a_not_enclosed_by_bs}.
    Hence, from now on, we assume that $j > 2$.

    Let $b'$ and $b''$ be the elements provided by \ref{items:diagram_pairs:b'} and \ref{items:diagram_pairs:b''}, respectively.
    Let $W'$ and $W''$ be witnessing paths from $a_{j-1}$ to $b'$ and $b''$ in $P$, respectively.
    Again, the paths $W'$ and $W''$ are disjoint from $W_L(b_1)$ and $W_R(b_1)$ as the cycle is strict.
    By \cref{prop:diagram:for_consistent_one_distinction_is_enough},  both $W'$ and $W''$ are right of $W_L(b_1)$ and left of $W_R(b_1)$.
    In particular, by the first observation of this proof yields $b' \prec_\uparrow b_1$ and $b'' \prec_\uparrow b_1$.
    It follows that $b_1 \prec_L b'$ and $b'' \prec_R b_1$.
    Since $b' \prec_L b_{j-1}$ and $b_{j-1} \prec_R b''$, we obtain $b_1 \prec_L b_{j-1}$ and $b_{j-1} \prec_R b_1$, which contradicts the minimality of $j$ and completes the proof.
\end{proof}

% Define partial orders \(\peR\) and \(\peL\) on \(B\), where for \(b_1, b_2 \in B\) and
% \(X \in \{L, R\}\) we have \(b_1 \preccurlyeq_X b_2\)  if and only if either \(b_1 = b_2\)
% or \(W_X(b_1)\) left of \(W_X(b_2)\) in the diagram and none of the paths \(W_X(b_1)\) and
% \(W_X(b_2)\) is a subpath of the other. Observe that
% if \(b_1\) and \(b_2\) are incomparable in $P$, then they are comparable in \(\preccurlyeq_X\).
% Therefore, by \cref{prop:reversible-from-order} applied to \(\peR\) and \(\succcurlyeq_L\), the sets
% \[
% \ext(J, \preccurlyeq_R) = \{(a, b) \in J: \textrm{there is no \(b' \in B \cap U_Q[a]\) such that \(b' \prec_R b\)}\}
% \]
% and
% \[
% \ext(J, \succcurlyeq_L) = \{(a, b) \in J: \textrm{there is no \(b'' \in B \cap U_Q[a]\) 
% such that \(b \prec_L b''\)}\}.
% \]
% are reversible.

% For each \((a, b) \in J\), we define sets
% \[
%   B'(a, b) := \{b' \in B \cap U_Q[a]: b' \prec_R b\},
% \]
% and
% \[
%   B''(a, b) := \{b'' \in B \cap U_Q[a]: b \prec_L b''\}.
% \]
% Hence, for every \((a, b) \in J \setminus (\ext(J, \preccurlyeq_R) \cup
% \ext(J, \preccurlyeq_L))\), we have \(B'(a, b) \neq \emptyset\)
% and \(B''(a, b) \neq \emptyset\).

\subsection{Bottom elements}\label{ssec:bottom}

An element $a\in A$ is a \defin{bottom element} if 
%Let us call a pair \((a, b) \in I_2\) a \defin{bottom} pair if
 there exists $b^*,b^{**} \in B$ and witnessing paths $W^*$ from $a$ to $b^*$ and $W^{**}$ from $a$ to $b^{**}$ such that $W^*$ is left of $x_0$ and $W^{**}$ is right of $x_0$.
 Note that $a \prec_\uparrow x_0$.
% \(a\) is drawn below \(x_0\) and there exist witnessing paths
% \(W_1\) and \(W_2\) in $P$ each from \(a\) to an element in \(B\),
% such that \(x_0\) is right of \(W_1\) and left of \(W_2\).
%Let \(J_{\textrm{bot}}\) denote the set of all bottom pairs.
See \cref{fig:bottom-element}.
\begin{figure}
    \centering
    \includegraphics{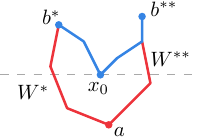}
    \caption{A bottom element \(a\). 
    %\jedrzej{Maybe a dotted / gray horizontal line through $x_0$?} \michal{Added!}
    } 
    \label{fig:bottom-element}
\end{figure}
In this subsection, we construct four reversible sets $J_5,J_6,J_7,J_8$ that cover the set
%We will next show that \(\dim_Q(J_{\textrm{bot}}) \le 4\).
\begin{align*} %\label{eq:bottom_leq_4}
    \set{(a,b)\in I_2 :\ \textrm{$a$ is a bottom element}}.    
\end{align*}

Since $x_0$ is a minimal element of $P$, every edge of the cover graph of $P$ incident with \(x_0\) has its lowest point in \(x_0\).  
Let \(e_1, \ldots, e_t\) be the edges of the cover graph of \(P\) incident with \(x_0\) listed from left to right, and denote \(e_0 = e_t\).
The vertex \(x_0\) is on the boundary
of the exterior face, so for some \(i \in [t]\) the edges \(e_{i-1}\) and \(e_i\) are consecutive edges in the facial walk along the exterior face (the indices are interpreted
cyclically, so \(e_0 = e_t\)).
We choose a non-trivial vertically monotone curve \defin{$\gamma_\infty$} in the exterior face of $P$ which has its lowest point in \(x_0\) and is otherwise disjoint from the diagram.
The curve \(\gamma_\infty\) may be left of \(e_1\) or right of \(e_t\) or 
between \(e_{i-1}\) and \(e_i\) for some \(i \in [t]\).
%\jedrzej{There is a small hack here. Maybe if there is no bottom pair, just define $J_5,\dots,J_8$ as empty and here assume that there is at least one bottom pair.}
In particular, $\gamma_\infty$ is consistent with every witnessing path from $x_0$.
By \cref{prop:diagram:for_consistent_one_distinction_is_enough}, this yields that $\gamma_\infty$ is either left or right of every non-trivial witnessing path from $x_0$.

% The element \(x_0\) is a minimal element of $P$ on the boundary of the exterior face
% of $P$. Therefore, there exists a non-trivial \(y\)-monotone curve \(\gamma_\infty\)
% such that \(x_0\textbf{}\) is the only point of \(\gamma_\infty\)
% on the diagram and the lowest point of \(\gamma_\infty\), and \(\gamma_\infty \setminus \{x_0\}\) is contained 
% in the exterior face of the diagram. Fix such a curve \(\gamma_\infty\).

\begin{proposition}\label{claim:diagram:e_infty}
  For every \(b \in B\)
  either all witnessing paths from $x_0$ to $b$
  are left of \(\gamma_\infty\), or all witnessing paths from $x_0$ to $b$ are right of \(\gamma_\infty\).
\end{proposition}
\begin{proof}
    Suppose to the contrary that there exists $b$ and two witnessing paths 
    $W, W'$ from  $x_0$ to $b$ such that 
    \(W\) is left of \(\gamma_\infty\) and
    \(W'\) is right of \(\gamma_\infty\).
    Therefore \(\gamma_\infty\) contains a point which is right of \(W_L(b)\) and
    right of \(W_R(b)\), and thus lies in the interior of the
    region bounded by the closed walk
    obtained as the union of \(W_L(b)\) and \(W_R(b)\).
    This contradicts the fact that \(\gamma_\infty\) belongs to the exterior face
    of the diagram.
\end{proof}

Let \defin{$B_L$} be the set of all $b \in B$ such that every witnessing path from $x_0$ to $b$ is left of $\gamma_\infty$, and let \defin{$B_R$} be the set of all $b \in B$ such that every witnessing path from $x_0$ to $b$ is right of $\gamma_\infty$.
By \cref{claim:diagram:e_infty}, $B = B_L \cup B_R$.

% \begin{proposition}
%     For every bottom element $a\in A$, every $b'\in B$ with $a\leq b'$ in $P$, 
%     and every witnessing path \(W\) from \(a\) to $b'$, 
%     if \(b' \in B_R\), then \(W\) is right of \(x_0\), and if \(b' \in B_L\), then \(W\) is left of \(x_0\).
% \end{proposition}
% \begin{proof}
%     We only show that if \(b' \in B_R\), then \(W\) is right of \(x_0\) because the
%     proof of the other item is symmetric.
%     Suppose towards a contradiction that \(b' \in B_R\) and \(W\) is not right of \(x_0\).
%     Since \(a\) is incomparable with \(b\), the witnessing path \(W\) does not contain
%     \(x_0\).
%     As \(a\) is below \(x_0\) and \(b'\) is above \(x_0\), this implies that
%     \(W\) is left of \(x_0\). Hence, \(W_R(a, b')\) is left of \(x_0\) as well.
%     The pair \((a, b)\) is bottom, so there exists an witnessing path \(W'\) from
%     \(a\) to an element \(b'' \in B\) which is right of \(x_0\).
%     TODO.
% \end{proof}

    \begin{proposition}\label{claim:diagram:bottom_as}
    For every bottom element $a\in A$, every $d\in B$ with $a\leq d$ in $P$, 
    and every witnessing path \(W\) from \(a\) to $d$, 
    if \(d \in B_R\), then \(W\) is right of \(x_0\), and if \(d \in B_L\), then \(W\) is left of \(x_0\).
    \end{proposition}
    \begin{proof}
        We prove only the first statement, as the proof of the second one is symmetric.
        Suppose to the contrary that there is a witnessing path $W$ from $a$ to $d$ that is left of $x_0$.
        Since $a$ is a bottom element, we can fix $b^{**}\in B$ and a witnessing path $W^{**}$ from $a$ to $b^{**}$ such that $W^{**}$ is right of $x_0$.
        By \Cref{claim:diagram:extreme_paths_exist}, $W_R(a,b^{**})$ is right of $x_0$.
        Let $q = \gcpe(W,W_R(a,b^{**}))$, $m = \lcse(W,W_R(d))$, and $W' = a[W_R(a,b^{**}]q[W]m[W_R(d)]d$.
        Let $m^{**} = \lcse(W_R(b^{**}),W_R(a,b^{**}))$.

        Since $\gamma_\infty$ intersects the diagram of $P$ only in $x_0$, $W$ left of $x_0$ implies $W$ left of $\gamma_\infty$, and $W_R(a,b^{**})$ right of $x_0$ implies $W_R(a,b^{**})$ right of $\gamma_\infty$.
        By \Cref{claim:diagram:e_infty}, both $W_R(d)$ and $W_R(b^{**})$ are right of $\gamma_\infty$.
        
        First, suppose that $W'$ intersects $W_R(b^{**})$ and say that the lowest element of this intersection is $w$.
        By \Cref{obs:curves-region}, the region of $q[W_R(a,b^{**})]m^{**}[W_R(b^{**})]w[W']q$ contains $\gamma_\infty$, which is a contradiction.
        Similarly, suppose that $W_R(a,b^{**})$ intersects $W_R(d)$ and say that the lowest element of this intersection is $w$.
        By \Cref{obs:curves-region}, the region of $q[W_R(a,b^{**})]w[W_R(d)]m[W']q$ contains $\gamma_\infty$, which is a contradiction.

        Thus, we may assume that $W'$ and $W_R(b^{**})$ are disjoint, and $W_R(a,b^{**})$ and $W_R(d)$ are disjoint.
        Consider a tuple $\calT = (W',W_R(d),W_R(b^{**}),W_R(a,b^{**}))$.
        Each pair of consecutive (cyclically) curves in the tuple shares an endpoint and is consistent, either by definition or by \Cref{claim:diagram:extremal_paths_are_consistent}.
        It follows that $\calT$ is a valid system of curves.
        Since $\gamma_\infty$ is right of $W'$ and left of all the other curves in $\calT$, the interior points of $\gamma_\infty$ are contained in the region of $\calT$.
        This contradicts $\gamma_\infty$ being in the exterior face.
    \end{proof}

Let \defin{\(\preccurlyeq_L^{\infty}\)} be a relation on $B$ defined as follows: 
for all $b_1,b_2 \in B$, 
$b_1 \preccurlyeq_L^{\infty} b_2$ if 
($b_1\in B_R$ and $b_2\in B_L$) or
($b_1,b_2 \in B_R$ and $b_1\preceq_L b_2$) or
($b_1,b_2 \in B_L$ and $b_1\preceq_L b_2$).
Note that \(\preccurlyeq_L^{\infty}\) partially orders $B$.
Moreover, for all $b_1,b_2 \in B$ if $b_1$ and $b_2$ are comparable in $\preccurlyeq_L$, then they are comparable in $\preccurlyeq_L^\infty$.
In particular, by \cref{claim:diagram:incomparable_b_are_comparable},
for all $b_1,b_2\in B$ with $b_1\parallel b_2$ in $P$, $b_1,b_2$ are comparable in \(\preccurlyeq_L^{\infty}\). 
Symmetrically, 
let \defin{\(\preccurlyeq_R^{\infty}\)} be a relation on $B$ defined as follows: 
for all $b_1,b_2 \in B$, 
$b_1 \preccurlyeq_R^{\infty} b_2$ if 
($b_1\in B_R$ and $b_2\in B_L$) or
($b_1,b_2 \in B_R$ and $b_1\preceq_R b_2$) or
($b_1,b_2 \in B_L$ and $b_1\preceq_R b_2$).
Again, 
note that \(\preccurlyeq_R^{\infty}\) partially orders $B$ and 
for all $b_1,b_2\in B$ with $b_1\parallel b_2$ in $P$, $b_1,b_2$ are comparable in \(\preccurlyeq_R^{\infty}\). 
% For each \(X \in \{L, R\}\), let \(\preccurlyeq_X^{\infty}\) denote the partial order on \(B\)
% where for any \(b_1, b_2 \in B\) we have \(b_1 \preccurlyeq_X^{\infty} b_2\) when
% one of the following holds:
% \begin{itemize}
% \item \(b_1 \preccurlyeq_X b_2\) and either
% \(W_X(b_2)\) is left of \(\gamma_\infty\) or \(W_X(b_1)\) is right of \(\gamma_\infty\), or
% \item \(W_X(b_1)\) is right of \(\gamma_\infty\) and \(W_X(b_2)\) is left of \(\gamma_\infty\).
% \end{itemize}
% By \cref{clm:e-infty}, \(\preccurlyeq_X^{\infty}\) is indeed a partial order, and it is
% obtained by ``shifting cyclically'' \(\preccurlyeq_X\).
% It is easy to see that if \(b_1, b_2 \in B\) are incomparable in $P$, then they are
% comparable in each \(\preccurlyeq_R^{\infty}\). 
% Therefore, by \cref{prop:reversible-from-order} applied to \(\preccurlyeq_R^{\infty}\) and \(\succcurlyeq_L^{\infty}\), the sets
% \[
% \ext(J, \preccurlyeq_R^{\infty}) = \{(a, b) \in J: \textrm{there does not exist \(b' \in B \cap U_Q[a]\)
% such that \(b' \prec_R^{\infty} b\)}\},
% \]
% and
% \[
% \ext(J, \succcurlyeq_L^{\infty}) = \{(a, b) \in J: \textrm{there does not exist \(b' \in B \cap U_Q[a]\)
% such that \(b \prec_L^{\infty} b''\)}\}
% \]
% are reversible.
Let
    \begin{align*}
    \defmath{J_5} &= \set{(a,b) \in I_2 : \text{there is no $b' \in B$ such that $a \leq b'$ in $P$ and $b' \prec_L^\infty b$}},\\
    \defmath{J_6} &= \set{(a,b) \in I_2 : \text{there is no $b'' \in B$ such that $a \leq b''$ in $P$ and $b \prec_R^\infty b''$}},\\
    \defmath{J_7} &= \set{(a,b) \in I_2 \setminus (J_5 \cup J_6) : \text{$a$ is a bottom element and } b \in B_L},\\
    \defmath{J_8} &= \set{(a,b) \in I_2 \setminus (J_5 \cup J_6) : \text{$a$ is a bottom element and } b \in B_R}.
    \end{align*}
All the sets $J_5$, $J_6$, $J_7$, and $J_8$ can be computed in polynomial time.
% It follows that for every $(a,b) \in I_2 \setminus (J_5 \cup J_6)$, we have
%     \begin{enumerateNumI}
%         \setcounter{enumi}{9}
%         \item there is \(b' \in B\) with $a \leq b'$ in $P$ and \(b' \prec_{L}^\infty b\),  \label{items:diagram_pairs:b'_infty}
%         \item there is \(b'' \in B\) with $a \leq b''$ in $P$ and \(b \prec_{R}^\infty b''\).\label{items:diagram_pairs:b''_infty}
%     \end{enumerateNumI}
% Denote
% \[J_0 := \bigcup\{\ext(J, \preceq) : {\preceq} \in \{\preccurlyeq_y, \succcurlyeq_y, \preccurlyeq_R, \succcurlyeq_L, \preccurlyeq_R^\infty, \succcurlyeq_L^\infty\}\},
% \]
% so that \(\dim_Q(J_0) \le 6\).
\begin{proposition}\label{claim:diagram:dim_I_7_leq_2}
    $J_5$, $J_6$, $J_7$, and $J_8$ are reversible in $P$.
\end{proposition}
\begin{proof}
    The sets $J_5$ and $J_6$ are reversible by \cref{prop:reversible-from-order}.
    We will prove that $J_8$ is reversible. The argument that $J_7$ is reversible is symmetric.
    Let $((a_1,b_1),\dots,(a_k,b_k))$ be an alternating cycle in $P$ contained in $J_8$.
    Without loss of generality assume that $a_1$ is the highest in the plane among all $\{a_i : i \in [k]\}$.

    Since $(a_1,b_1) \in J_8$, we have $b_1 \in B_R$. 
    Since $(a_1,b_1) \notin J_5$, there exists $b' \in B$ with 
    $a_1\leq b'$ in $P$ and $b' \prec_L^{\infty} b_1$.
    Let $W$ be a witnessing path from $a_1$ to $b'$.
    Since $b_1\in B_R$ and $b' \prec_L^{\infty} b_1$, we have $b'\in B_R$.
    Therefore, by \cref{claim:diagram:bottom_as}, $W$ is right of $x_0$.
    Since $a_1$ is a bottom element, there is $b^* \in B$ and a witnessing path $W^*$ from $a_1$ to $b^*$ with $W^*$ left of $x_0$.
    Let $W'$ be a witnessing path from $a_k$ to $b_1$.
    Since $(a_k,b_k) \in J_8$, $b_1 \in B_R$, and $a_k \leq b_1$ in $P$, by \cref{claim:diagram:bottom_as}, $W'$ is right of $x_0$.

    Let $\ell$ be the horizontal line containing $x_0$ and let $\ell'$ be the horizontal line containing $b'$.
    Let $p^*$ be the common point of $\ell$ and $W^*$.
    Let $\calR$ be the unbounded region containing $x_0$ enclosed by the ray emanating left from $p^*$ (contained in $\ell$), $p^* [W^*] a_1 [W] b'$, and the ray emanating left from $b'$ (contained in $\ell'$).
    Since $a_k$ is lower in the plane than $a_1$, $a_k$ lies in the exterior of $\calR$.
    Next, we prove that $b_1$ is in the interior of $\calR$.
    Since $x_0$ is right of $W^*$ and $x_0$ is left of $W$, $x_0$ is in the interior of $\calR$.
    Consider $W_L(b_1)$.
    It is a vertically monotone curve, so the only way to escape from $\calR$ is through $a_1 [W] b'$ or the ray emanating left from $b'$.
    The former is impossible as it would imply $a_1 \leq b_1$ in $P$.
    Suppose to the contrary that $W_L(b_1)$ intersects the ray emanating left from $b'$.
    In particular, $b'$ is right of $W_L(b_1)$, which implies that $W_L(b')$ is not left of $W_L(b_1)$, which contradicts $b' \prec_L b_1$.
    This shows that $b_1$ lies in the interior of $\calR$.
    The path $W'$ connects $a_k$ in the exterior of $\calR$ and $b_1$ in the interior of $\calR$.
    Therefore, $W'$ must intersect the boundary of $\calR$.
    Moreover, since $W'$ is a vertically monotone curve, it must intersect either the ray emanating left from $p^*$ or $p^* [W^*] a_1 [W] b'$.
    The former case implies that $W'$ is left of $x_0$, which is false, and the latter case implies that $a_1 \leq b_1$ in $P$, which is also false.
    This contradiction yields that there is no alternating cycle in $P$ contained in $J_8$, thus, $J_8$ is reversible.    
\end{proof}

For a pair $(a,b) \in \Inc(P)$ we may or may not have that
    \begin{enumerateNumI}
        \setcounter{enumi}{9}
        \item $a$ is not a bottom element.  \label{items:diagram_pairs:a_bottom}
    \end{enumerateNumI}

Let 
\begin{align*}
    \defmath{I_3} &= \{(a,b) \in I_2 : (a,b) \text{ satisfies \ref{items:diagram_pairs:a_bottom}}\}\\
    &=\{(a,b) \in \closure{I} : (a,b) \text{ satisfies \ref{items:diagram_pairs:b'_vertical}, \ref{items:diagram_pairs:b''_vertical}, \ref{items:diagram_pairs:b'}, \ref{items:diagram_pairs:b''}, and \ref{items:diagram_pairs:a_bottom}}\}.
\end{align*}
We also have $\closure{I} \setminus \bigcup_{i=1}^8 J_i \subset I_3$.

\subsection{Left and right pairs}\label{ssec:left-right}
For each $(a,b) \in I_3$, we say that $(a,b)$ is a \defin{left pair} if for every $d \in B$ with $a \leq d$ in $P$, all witnessing paths from $a$ to $d$ are left of $W_L(b)$.
Symmetrically, we say that $(a,b)$ is a \defin{right pair} if for every $d \in B$ with $a \leq d$ in $P$, all witnessing paths from $a$ to $d$ are right of $W_R(b)$.
\begin{align*}
    \defmath{I_L'} &= \{(a,b) \in I_3 : \text{$(a,b)$ is a left pair}\},\\
    \defmath{I_R'} &= \{(a,b) \in I_3 : \text{$(a,b)$ is a right pair}\}.
\end{align*}
We prove that $I_L'$ and $I_R'$ is a partition of $I_3$.
\begin{proposition}\label{claim:diagram:left_right_parititon}
    $I_3 = I_L' \cup I_R'.$
\end{proposition}
\begin{proof}
    Suppose to the contrary that there is a pair $(a,b) \in I_3 \setminus (I_L' \cup I_R')$.
    It follows that there are $d',d'' \in B$ with $a \leq d'$ in $P$, $a \leq d''$ in $P$, and witnessing paths $W'$ from $a$ to $d'$, $W''$ from $a$ to $d''$ such that $W'$ is not right of $W_R(b)$ and $W''$ is not left of $W_L(b)$.
    By \ref{items:diagram_pairs:a-d_paths}, $W'$ is left of $W_L(b)$ and $W''$ is right of $W_R(b)$.
    Suppose that $x_0 \prec_\uparrow a$, and consider the horizontal line $\ell$ containing $a$. 
    Since $a\prec_{\uparrow} b$ (by~\ref{items:diagram_pairs:a_lower_b}), $\ell$ intersects both $\WL(b)$ and $\WR(b)$. 
    The order of intersection of paths with $\ell$ starting with the leftmost is $W',W_L(b),W_R(b),W''$.
    Since $W'$ and $W''$ coincide on $\ell$ in $a$, all four paths coincide, which is a clear contradiction.
    Therefore, $a \prec_\uparrow x_0$.
    Since $W'$ is left of $W_L(b)$ and $W''$ is right of $W_R(b)$, we have $W'$ left of $x_0$ and $W''$ right of $x_0$, which yields that $a$ is a bottom element.
    This contradicts that $(a,b) \in I_3$ by~\ref{items:diagram_pairs:a_bottom}.
\end{proof}

For every pair $(a,b) \in \Inc(P)$, we define two sets:
\begin{align*}
    \defmath{B'(a,b)} &= \{ b' \in B : \text{$a \leq b'$ in $P$, $b' \prec_L b$, and $b' \prec_R b$}\}, \\
    \defmath{B''(a,b)} &= \{ b'' \in B : \text{$a \leq b''$ in $P$, $b \prec_L b''$, and $b \prec_R b''$}\}.
\end{align*}
Items \ref{items:diagram_pairs:b'} and \ref{items:diagram_pairs:b''} assure that for every $(a,b) \in I_3$, both sets $B'(a,b)$ and $B''(a,b)$ are nonempty.
We define two further sets of incomparable pairs
\begin{align*}
    \defmath{J_L} &= \{(a,b) \in I_L' : \text{for every $b' \in B'(a,b)$, $W_L(a,b')$ is right of $W_L(b')$\}},\\
    \defmath{J_R} &= \{(a,b) \in I_R' : \text{for every $b'' \in B''(a,b)$, $W_R(a,b'')$ is left of $W_R(b'')$\}}.
\end{align*}
We will prove that they are reversible.

% \begin{proposition}
%   For every \((a, b) \in J_R\), and every \(b'' \in B''(a, b)\) we have
%   not only \(b \prec_L b''\), but also \(b \prec_R b''\).
% \end{proposition}
% \begin{proof}
%   Let \(u\) denote the greatest common element of \(\WL(b)\) and \(\WL(b'')\).
%   Therefore, the subpaths \(u\WL(b)b\) and \(u\WL(b'')b''\) are non-trivial and
%   only intersect in \(u\).
%   Since \(b \prec_L b''\), \(u\WL(b)b\) is left of \(u\WL(b'')b''\).
%   In particular, \(b''\) is not right of \(\WL(b)\).
%   Therefore, by \cref{clm:escape-pairs}, the element \(b''\) is not left of \(\WR(b)\).
%   Hence, the subpath \(u\WL(b'')b''\) must intersect \(\WR(b)\).
%   Let \(v\) denote the greatest common element of \(\WR(b)\) and \(u\WL(b'')b''\).
%   Since \(b \prec_L b''\), the subpath \(v\WR(b)b\) is non-trivial, and
%   Since \(b''\) is not left of \(\WR(b)\), the subpath \(v\WR(b')b'\) is non-trivial,
%   and \(v\WR(b')b'\) is right of \(\WR(b)\).
%   Therefore, the witnessing \(x_0\)--\(b''\) path \(x_0 \WR(b) v \WL(b'') b''\)
%   is right of \(\WR(b)\), is strictly right of some point on \(\WR(b)\),
%   and contains a point strictly right of \(\WR(b)\)
%   This implies that the rightmost witnessing \(x_0\)--\(b''\) path \(\WR(b'')\)
%   is right of \(\WR(b)\), is strictly right of some point ob \(\WR(b)\),
%   and contains a point strictly right of \(\WR(b)\), which means that
%   \(b \prec_R b''\).
% \end{proof}

\begin{proposition}\label{prop:a-higher-x0}
    For every $(a,b) \in J_L \cup J_R$, we have $x_0 \prec_{\uparrow} a$.
    In particular, if $(a,b) \in J_R$, then $a$ is right of $W_R(b)$; and if $(a,b) \in J_L$, then $a$ is left of $W_L(b)$.
\end{proposition}
\begin{proof}
    We prove the statement for $(a,b) \in J_R$, the proof for $(a,b) \in J_L$ is symmetric.
    Let $b'' \in B''(a,b)$.
    Let $q = \gcpe(W_R(b),W_R(b''))$ and let $\ell$ be the horizontal line containing $q$.
    Clearly, $x_0 \preccurlyeq_{\uparrow} q$, hence, it suffices to prove that $q \prec_{\uparrow} a$, so we suppose to the contrary that $a \preccurlyeq_{\uparrow} q$.
    Let $p$ be the point of $W_R(a,b'')$ in $\ell$.
    Since $a$ and $b$ are incomparable in $P$, $p \neq q$.
    Since $(a,b)$ is a right pair, $p$ is right of $W_R(b)$, and so, $p$ is right of $q$.
    However, this implies that $W_R(a,b'')$ is not left of $W_R(b'')$, which is a contradiction with $(a,b) \in J_R$.

    By \ref{items:diagram_pairs:a_lower_b}, we also have $a \prec_{\uparrow} b$.
    This and the fact that $(a,b)$ is a right pair implies that $a$ is right of $W_R(b)$.
    This completes the proof.
\end{proof}

\begin{proposition}\label{clm:IRempty}
  \(J_L\) and $J_R$ are reversible in $P$. 
\end{proposition}
\begin{proof}
    The proofs for both sets are symmetric, thus, we show the proof only for $J_R$.
    Suppose
to the contrary that there is a strict alternating cycle \(((a_1, b_1), \ldots, (a_k, b_k))\) in $P$ contained in \(J_R\). Without loss of generality, assume that \(b_1\) is a maximal element of \(\preccurlyeq_\uparrow\) among all \(b_1, \ldots, b_k\). Let \(b' \in B'(a_1, b_1)\). We claim that $b_1 \prec_\uparrow b'$.
By \Cref{prop:a-higher-x0}, $a_1$ is right of $W_R(b_1)$.
The paths $W_R(b_1)$ and $W_R(a_1,b')$ are disjoint, hence, by \Cref{obs:where-curves-end}, $b'$ is either right of $W_R(b_1)$ or $b_1 \prec_\uparrow b'$.
The former is false as $b' \in B'(a_1,b_1)$, hence, $b_1 \prec_\uparrow b'$, as desired.
Next, we claim that \(b_2\) is right of \(W_R(b_1)\).
The paths $W_R(b_1)$ and $W_R(a_1,b_2)$ are disjoint, so $b_2$ is either right of $W_R(b_1)$ or $b_1 \prec_\uparrow b_2$. The latter is false, hence, $b_2$ is right of \(W_R(b_1)\), as desired.
Thus,  \(b_1 \prec_R b_2\).

Since \((a_1, b_1) \in J_R\), \(W_R(a_1,b_2)\) is left of \(W_R(b_2)\), and in particular $a_1$ is left of $W_R(b_2)$ by \cref{prop:a-higher-x0}.
Let \(S\) denote the set of all elements \(x\) such that either
\begin{enumerate}
  \item \(x_0 \preccurlyeq_\uparrow x \prec_\uparrow b_2\) and \(x\) is not left of \(W_R(b_2)\), or
  \item \(b_2 \prec_\uparrow x \preccurlyeq_\uparrow b_1\) and \(x\) is not left of \(W_R(a_1, b')\).
\end{enumerate}
See~\Cref{fig:IRempty}.
\begin{figure}
    \centering
    \includegraphics{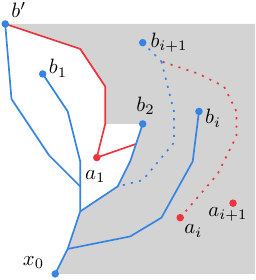}
    \caption{Proof that \(J_R\) is reversible. \(S\) is the set of elements in the gray area. For \(i \in [k] \setminus \{1\}\), we show that if \(b_2 \preccurlyeq_R b_i\) then \(a_i \in S\), if \(a_i\in S\) then \(b_{i+1} \in S\), and if \(b_{i+1} \in S\), then \(b_2 \prec_R b_{i+1}\).}
    \label{fig:IRempty}
\end{figure}
We prove inductively that for each \(i \in [k] \setminus \{1\}\), we have \(a_i \in S\) and \(b_2 \prec_R b_{i+1}\) (here $b_{k+1} = b_1$). 
This will imply \(b_1 \prec_R b_2 \prec_R b_{k+1} = b_1\), which is a contradiction.
It suffices to show two implications for \(i \in [k] \setminus \{1\}\): (1) if \(b_2 \preccurlyeq_R b_{i}\), then \(a_i \in S\), and  (2) if \(a_i \in S\) then \(b_2 \prec_R b_{i+1}\).
%\later{nit: Make (1) a clickable reference (inline enumerate). Maybe unhardcode all enumerates.}

For the proof of (1), observe that by \Cref{prop:a-higher-x0}, \(a_i\) is right of \(W_R(b_i)\).
If \(i = 2\), then \(a_i \in S\) as required.
Otherwise, we have \(3 \le i \le k\) and \(b_2 \prec_R b_{i}\).
Suppose first that \(a_i \preccurlyeq_{\uparrow} b_2\).
Since \(b_2 \prec_R b_{i}\) and \(a_i\) is right of \(W_R(b_i)\), \(a_i\) is right of \(W_R(b_2)\), and thus \(a_i \in S\).
Now suppose that \(b_2 \preccurlyeq_{\uparrow} a_i\) (\(\preccurlyeq_{\uparrow} b_i \preccurlyeq_{\uparrow} b_1\)).
Then, since \(b_2 \prec_R b_{i}\), \(b_2\) is left of \(W_R(b_i)\).
The path \(W_R(b_i)\) is disjoint from \(W_R(a_1, b')\) as otherwise, the union \(W_R(b_i) \cup W_R(a_1, b')\) would contain a witnessing path from \(x_0\) to \(b'\) which is right of \(W_R(b_2)\), contrary to the assumption \(b' \prec_R b_1 \prec_R b_2\).
Consider the horizontal line containing $b_2$.
Its intersection with $W_R(b_i)$ is right of $b_2$, and so, right of the intersection of the line with $W_R(a_1,b')$.
Thus, by \Cref{obs:where-curves-end}, the highest point of $W_R(b_i)$, i.e.\ $b_i$, is either right of $W_R(a_1,b')$ or $b' \prec_\uparrow b_i$.
The latter is false. 
The former implies that  \(W_R(b_i)\) is right of \(W_R(a_1, b')\). 
Since $b_2\prec_\uparrow a_i \prec_\uparrow b_i$, 
we get that \(a_i\) is right of \(W_R(a_1, b')\), so $a_i \in S$.

For the proof of (2), suppose that \(a_i \in S\).
Since the alternating cycle is strict, the witnessing path \(W_R(a_i, b_{i+1})\) is disjoint from \(W_R(a_1, b')\) and from \(W_R(b_2)\). 
By \Cref{obs:where-curves-end}, 
either $b_{i+1}$ is right of $W_R(b_2)$  or 
$b_2$ is left of $W_R(a_i,b_{i+1})$. 
In the former case, we obtain $b_2\prec_R b_{i+1}$ as desired. 
Thus, we assume that $b_2$ is left of $W_R(a_i,b_{i+1})$. 
By \Cref{obs:where-curves-end} again 
either $b_{i+1}$ is right of $W_R(a_1,b')$ or 
$b'$ is left of $W_R(b_{i+1})$. 
The latter case contradicts $b_{i+1} \preceq_\uparrow b_1 \prec_\uparrow b'$. 
Thus, we have that $b_{i+1}$ is right of $W_R(a_1,b')$.
Note that if \(b_2\) is left of \(W_R(b_{i+1})\), then \(b_2 \prec_R b_{i+1}\) as required, so suppose that \(b_2\) is right of \(W_R(b_{i+1})\).
Since the path \(W_R(b_{i+1})\) is disjoint from \(W_R(a_1, b')\) and \(b_{i+1}\) is right of \(W_R(a_1, b')\), \Cref{obs:where-curves-end} implies that \(a_1\) is left of \(W_R(b_{i+1})\).
This combined with the fact that \(b_2\) is right of \(W_R(b_{i+1})\) implies that the paths \(W_R(b_{i+1})\) intersects \(W_R(a_1, b_2)\), so \(a_1 \le b_{i+1}\) in $P$ contrary to strictness of the cycle.
\end{proof}

For each $(a,b) \in \Inc(P)$, we define
\begin{align*}
    \defmath{B_*'(a,b)} &= \{b' \in B'(a,b) : W_R(a,b') \text{ is left of } W_R(b')\},\\ 
    \defmath{B_*''(a,b)} &= \{ b'' \in B''(a,b) : \text{$W_R(a,b'')$ is right of $W_R(b'')$}\}.
    %&= \{ b'' \in B : \text{$a \leq b''$ in $P$, $b \prec_L b''$, $b \prec_R b''$, and $W_R(a,b'')$ is right of $W_R(b'')$}\}.
\end{align*}
Let
\begin{align*}
    \defmath{I_L}&= \{ (a,b) \in I_L' : B_*'(a,b) \neq \emptyset\},\\
    \defmath{I_R} &= \{ (a,b) \in I_R' : B_*''(a,b) \neq \emptyset\}.
\end{align*}
Note that $I_L' \setminus J_L = I_L$ and $I_R'\setminus J_R = I_R$.
In particular, we have
\begin{align*}
    \dim_P(I) \leq \dim_P(\closure{I}) &\leq \dim_P(I_3) + \dim_P(J_1 \cup \dots \cup J_8) \leq \dim_P(I_3) + 8 \\
    &\leq \dim_P(I_L') + \dim_P(I_R') + 8 \leq \dim_P(I_L) + \dim_P(I_R) + 10.
\end{align*}

Next, the algorithm processes $I_L$ and $I_R$ in parallel.
We will show the procedure for $I_R$ as the one for $I_L$ is symmetric.
More formally, we may take the mirror image of the diagram of $P$ along a vertical line, then the notions of left pairs and right pairs are swapped.

For convenience, we repeat the definition of $I_R$.
The following properties may or may not be satisfied for each $(a,b) \in \Inc(P)$,
    \begin{enumerateNumI}
        \setcounter{enumi}{10}
        \item $(a,b)$ is a right pair,\label{items:diagram_pairs:right}
        \item $B_*''(a,b) \neq \emptyset$.\label{items:diagram_pairs:B''_*} \label{I-last}
    \end{enumerateNumI}
We have
\begin{align*}
    I_R &= \{(a,b) \in I_3 : (a,b) \text{ satisfies \ref{items:diagram_pairs:right} and \ref{items:diagram_pairs:B''_*}}\}\\
    &=\{(a,b) \in \closure{I} : (a,b) \text{ satisfies \ref{items:diagram_pairs:b'_vertical}, \ref{items:diagram_pairs:b''_vertical}, \ref{items:diagram_pairs:b'}, \ref{items:diagram_pairs:b''}, \ref{items:diagram_pairs:a_bottom}, \ref{items:diagram_pairs:right}, and \ref{items:diagram_pairs:B''_*}}\}.
\end{align*}
Note that every $(a,b) \in I_R$ also satisfies \ref{items:diagram_pairs:a_lower_b}--\ref{items:diagram_pairs:a-d_paths} and \ref{items:diagram_pairs:a_in_A}.

\subsection{Regions}\label{ssec:regions}

For all $a \in A$ and $b',b'' \in B$, we list several conditions that may or may not be satisfied:
\begin{enumerateNumr}
    \item $b' \prec_R b''$, \label{item-regions-b'-r-b''} \label{item-regions-first}
    \item $a \leq b'$ and $a \leq b''$ in $P$, \label{item-regions-leq}
    \item $a$ is not a bottom element, \label{item-regions-not-bottom}
    \item $W_R(b'')$ is left of $W_R(a,b'')$, \label{item-regions-W-left-W}
    \item $\gcpe(W_R(b'),W_R(b''))$ and $a$ are incomparable in $P$. \label{item-regions-inc}\label{item-regions-last}
\end{enumerateNumr}

\begin{proposition}\label{prop:tuple-is-valid}
    Let $a \in A$ and $b',b'' \in B$ be such that \ref{item-regions-first}--\ref{item-regions-last} hold.
    Let
    \[\gamma_1 = W_R(b'), \ \gamma_2 = W_R(b''), \ \gamma_3 = W_R(a,b''), \ \gamma_4 = W_R(a,b'),\]
    \[q_1 = \gcpe(\gamma_1, \gamma_2), \ q_3 = \gcpe(\gamma_3,\gamma_4), \ m_2 = \lcse(\gamma_2, \gamma_3), \ m_4 = \lcse(\gamma_1, \gamma_4),\]
    \[\gamma = q_1[\gamma_2]m_2[\gamma_3]q_3[\gamma_4]m_4[\gamma_1]q_1.\]
    Then
    \begin{enumerateNump}
    \item $q_1 < m_4$, $q_1<m_2$, $q_3<m_2$, $q_3<m_4$ in $P$, \label{prop:tuple-is-valid:qs-ms}
    \item $\gamma_i$ is left of $\gamma_j$ for all $i,j \in [4]$ with $i < j$, \label{prop:tuple-is-valid:left}
    \item $m_2\prec_\uparrow m_4$, \label{prop:tuple-is-valid:m2-m4}
    \item $\gamma_1 \cap \gamma_3 = \emptyset$ and $\gamma_2 \cap \gamma_4 = \emptyset$, \label{prop:tuple-is-valid:disjointness}
    \item $\gamma$ is a simple closed curve,\label{prop:tuple-is-valid:simple-closed}
    \item all elements of $P$ in $a[W_R(a,b'')]m_2$ and $a[W_R(a,b')]m_4$ except $m_2$ and $m_4$ are in $A$. \label{prop:tuple-is-valid:inA}
    \end{enumerateNump}
\end{proposition}
\begin{proof}
    First,  we prove \ref{prop:tuple-is-valid:qs-ms}.
    Clearly, each pair of elements in the required inequalities is comparable in $P$ as they lie in one of the witnessing paths $\gamma_i$ for some $i \in [4]$.
    Note also that $m_4 \not\leq q_1$ and $m_2 \not\leq q_1$ in $P$, as the opposite gives $a \leq q_1$ in $P$, which is false by \ref{item-regions-inc}.
    This implies $q_1 < m_4$ and $q_1 < m_2$ in $P$.
    Next, suppose to the contrary that $m_2 \leq q_3$ in $P$.
    Consider the path $W = x_0[W_R(b'')]m_2[W_R(a,b'')]q_3[W_R(a,b')]b'$. 
    Since  $q_1=\gcpe(W_R(b'),W_R(b''))$ and $q_1<m_2$ in $P$, by \ref{item-regions-b'-r-b''}, we have that 
    $W_R(b')$ is left of $x_0[W_R(b'')]m_2$. 
    Therefore, $W_R(b')$ is not right of $W$, which is a contradiction with \cref{claim:diagram:extreme_paths_exist}.    
    % It contains the element $m_2$, hence $m_2 \prec_\uparrow b'$.
    % Also, $q_1 < m_2$ in $P$, thus, $m_2$ does not lie in $W_R(b')$ by the bottom-consitency of $W_R(b')$ and $W_R(b'')$ (\cref{claim:diagram:extremal_paths_are_consistent}).
    % On the other hand, $m_2$ lies in $W_R(b'')$, thus, by \ref{item-regions-b'-r-b''}, $m_2$ is right of $W_R(b')$.
    % It follows that $W$ is a witnessing path from $x_0$ to $b'$ that is not left or equal to $W_R(b')$, which contradicts \cref{claim:diagram:extreme_paths_exist}.
    The obtained contradiction gives $q_3 < m_2$ in $P$.
    Finally, suppose to the contrary that $m_4 \leq q_3$ in $P$.
    Since $m_4$ and $q_3$ lie in $W_R(a,b')$, $q_3$ lie in $W_R(a,b'')$, and the two paths $W_R(a,b')$, $W_R(a,b'')$ are bottom-consistent, 
    we get that $m_4$ lies in $W_R(a,b'')$.
    Consider the horizontal line $\ell$ containing $m_4$.
    Note that we have $q_1 < m_4 \leq q_3 < m_2$ in $P$.
    Since $q_1$ and $m_2$ are both elements of $W_R(b'')$, $W_R(b'')$ intersects $\ell$.
    Also, $m_4$ lies in $W_R(b')$ and $q_1 < m_4$ in $P$.
    Since $W_R(b')$ and $W_R(b'')$ are bottom-consistent and by \ref{item-regions-W-left-W}, $m_4$ is left of $W_R(b'')$.
    However, this yields a point in $W_R(a,b'')$ left of $W_R(b'')$, which contradicts \ref{item-regions-W-left-W}.
    The obtained contradiction gives $q_3 < m_4$ in $P$ and completes the proof of \ref{prop:tuple-is-valid:qs-ms}.

    By \ref{item-regions-b'-r-b''} and \ref{item-regions-W-left-W}, $\gamma_1$ is left of $\gamma_2$ and $\gamma_2$ is left of $\gamma_3$.
    We argue now that $\gamma_3$ is left of $\gamma_4$. 
    %Later, to complete the proof of \ref{prop:tuple-is-valid:left}, we need \ref{prop:tuple-is-valid:m2-m4} that we prove later.
    The first step is to find a horizontal line $\ell_1$ that intersects both $W_R(b'')$ and $W_R(a,b')$, and the intersection with $W_R(b'')$ is left of 
    the intersection with $W_R(a,b')$.
    If $a \prec_\uparrow x_0$, then we set $\ell_1$ to be the horizontal line containing $x_0$.
    By \ref{item-regions-W-left-W}, $W_R(a,b'')$ is right of $x_0$. 
    This implies that $W_R(a,b')$ is also right of $x_0$ as otherwise $a$ would be a bottom element which contradicts \ref{item-regions-not-bottom}. 
    If $x_0 \preccurlyeq_\uparrow a$, then we set $\ell_1$ to be the horizontal line containing $a$.
    Here, $a$ is right of $W_R(b'')$ by \ref{item-regions-W-left-W}. 
    This completes the construction of $\ell_1$. 
    Let $d$ be one of $b'$ and $b''$ such that $d \preccurlyeq_\uparrow b'$ and $d\preccurlyeq_\uparrow b''$ and let $\ell_2$ be the horizontal line containing $d$.
    Note that both $W_R(a,b')$ and $W_R(b'')$ intersect $\ell_2$, say in the points $v$ and $w$, respectively.
    If $d = b'$, then $v = b'$ and by \ref{item-regions-b'-r-b''}, $v$ is left of $w$.
    If $d = b''$, then $w = b''$.
    Assume that $v$ is left of $w$ or $v = w$.
    Since the intersections of $W_R(a,b')$ and $W_R(b'')$ are in different order on $\ell_1$ and $\ell_2$, by Darboux property, these paths intersect, say in an element $y$.
    If $y$ lies in $W_R(b')$, then since $y$ is a common element of $W_R(b')$ and $W_R(b'')$, $y \leq q_1$ in $P$, which implies $a \leq q_1$ in $P$ contradicting \ref{item-regions-inc}.
    If $y$ does not lie in $W_R(b')$, then $y$ is right of $W_R(b')$ by \ref{item-regions-b'-r-b''}.
    Thus, the path $x_0[W_R(b'')]y[W_R(a,b')]b'$ is a witnessing path from $x_0$ to $b'$ in $P$ and contains an element right of $W_R(b')$, which contradicts \cref{claim:diagram:extreme_paths_exist}.
    Summarizing, the case where $v$ is left of $w$ or $v = w$ leads to a contradiction.
    Next, we assume that $w$ is left of $v$.
    Note that in this case $d = w = b''$ as discussed before.
    Since $w$ is left of $W_R(a,b')$, by \cref{prop:diagram:for_consistent_one_distinction_is_enough}, $W_R(b'')$ is left of $W_R(a,b')$. 
    This completes the proof that $\gamma_3$ is left of $\gamma_4$.

    Next, we prove \ref{prop:tuple-is-valid:m2-m4}.
    Suppose to the contrary that $m_4 \preccurlyeq_\uparrow m_2$.
    It follows that all four curves $\gamma_1$, $\gamma_2$, $\gamma_3$, and $\gamma_4$ intersect the horizontal line containing $m_4$.
    Moreover, since $\gamma_1$ is left of $\gamma_2$, $\gamma_2$ is left of $\gamma_3$, $\gamma_3$ is left of $\gamma_4$, and $m_4$ is an element of both $\gamma_1$ and $\gamma_4$, we conclude that $m_4$ is an element of all the four curves.
    In particular, $m_4$ is a common element of $W_R(b')$ and $W_R(b'')$, which implies that $m_4 \leq q_1$ in $P$.
    On the other hand, $a \leq m_4$ in $P$ as $m_4$ is an element of $\gamma_4$.
    This is a contradiction with \ref{item-regions-inc} that completes the proof of \ref{prop:tuple-is-valid:m2-m4}.

    With item \ref{prop:tuple-is-valid:m2-m4} proven, we can complete the proof of \ref{prop:tuple-is-valid:left}.
    Since $m_4$ is an element of $\gamma_1$ and $\gamma_4$, 
    since $m_2\prec_\uparrow m_4$ (by~\ref{prop:tuple-is-valid:m2-m4}) and since $a \prec_\uparrow m_2$ and $x_0 \prec_\uparrow m_2$, we conclude that all four curves intersect the horizontal line $\ell$ containing $m_2$.
    Let $p_1$ and $p_4$ be the points of $\gamma_1$ and $\gamma_4$ in $\ell$, respectively.
    By the definitions of $q_1$, $q_3$, and $m_4$, by \ref{prop:tuple-is-valid:qs-ms}, by the consistency of the respective curves, and since $\gamma_1$ is left of $\gamma_2$, $\gamma_2$ is left of $\gamma_3$, and $\gamma_3$ is left of $\gamma_4$, we obtain that $p_1$ is left of $m_2$ and $m_2$ is left of $p_4$.
    By \cref{prop:diagram:for_consistent_one_distinction_is_enough}, this implies that $\gamma_1$ is left of $\gamma_3$, $\gamma_2$ is left of $\gamma_4$, and $\gamma_1$ is left of $\gamma_4$, which completes the proof of \ref{prop:tuple-is-valid:left}.

    Next, we prove \ref{prop:tuple-is-valid:disjointness} starting with $\gamma_1 \cap \gamma_3 = \emptyset$.
    Let $p$ be any point of $\gamma_3$ and let $\ell$ be the horizontal line containing $p$.
    We show that either there is no point of $\gamma_1$ in $\ell$, or the point of $\gamma_1$ in $\ell$ is left of $p$.
    We distinguish cases depending on the vertical coordinate of $p$.
    If $p \prec_\uparrow x_0$ or $b' \prec_\uparrow p$, then there is no point of $\gamma_1 = W_R(b')$ in $\ell$.
    Note that $m_2 \prec_\uparrow m_4 \preccurlyeq_\uparrow b'$.
    Suppose that $x_0 \preccurlyeq_\uparrow p \prec_\uparrow m_2$.
    The curve $\gamma_2$ intersects $\ell$, say in a point $p_2$.
    By the definition of $m_2$, since $\gamma_2$ and $\gamma_3$ are top-consistent and $\gamma_2$ is left of $\gamma_3$, $p_2$ is left of $p$.
    Since $\gamma_1$ is left of $\gamma_2$, the point of $\gamma_1$ in $\ell$ is left of $p_2$, and so, left of $p$.
    Finally, assume that $m_2 \preccurlyeq_\uparrow p \preccurlyeq_\uparrow b'$.
    In this case, $p$ is an element of $\gamma_2$ by the definition of $m_2$.
    Since $m_2 \prec_\uparrow b'$, the curve $\gamma_1 = W_R(b')$ intersects $\ell$, say in $p_1$.
    By \ref{prop:tuple-is-valid:qs-ms}, $q_1 < m_2$ in $P$, hence, by the definition of $q_1$ and since $\gamma_1$ and $\gamma_2$ are bottom-consistent, $p_1$ is left of $p$, as desired.
    This completes the proof that $\gamma_1 \cap \gamma_3 = \emptyset$.

    With a very similar idea, we now prove that $\gamma_2 \cap \gamma_4 = \emptyset$.
    Let $p$ be any point of $\gamma_4$ and let $\ell$ be the horizontal line containing $p$.
    We show that either there is no point of $\gamma_2$ in $\ell$, or the point of $\gamma_2$ in $\ell$ is left of $p$.
    We distinguish cases depending on the vertical coordinate of $p$.
    If $p \prec_\uparrow x_0$ or $b'' \prec_\uparrow p$, then there is no point of $\gamma_2 = W_R(b'')$ in $\ell$.
    Suppose that $x_0 \preccurlyeq_\uparrow p \prec_\uparrow q_3$.
    In this case, $p$ is an element of $\gamma_3$ by the definition of $q_3$.
    Since $q_3 \prec_\uparrow m_2 \preccurlyeq_\uparrow b''$ (by \ref{prop:tuple-is-valid:qs-ms}), the curve $\gamma_2 = W_R(b'')$ intersects $\ell$, say in $p_2$.
    Since $q_3 < m_2$ in $P$, by the definition of $m_2$, and since $\gamma_2$ and $\gamma_3$ are top-consistent, $p_2$ is left of $p$, as desired.
    Finally, assume that $q_3 \preccurlyeq_\uparrow p \preccurlyeq_\uparrow b''$.
    The curve $\gamma_3$ intersects $\ell$, say in a point $p_3$.
    By the definition of $q_3$, since $\gamma_3$ and $\gamma_4$ are bottom-consistent and $\gamma_3$ is left of $\gamma_4$, $p_3$ is left of $p$.
    Since $\gamma_2$ is left of $\gamma_3$, the point of $\gamma_2$ in $\ell$ is left of $p_3$, and so, left of $p$.
    This completes the proof that $\gamma_2 \cap \gamma_4 = \emptyset$, and the proof of \ref{prop:tuple-is-valid:disjointness}.

    Item \ref{prop:tuple-is-valid:simple-closed} follows directly from \ref{prop:tuple-is-valid:disjointness} and the consistency of pairs of respective curves among $\gamma_1$, $\gamma_2$, $\gamma_3$, and $\gamma_4$.
    Item \ref{prop:tuple-is-valid:inA} follows directly from \cref{prop:exposed-paths}.\ref{prop:exposed-paths:right}.
\end{proof}

Let $a \in A$ and $b',b'' \in B$ satisfy \ref{item-regions-first}--\ref{item-regions-last}.
Let $q_1$, $m_2$, $q_3$, $m_4$, and $\gamma$ be as in the statement of \Cref{prop:tuple-is-valid}.
We define the region \defin{$\calR(a,b',b'')$} to be the region of $\gamma$.
Moreover, we say that $(q_1,m_2,q_3,m_4)$ are the \defin{extreme points} of this region.

\begin{proposition}\label{prop:properties-regions}
    Let $a \in A$ and $b',b'' \in B$ satisfy \ref{item-regions-first}--\ref{item-regions-last} and let $\calR = \calR(a,b',b'')$.
    Let $(q_1,m_2,q_3,m_4)$ be the extreme points of $\calR$.
    \begin{enumerate}
        \item For every $p \in \calR\setminus\{m_4\}$, we have $p \prec_\uparrow m_4$. \label{prop:properties-regions:b-lower-m4}
        \item $W_R(a,q_3)\setminus \{q_3\}$ is disjoint from $\calR$. \label{prop:properties-regions:a-outside}
        \item For every $p \in \calR$, $p$ is not left of $W_R(b')$ and $p$ is not right of $W_R(a,b')$. \label{prop:properties-regions:p-inside}
        \item For every witnessing path $W$ from $m_2$ to an element $d$ in $P$, we have
        $m_2[W]d\setminus\set{m_2}\subset \Int\calR$.
        \label{prop:properties-regions:d-over-m2}
        \item For every $d \in B$ with $d \in \calR$, $q_1$ lies in $W_R(d)$. \label{prop:properties-regions:d-inside}
        \item $x_0[W_R(b'')]q_1\setminus\{q_1\}$ is disjoint from $\calR$ and $m_2[W_R(b'')]b'' \setminus\{m_2\} \subset \Int \calR$. \label{prop:properties-regions:b''}
        \item For every $b \in B$ such that $b' \prec_R b \prec_R b''$ and $m_4$ does not lie in $W_R(b)$ in $P$, we have $b \in \Int \calR$, and moreover $q_1$ lies in $W_R(b)$ and $q_1[W_R(b)]b \subset \calR$. \label{prop:properties-regions:b-inside}
        \item For every $d \in B$ such that $d \prec_R b'$, we have $d \notin \calR$. \label{prop:properties-regions:d-outside}
        %\item $b' \notin \calR\setminus \{m_4\}$. \jedrzej{This is never used.} \label{prop:properties-regions:b'}
        \item For all $c \in A$ and $d \in B$ with $c \leq d$, $c \parallel b'$ in $P$, and $d \prec_R b'$, we have $c \notin \calR$. \label{prop:properties-regions:c-outside}
    \end{enumerate}
\end{proposition}
\begin{proof}
    Let
    \[\gamma_1' = q_1[W_R(b')]m_4, \ \gamma_2' = q_1[W_R(b'')] m_2, \ \gamma_3' = q_3[W_R(a,b'')]m_2, \ \gamma_4' = q_3[W_R(a,b')]m_4.\] 

    Since $m_2 \prec_\uparrow m_4$ by \ref{prop:tuple-is-valid:m2-m4}, $m_4$ is the highest point among all the points of $\gamma_1'$, $\gamma_2'$, $\gamma_3'$, and $\gamma_4'$.
    Thus, any point $p$ with $m_4 \prec_\uparrow p$ is neither left nor right of any of $\gamma_i'$ for $i \in [4]$.
    By \Cref{obs:curves-region}, this implies $p \notin \calR$.
    A point $p \neq m_4$ on the horizontal line containing $m_4$ is either only left of both $\gamma_1'$ and $\gamma_4'$ or is only right of both $\gamma_1'$ and $\gamma_4'$, and thus, again $p \notin \calR$ by \Cref{obs:curves-region}.
    This gives \ref{prop:properties-regions:b-lower-m4}.

    Items \ref{prop:properties-regions:a-outside} and \ref{prop:properties-regions:p-inside} follow directly from \ref{prop:tuple-is-valid:left}, \ref{prop:tuple-is-valid:disjointness}, and \cref{obs:curves-region}.

    Next, we prove~\ref{prop:properties-regions:d-over-m2}. 
    Let $d$ be an element in $P$ such that
    $m_2<d$ in $P$. 
    Let $W$ be a witnessing path from $m_2$ to $d$ in $P$. 
    Let $p$ be a point in $W \setminus\{m_2\}$.
    Note that $q_3\prec_\uparrow m_2 \prec_\uparrow p$. 
    We claim that $m_2[W]p$ is disjoint from $W_R(b')$. 
    Indeed, otherwise we have an element $u$ in the intersection and 
    we can construct a witnessing path $x_0[W_R(b'')]m_2[W]u[W_R(b')]b'$ 
    that contains an element (namely $m_2$) that is right of $W_R(b')$, which contradicts \cref{claim:diagram:extreme_paths_exist}. 
    Now we claim that $W$ is disjoint from $q_3[W_R(a,b')]m_4\setminus\set{m_4}$. 
    Indeed, $W$ contains only elements of $B$ while by~\ref{prop:tuple-is-valid:inA} $q_3[W_R(a,b')]m_4\setminus\set{m_4}$ contains only elements of $A$, a contradiction.
    By \cref{prop:tuple-is-valid:left}, $m_2$ is right of $W_R(b')$ and left of $W_R(a,b')$.
    Since $W$ is disjoint from both $W_R(b')$ and $W_R(a,b')$, by \cref{prop:diagram:for_consistent_one_distinction_is_enough}, $W$ is right of $W_R(b')$ and $W$ is left of $W_R(a,b')$.
    By \cref{obs:where-curves-end}, either $m_4 \prec_\uparrow p$ or $p$ is right of $W_R(b')$ and left of $W_R(a,b')$.
    The former case leads to a contradiction by considering the element of $W$ in the horizontal line containing $m_4$ as this point has to be simultaneously left and right of $m_4$.
    Thus, we have $p$ is right of $W_R(b')$ and left of $W_R(a,b')$.
    Finally, consider the horizontal line $\ell$ containing $p$. 
    Since $m_2\prec_\uparrow p \prec_\uparrow m_4$, the line $\ell$ intersects $\partial\calR$ only in $q_1[W_R(b')]m_4$ and $q_3[W_R(a,b')]m_4$.
    It follows that $p \in \Int \calR$ by \cref{obs:curves-region}.

    % Next, we prove \ref{prop:properties-regions:b''}.
    % Clearly, $q_3[W_R(b'')]m_2 \subset \calR$.
    % We claim that $m_2[W_R(b'')]b'' \setminus \{m_2\} \subset \Int \calR$.
    % Note that $m_2 \preccurlyeq_\uparrow b''$.
    % We prove that $b'' \preccurlyeq_\uparrow m_4$.
    % Suppose otherwise.
    % By \ref{item-regions-b'-r-b''} and \ref{prop:tuple-is-valid:m2-m4}, $m_4$ is left of $W_R(b'')$.
    % However, this contradicts \ref{prop:tuple-is-valid:left} stating that $W_R(b'')$ is left of $\gamma_4$.
    % It follows that indeed, $b'' \preccurlyeq_\uparrow m_4$.
    % Consider the horizontal line $\ell$ containing any element $u$ of $m_2[W_R(b'')]b'' \setminus\{m_2\}$.
    % Since $m_2 \prec_\uparrow u \preccurlyeq_\uparrow m_4$, $\ell$ intersects only $\gamma_1'$ and $\gamma_4'$ among the curves $\gamma_1'$, $\gamma_2'$, $\gamma_3'$, and $\gamma_4'$.
    % By \ref{prop:tuple-is-valid:qs-ms}, \ref{prop:tuple-is-valid:left}, \ref{prop:tuple-is-valid:disjointness}, and the consistency of $W_R(b')$ and $W_R(b'')$, $u$ is right of $\gamma_1'$ and left of $\gamma_4'$.
    % Thus, by \cref{obs:curves-region}, it follows that $m_2[W_R(b'')]b'' \subset \calR$, as claimed.

    Item \ref{prop:properties-regions:d-over-m2} implies that $m_2[W_R(b'')]b'' \setminus\{m_2\} \subset \Int \calR$.
    To complete the proof of \ref{prop:properties-regions:b''}, it suffices to verify that $x_0[W_R(b'')]q_1\setminus\{q_1\}$ is disjoint from $\calR$.
    %\later{When showing disjointness / containment of curves, take points not elements!}
    Let $p$ be a point of $x_0[W_R(b'')]q_1\setminus\{q_1\}$.
    If $p \prec_\uparrow q_3$, then $p$ is neither left nor right of any of $\gamma_i'$ for $i \in [4]$, thus, $p \notin \calR$ by \cref{obs:curves-region}.
    If $q_3\preccurlyeq_\uparrow p$, then by \ref{prop:tuple-is-valid:left}, $p$ is left of $\gamma_3'$ and $\gamma_4'$.
    Since $p \prec_\uparrow q_1$, $p$ is neither left nor right of $\gamma_1'$ and $\gamma_2'$.
    Therefore, again by \cref{obs:curves-region}, $p \notin \calR$.
    This completes the proof of \ref{prop:properties-regions:b''}.

    For the proof of \ref{prop:properties-regions:d-inside}, let $d \in B$ with $d \in \calR$.
    By \ref{prop:properties-regions:b''}, $x_0 \notin \Int \calR$.
    Thus, $W_R(d)$ intersects $\partial \calR$.
    By \ref{prop:tuple-is-valid:inA}, $W_R(d)$ intersects $\gamma_1'$ or $\gamma_2'$. 
    In each case, by \cref{claim:diagram:extremal_paths_are_consistent}, we obtain that $q_1$ lies in $W_R(d)$, as desired.

    Next, we prove \ref{prop:properties-regions:b-inside}.
    Let $b \in B$ be such that $b' \prec_R b \prec_R b''$ and $m_4 \parallel b$ in $P$.
    Since $W_R(b')$ is left of $W_R(b)$ which is left of $W_R(b'')$, 
    by \cref{prop:sandwich}, $q_1$ lies in $W_R(b) \setminus \{q_1\}$.
    In particular, $q_1 \prec_\uparrow b$.
    Since $b' \prec_R b \prec_R b''$ and since $b \in B$, $b$ does not lie in the boundary of $\calR$.
    Suppose that $b \prec_\uparrow m_2$.
    By \ref{prop:tuple-is-valid:m2-m4}, we have $m_2\prec_\uparrow b'$ and we also have $m_2 \preccurlyeq_\uparrow b''$.
    In particular, $b$ is right of $\gamma_1'$ and left of $\gamma_2'$.
    By \ref{prop:tuple-is-valid:left}, $b$ is right of exactly one among the curves $\gamma_1'$, $\gamma_2'$, $\gamma_3'$, and $\gamma_4'$.
    Thus, by \cref{obs:curves-region}, $b \in \Int \calR$.

    Hence, we may assume that $m_2 \prec_\uparrow b$.
    %Let $\gamma_4 = W_R(a,b')$.
    Note that $W_R(b)$ does not intersect $W_R(a,b')$.
    Indeed, otherwise, if they intersect in an element $u$, then $x_0[W_R(b)]u[W_R(a,b')]b'$ contains an element right of $W_R(b')$, which contradicts \cref{claim:diagram:extreme_paths_exist}.

    We will either conclude that $b \in \Int \calR$ or we will find two horizontal lines $\ell_1$ and $\ell_2$ such that the intersection of $W_R(b)$ and $W_R(a,b')$ with these lines are in the opposite order.
    By the Darboux property, this will imply that $W_R(b)$ intersects $W_R(a,b')$.
    Since $W_R(b)$ has only elements in $B$, by \cref{prop:exposed-paths}.\ref{prop:exposed-paths:right}, they must intersect in $m_4[W_R(a,b')]b'$, which contradicts $m_4$ not lying in $W_R(b)$.

    Let $\ell_1$ be the horizontal line containing $m_2$. 
    Since $b\prec_R b''$, 
    the intersection point of $W_R(b)$ with $\ell_1$ is left of or equal to $m_2$, 
    which is left of the intersection point of $W_R(a,b')$ with $\ell_1$ 
    (by \ref{prop:tuple-is-valid:left} and \ref{prop:tuple-is-valid:disjointness}). 
    Thus, the point of $W_R(b)$ in $\ell_1$ is also left of the point of $W_R(a,b')$ in $\ell_1$.

    Next, we construct $\ell_2$.
    First, consider the case where $m_2 \prec_\uparrow b \prec_\uparrow m_4$. 
    In this case $\ell_2$ is the horizontal line containing $b$. 
    By the assumption, $b$ is right of $\gamma_1'$.
    By the case distinction, the horizontal line containing $b$ intersects only $\gamma_1'$ and $\gamma_4'$ among the curves $\gamma_1'$, $\gamma_2'$, $\gamma_3'$, and $\gamma_4'$.
    If $b$ is left of $\gamma_4'$, then $b \in \calR$ by \cref{obs:curves-region}. 
    Thus, we obtain that $b$ is right of $\gamma_4'$, and also of $W_R(a,b')$, as desired.
    In the case where $m_4 \prec_\uparrow b \preccurlyeq_\uparrow b'$, $b$ is right of $W_R(a,b')$ as $m_4[W_R(a,b')]b' = m_4[W_R(b')]b'$. 
    Thus, again we take $\ell_2$ to be the horizontal line containing $b$. 
    Finally, when $b' \prec_\uparrow b$, $b'$ is left of $W_R(b)$ so we take 
    $\ell_2$ to be the horizontal line containing $b'$.
    This completes the proof that $b \in \Int \calR$.
    The path $q_1[W_R(b)]b$ is disjoint from $\gamma_3'$ and $\gamma_4'$.
    Since this path is consistent with $W_R(b')$ and $W_R(b'')$, the fact that $b \in \Int \calR$ implies that $q_1[W_R(b)]b \subset \calR$, so~\ref{prop:properties-regions:b-inside} holds.

    For the proof of \ref{prop:properties-regions:d-outside}, let $d \in B$ with $d \prec_R b'$.
    If $d\preccurlyeq_\uparrow b'$, then $d$ is left of $W_R(b')$, thus $d \notin \calR$ by \ref{prop:tuple-is-valid:left} and \cref{obs:curves-region}.
    Otherwise, $d$ is neither left nor right of any $\gamma_i'$ for $i \in [4]$, hence, by~\cref{obs:curves-region}, $d \notin \calR$.
    %In fact, the last argument applies also to $b'$ assuming that $b' \neq m_4$.
    %Thus, we also obtain \ref{prop:properties-regions:b'}.

    For the proof of \ref{prop:properties-regions:c-outside}, fix $c \in A$ and $d \in B$ with $c \leq d$, $c \parallel b'$ in $P$, and $d \prec_R b'$.
    By \ref{prop:properties-regions:d-outside}, we have $d \notin \calR$.
    Suppose to the contrary that $c \in \calR$ and let $W$ be a witnessing path from $c$ to $d$ in $P$.

    Let $\ell_0$ be the horizontal line containing $x_0$ and let $r_0$ be the ray contained in $\ell_0$ emanating left from $x_0$.
    Let $\ell'$ be the horizontal line containing $b'$ and let $r'$ be the ray contained in $\ell'$ emanating right from $b'$.
    We define $\rho = r_0 \cup W_R(b') \cup r'$.
    Let $P_L$ and $P_R$ be the closures of the two components of the complement of $\rho$ in the plane.
    We set $P_L$ to be the component containing points left of $W_L(b')$ and $P_R$ to be the component containing points right of $W_R(b')$.
    Since every point of $P_L$ is left of $W_R(b')$ or higher than $b'$ or is in $W_R(b')$, $P_L$ does not contain points of $\calR$ other than ones in $W_R(b')$ (by \cref{obs:curves-region}).
    It follows that $\calR \subset P_R$ and in particular $c \in P_R$.
    On the other hand, we claim that $d \in P_L$.
    If $d \preccurlyeq_\uparrow b'$, then $d$ is left of $W_R(b')$ and $d \in P_L$.
    If $b' \prec_\uparrow d'$, then $d$ is higher than $\ell'$, hence, again $d \in P_L$, as claimed.

    Recall that $c \in \calR$ and $d \notin \calR$, and $c \in P_R$ and $d \in P_L$.
    It follows that $W$ intersects $\partial \calR$ and $\rho$.
    Since $c \parallel b'$ in $P$, $W$ is disjoint from $W_R(b')$ and $W_R(a,b')$.
    Let $u$ be an element of $W \cap \partial \calR$ and let $p$ be the point in $W \cap \rho$.
    We also show that $u$ does not lie in $\gamma_2'$.
    If this is the case, consider the path $x_0[W_R(b'')]u[W]d$.
    By \ref{prop:tuple-is-valid:left}, $u$ is right of $W_R(b')$, and hence, $u$ is right of $W_R(d)$. 
    Thus, $x_0[W_R(b'')]u[W]d$
    contains an element right of $W_R(d)$, which contradicts \cref{claim:diagram:extreme_paths_exist}.
    We have $u$ in $W_R(a,b'')$, and either $p \in r'$ or $p \in r_0$. 
    Since $\calR \subset P_R$ and $W$ is disjoint from $W_R(b')$, $p$ is a point in $u[W]d$.

    First, suppose that $p \in r'$.
    By \ref{prop:tuple-is-valid:left}, $u$ is left of $W_R(a,b')$ and by the assumptions $p$ is right of $b'$.
    It follows that $W$ and $W_R(a,b')$ cross two horizontal lines in opposite orders, and so, by the Darboux property, they intersect.
    However, since $c \parallel b'$ in $P$, this is a contradiction.
    Next, we suppose that $p \in r_0$.
    It follows that $a \prec_\uparrow x_0$.
    Thus, by \ref{item-regions-W-left-W}, $x_0$ is left of $W_R(a,b'')$.
    On the other hand, $x_0$ is left of $W$ and $W$ is a witnessing path from $a$ to an element of $B$ in $P$.
    Therefore, $a$ is a bottom element, which contradicts \ref{item-regions-not-bottom}.
    This completes the proof of \ref{prop:properties-regions:c-outside}.    
\end{proof}

\begin{proposition}\label{prop:two-regions}
    Let $a,c \in A$ and $b',d',b'' \in B$ such that $a$, $b'$, $b''$, and $c$, $d'$, $b''$ satisfy \ref{item-regions-first}--\ref{item-regions-last}.
    Let $\calR' = \calR(a,b',b'')$ with the extreme points $(q_1',m_2',q_3',m_4')$ and $\calR = \calR(c,d',b'')$ with the extreme points $(q_1,m_2,q_3,m_4)$.
    %Let $x,y \in A$ with $x \leq y$ in $P$, and let $W$ be a witnessing path from $x$ to $y$ in $P$
    Assume that $\calR' \subset \calR$. %, $x \in \calR'$, and $y \notin \Int \calR$. 
    Let $W$ be a witnessing path in $P$ with all the elements in $A$ starting in $\calR'$ and ending in an element not in $\Int \calR$.
    %\later{Unify enumerates (each is a statement or there is one statement).}
    \begin{enumerate}
        \item $q_1 \leq q_1'$ in $P$. \label{prop:two-regions:q}
        \item $m_2 \leq m_2'$ in $P$. \label{prop:two-regions:m}
        % \item If $W$ intersects $q_3'[W_R(a,b')]m_4' \setminus\{q_3'\}$, then $y$ does not lie in $q_3[W_R(c,b'')]m_2$. \label{prop:two-regions:gamma4}
        % \item If $W$ is disjoint from $q_3'[W_R(a,b')]m_4' \setminus\{q_3'\}$, then $y$ lies in $q_3[W_R(c,b'')]m_2$. \label{prop:two-regions:gamma3}
        % \item if $W$ is disjoint from $\gamma_4(\calR')$, then         
        % %$x$ lies in $\gamma_3(\calR')$ and 
        % $y$ lies in $\gamma_3(\calR)$;
        % \item if $W$ intersects $\gamma_4(\calR')$, then 
        % $y$ lies in $\gamma_4(\calR)$; 
        \item If $W$ intersects $q_3'[W_R(a,b')]m_4' \setminus\{q_3'\}$, then 
        $W$ intersects $q_3[W_R(c,d')]m_4 \setminus\{q_3\}$. \label{prop:two-regions:gamma3}
        \item If $W$ is disjoint from $q_3'[W_R(a,b')]m_4'$, then 
        $W$ intersects $q_3[W_R(c,b'')]m_2 \setminus\{q_3\}$.\label{prop:two-regions:gamma4} % $q_3'[W_R(a,b'')]m_2' \setminus\{q_3'\}$ and $q_3[W_R(c,b'')]m_4 \setminus\{q_3\}$.\label{prop:two-regions:gamma4}
    \end{enumerate}
\end{proposition}
\begin{proof}
    Assume that $\calR' \subset \calR$.
    Since $q_1' \in \calR$, by \cref{prop:properties-regions}.\ref{prop:properties-regions:d-inside}, $q_1$ lies in $W_R(q_1')$, and so, $q_1 \leq q_1'$ in $P$ giving \ref{prop:two-regions:q}.
    For the proof of \ref{prop:two-regions:m}, suppose to the contrary that $m_2' < m_2$ in $P$.
    Then, by \cref{prop:properties-regions}.\ref{prop:properties-regions:b''}, $m_2 \in \Int \calR'$.
    However, $m_2 \in \partial \calR$, thus, every open ball centered at $m_2$ contains a point outside of $\calR$.
    This contradicts $\calR' \subset \calR$ and yields \ref{prop:two-regions:m}.

    For the proofs of the remaining items, we assume without loss of generality that $W$ is a witnessing path from an element $x \in A$ to an element $y \in A$ with $x \in \partial \calR'$ and $y \in \partial \calR$ and these are the only elements in the respective boundaries.
    For the proof of~\ref{prop:two-regions:gamma3}, assume that $x$ lies in $q_3'[W_R(a,b')]m_4' \setminus\{q_3'\}$. 
    Since $y\in A$ and $y\in\partial \calR$, by \ref{prop:tuple-is-valid:inA},
    $y$ lies either in $q_3[W_R(c,b'')]m_2$ or in 
    $q_3[W_R(c,d')]m_4 \setminus \{q_3\}$.
    Suppose to the contrary that $y$ lies in $q_3[W_R(c,b'')]m_2$.
    Let $\ell_x$ and $\ell_y$ be the horizontal lines containing $x$ and $y$, respectively.
    We have 
    \[ q_3' \prec_\uparrow x \preccurlyeq_\uparrow y \preccurlyeq_\uparrow m_2 \preccurlyeq_\uparrow m_2'.\]
    In particular, both $\ell_x$ and $\ell_y$ intersect $W_R(a,b'')$.
    Let $u$ be the point of $W_R(a,b'')$ in $\ell_y$.
    If $y$ is right of $u$, then $a[W_R(a,b')]x[W]y[W_R(c,b'')]b''$ is a witnessing path from $a$ to $b''$ in $P$ with an element right of $W_R(a,b'')$, which contradicts \cref{claim:diagram:extreme_paths_exist}.
    Thus, $y$ is not right of $u$.
    In other words, the intersection of $W_R(a,b'')$ with $\ell_y$ is not left of $y$.
    By \ref{prop:tuple-is-valid:left}, the intersection of $W_R(a,b'')$ with $\ell_x$ is left of $x$.
    By the Darboux property, it follows that $W$ and $W_R(a,b'')$ intersect in an element $v$ with $x \prec_\uparrow v \preccurlyeq_\uparrow y$.
    However, $a[W_R(a,b')]x[W]v[W_R(a,b'')]b''$ is a witnessing path from $a$ to $b''$ containing a point right of $W_R(a,b'')$, which again contradicts \cref{claim:diagram:extreme_paths_exist}.
    This completes the proof of \ref{prop:two-regions:gamma3}.

    For the proof of~\ref{prop:two-regions:gamma4}, assume that $W$ is disjoint from $q_3'[W_R(a,b')]m_4'$, and so, $x$ lies in $q_3'[W_R(a,b'')]m_2' \setminus\{q_3'\}$.
    Since $y\in A$ and $y\in\partial \calR$, by \ref{prop:tuple-is-valid:inA},
    $y$ lies either in $q_3[W_R(c,b'')]m_2  \setminus \{q_3\}$ or in 
    $q_3[W_R(c,d')]m_4$.
    Suppose to the contrary that $y$ lies in $q_3[W_R(c,d')]m_4$.
    If $x_0 \preccurlyeq_\uparrow x$, then let $p = x$ and otherwise, let $p$ be the point of $W$ in the horizontal line containing $x_0$.
    By \ref{prop:tuple-is-valid:left}, $p$ is right of $W_R(b')$ and left of $W_R(a,b')$.
    The path $W$ is disjoint from both $W_R(b')$ and $W_R(a,b')$.
    By \cref{prop:diagram:for_consistent_one_distinction_is_enough}, $W$ is right of $W_R(b')$ and left of $W_R(a,b')$.
    In particular, $y \prec_\uparrow m_4'$.
    Consider the element $u$ of $q_3'[W_R(a,b')]m_4'$ in the horizontal line containing $y$.
    It follows that $u$ is right of $W_R(c,d')$, which by \cref{prop:properties-regions}.\ref{prop:properties-regions:p-inside}, $u \notin \calR$, which contradicts $\calR' \subset \calR$.
    This completes the proof of \ref{prop:two-regions:gamma4}.
\end{proof}

\subsection{The smallest region of a pair}\label{ssec:smallest-region}

In this subsection, for each pair $(a,b) \in I_R$ and fixed $b'' \in B_*''$, we identify $b' \in B'(a,b)$ so that the region $\calR(a,b',b'')$ is the smallest possible.
First, we justify that $\calR(a,b',b'')$ is actually well-defined in this setup.
%\later{I'm a bit annoyed by the spacing before and after statements. Maybe it can be nicer.}

\begin{proposition}\label{prop:satisfy-items-r}
    Let $(a,b) \in I_R$, $b' \in B'(a,b)$, and $b'' \in B''_*(a,b)$.
    Then, $a$, $b'$, and $b''$ satisfy \ref{item-regions-first}--\ref{item-regions-last}.
\end{proposition}
\begin{proof}
    Items \ref{item-regions-b'-r-b''}, \ref{item-regions-leq}, and \ref{item-regions-W-left-W} follow from the definition of $B'(a,b)$ and $B''_*(a,b)$.
    Item \ref{item-regions-not-bottom} follows from \ref{items:diagram_pairs:a_bottom}.
    Finally, let $q = \gcpe(W_R(b'),W_R(b''))$.
    Since $b' \prec_R b \prec_R b''$, by \cref{prop:sandwich}, $q$ lies in $W_R(b)$.
    Thus, $a \not\leq q$ in $P$.
    On the other hand, $a \in A$ (by \ref{items:diagram_pairs:a_in_A}), hence $q \not< a$ in $P$.
    This completes the proof.
\end{proof}

Let $(a,b) \in I_R$.
Let $B_R'(a,b)$ be the set of all the maximal elements with respect to $\prec_R$ of the elements in $B'(a,b)$.
It follows that $B_R'(a,b)$ is an antichain with respect to $\prec_R$.
In other words for all $d,d' \in B_R'(a,b)$, either $W_R(d)$ is a subpath of $W_R(d')$ or $W_R(d')$ is a subpath of $W_R(d)$.
In particular, the elements of $B_R'(a,b)$ can be labeled $b_1',\dots,b_m'$ so that $W_R(b_i')$ is a subpath of $W_R(b_{i+1}')$ for every $i \in [m-1]$.
We set \defin{$z'(a,b)$} to be $b_1'$, i.e.\ the lowest element of $B_R'(a,b)$ in $P$.
See \cref{fig:smallest-region}.
\begin{figure}
    \centering
    \includegraphics{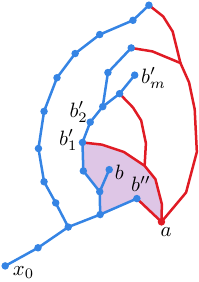}
    \caption{For any \((a, b) \in I_R\) and \(b'' \in B''_*(a, b)\), the region \(\calR(a, b', b'')\) is minimal for \(b' = z'(a, b) = b_1'\).}
    \label{fig:smallest-region}
\end{figure}

\begin{proposition}\label{prop:m4=z'}
    For every $(a,b) \in I_R$, we have $z'(a,b) = \gcpe(W_R(z'(a,b)),W_R(a,z'(a,b)))$.
    %\jedrzej{If at the end, there is no usage, we can remove it (although, it describes the situation.}
\end{proposition}
\begin{proof}
    Let $z = \gcpe(W_R(z'(a,b)),W_R(a,z'(a,b)))$.
    Since $a \leq z$ in $P$, we have $\gcpe(W_R(z'(a,b),W_R(b)) < z$ in $P$.
    Thus, $z'(a,b) \prec_R b$ implies $z \prec_R b$, and also $z \prec_L b$ by \cref{prop:one-side-implies-the-other}.
    In particular, $z \in B'(a,b)$.
    By the definition of $z'(a,b)$, we conclude that $z = z'(a,b)$, as desired.
\end{proof}

\begin{proposition}\label{prop:smallest-region}
    Let $(a,b) \in I_R$, $b' \in B'(a,b)$, and $b'' \in B''_*(a,b)$.
    Then, $\calR(a,z'(a,b),b'') \subset \calR(a,b',b'')$.
\end{proposition}
\begin{proof}
    For brevity, we denote $z' = z'(a,b)$.
    Let $\calR = \calR(a,b',b'')$ with the extreme points $(q_1,m_2,q_3,m_4)$ and $\calR' = \calR(a,z',b'')$ with the extreme points $(q_1',m_2',q_3',m_4')$.
    We will prove that $\partial\calR' \subset \calR$, which by \cref{obs:region_containment} will complete the proof.

    Since $b' \prec_R b$, $q_1 < m_4$ (by \ref{prop:tuple-is-valid:qs-ms}), and $W_R(m_4) = x_0[W_R(b')]b'$, we have $m_4 \in B'(a,b)$.
    It follows that either $m_4 \prec_R z'$ or $z'$ lies in $W_R(m_4)$.
    We claim that $q_1[W_R(z')]z' \subset \calR$.
    If $z'$ lies in $W_R(m_4)$, then we also have $q_1 < z'$ as otherwise $a \leq z' \leq q_1 \leq b$ in $P$, which is a contradiction.
    It follows that $q_1[W_R(z')]z' \subset \calR$, as claimed.
    If $m_4 \prec_R z'$, then $m_4 \parallel z'$ in $P$, hence, by \cref{prop:properties-regions}.\ref{prop:properties-regions:b-inside}, $q_1[W_R(z')]z' \subset \calR$, as claimed.

    Since $q_1$ lies in both $W_R(z')$ and $W_R(b'')$, $q_1 \leq q_1'$ in $P$.
    By \cref{prop:properties-regions}.\ref{prop:properties-regions:b''}, $q_1'[W_R(b'')]b'' \subset \calR$.

    Next, we show that $q_3'[W_R(a,z')]z' \subset \calR$.
    First, we argue that $W_R(a,z')$ is disjoint from $W_R(b'')$.
    Let $q = \gcpe(W_R(b),W_R(b'')$.
    By \Cref{prop:properties-regions}.\ref{prop:properties-regions:b-inside}, $q_1 \leq q$ in $P$.
    Suppose that $W_R(a,z')$ intersects $W_R(b'')$, say in an element $u$.
    If $u \leq q$ in $P$, then $a \leq b$ in $P$, which is a contradiction.
    Thus, $q < u$ in $P$.
    However, since $z' \prec_R b \prec_R b''$, the path $x_0[W_R(b'')]u[W_R(a,z')]z'$ is a witnessing path from $x_0$ to $z'$ in $P$ that contains a vertex right of $W_R(z')$, which contradicts \cref{claim:diagram:extreme_paths_exist}.
    It follows that $W_R(a,z')$ is indeed disjoint from $W_R(b'')$.
    
    Suppose to the contrary that $q_3'[W_R(a,z')]z'$ is not contained in $\calR$.
    Recall that $z' \in \calR$.
    Let $e$ be the highest edge of $q_3'[W_R(a,z')]z'$ not in $\calR$ with the higher endpoint $u$.
    By the previous paragraph, $u$ does not lie in $W_R(b'')$.
    By \cref{claim:diagram:extremal_paths_are_consistent}, if $u$ is an element of $W$ with $W \in \{W_R(a,b'), W_R(a,b'')\}$, then $a[W]u = a[W_R(a,z')]u$, hence, $q_3'$ is an element of $q_3[W_R(a,b'')]b''$. 
    This contradicts the definition of $e$.

    %\later{The next paragraph is repeated in the proof of \cref{prop:sac-region-containment}. However, there is no clear abstraction.}
    We conclude that $u$ lies in $W_R(b')$.
    Note that $q_1 \prec_\uparrow u$.
    There exists a point $p$ in $e$ (take one very close to $u$) such that $p$ is not in $\calR$ and $p$ is left of $W_R(b')$ by \cref{obs:curves-region}.
    Let $\ell$ be the horizontal line containing $p$.
    If $p \preccurlyeq_\uparrow b$, then we have an immediate contradiction with $(a,b)$ being a right pair as $b' \prec_R b$.
    Thus, we may assume that $b \prec_\uparrow p$.
    Let $\ell'$ be the horizontal line containing $b$.
    The point of $W_R(b')$ in $\ell'$ is left of $b$ (as $b' \prec_R b$ and $b \prec_\uparrow m_4 \prec_\uparrow b'$) and $b$ is left of the point of $W_R(a,z')$ in $\ell'$ (as $(a,b)$ is a right pair).
    The paths $W_R(b')$ and $W_R(a,z')$ intersect $\ell$ and $\ell'$ in the opposite order, thus by the Darboux property, they intersect in a point $p'$ with $p' \prec_\uparrow p$.
    However, by \cref{claim:extreme_subpaths}, $p'[W_R(a,z')]u = p'[W_R(b')]u$, which contradicts the definition of $p$.
    The contradiction gives that $q_3'[W_R(a,z')]z'$ is contained in $\calR$, as desired.
    
    By the definition, $q_3'$ lies in $W_R(a,b'')$.
    \cref{prop:properties-regions}.\ref{prop:properties-regions:a-outside} implies that $a[W_R(a,b'')]q_3 \setminus \{q_3\}$ is not in $\calR$.
    Therefore, $q_3'$ lies in $q_3[W_R(a,b'')]b''$.
    In particular, $q_3'[W_R(a,b'')]m_2' \subset \calR$ by \cref{prop:properties-regions}.\ref{prop:properties-regions:b''}.
    This completes the proof that $\partial\calR' \subset \calR$ and by \cref{obs:region_containment}, we obtain $\calR' \subset \calR$, as desired.
\end{proof}

\subsection{Strict alternating cycles}\label{ssec:sacs}

In this subsection, we study the structure of strict alternating cycles in $I_R$.

\begin{proposition}\label{prop:force-inB*}
    Let $(a,b) \in I_R$ and $(c,d) \in I_R'$ such that $((a,b),(c,d))$ is a strict alterntaing cycle in $P$ with $d \prec_R b$.
    Let $b'' \in B_*''(a,b)$. Then
    \begin{enumerate}
        \item $b,b'' \in B_*''(c,d)$,\label{prop:force-inB*:bb'}
        \item $\gcpe(W_R(b''),W_R(c,b'')) \leq \gcpe(W_R(b),W_R(b''))$ in $P$, \label{prop:force-inB*:cleq}
        \item $\gcpe(W_R(d),W_R(b'')) < \gcpe(W_R(b),W_R(b''))$ in $P$.\label{prop:force-inB*:gcpe}
    \end{enumerate}
\end{proposition}
\begin{proof}
    Since $(a,b)\in I_1$, $a\leq d$ in $P$, and $d \prec_R b$, 
    by \cref{prop:one-side-implies-the-other}, 
    we also have $d \prec_L b$.
    So $d \in B'(a,b)$, and hence, by \cref{prop:satisfy-items-r}, $a$, $d$, and $b''$ satisfy \ref{item-regions-first}--\ref{item-regions-last}.
    Let $\calR = \calR(a,d,b'')$ with the extreme points $(q_1,m_2,q_3,m_4)$.

    By \Cref{prop:properties-regions}.\ref{prop:properties-regions:b-inside}, we have $b \in \calR$.
    Since $(c,d) \in I_R'$, there exists $d' \in B'(c,d)$.
    Thus by \Cref{prop:properties-regions}.\ref{prop:properties-regions:c-outside}, $c \notin \calR$. 
    Since $c\not\in\calR$ and $b\in\calR$, $W_R(c,b)$ intersects $\partial \calR$.
    Since $a \parallel b$ and $c \parallel d$ in $P$, $W_R(c,b)$ must intersect $W_R(b'')$.
    We conclude that $c \leq b''$ in $P$. 
    Let $q = \gcpe(W_R(b),W_R(b''))$ and let $u$ be an element in the intersection of $W_R(c,b)$ and $W_R(b'')$.
    We claim that $u \leq q$ in $P$.
    Suppose to the contrary that $q < u$ in $P$.
    Since $b \prec_R b''$, $u$ is right of $W_R(b)$.
    The path $x_0[W_R(b'')]u[W_R(c,b)]b$ contains an element right of $W_R(b)$, which contradicts \cref{claim:diagram:extreme_paths_exist}.
    Therefore, $u \leq q$ in $P$, which by \cref{claim:diagram:extremal_paths_are_consistent} gives \ref{prop:force-inB*:cleq}.
    
    Since $c \parallel d$ in $P$ and $d \prec_R b$, we also have $\gcpe(W_R(d),W_R(b'')) < q \leq \gcpe(W_R(b),W_R(b''))$ in $P$, which gives \ref{prop:force-inB*:gcpe}.

    Recall that $d \prec_R b$, $d \prec_L b$, $b \prec_R b''$, and $b \prec_L b''$.
    It follows that $d \prec_R b''$ and $d \prec_L b''$.
    We also have $c \leq b$ and $c\leq b''$ in $P$.
    To conclude, it suffices to prove that $W_R(c,b)$ is right of $W_R(b)$ and $W_R(c,b'')$ is right of $W_R(b'')$.
    
    Suppose that $c \preccurlyeq_\uparrow x_0$. 
    Then since $(c,d)$ is a right pair, every witnessing path from $c$ to an element of $B$ is right of $x_0$, and so, 
    right of $W_R(b)$ and $W_R(b'')$. 
    We conclude that $W_R(c,b)$ is right of $W_R(b)$ and also that $W_R(c,b'')$ is right of $W_R(b'')$. 
    Thus, $b,b''\in B_*''(c,d)$, as desired.

    Finally, suppose that $x_0\prec_\uparrow c$. 
    All we need to prove is that $c$ is right of $W_R(b)$ and $W_R(b'')$.    
    Since $b \prec_R b''$, if $c$ is right of $W_R(b'')$, then $c$ is right of $W_R(b)$.
    Thus, we suppose to the contrary that $c$ is left of $W_R(b'')$ and the contradiction will complete the proof.
    Since the pair $(c,d)$ is right, $c$ is right of $W_R(d)$.
    Let $q = \gcpe(W_R(d),W_R(b))$.
    By \cref{prop:sandwich}, $q_1 \leq q$ in $P$, and so, $q_1 \preccurlyeq_\uparrow q$.
    Since $W_R(d)$ and $W_R(b)$ have distinct intersections with the horizontal line containing $c$, we also have $q \prec_\uparrow c$.
    Also, $q_1\preceq_\uparrow q\prec_\uparrow c \prec_\uparrow b \prec_\uparrow m_4$ (the last relation follows from $b \in \calR$ and \Cref{prop:properties-regions}.\ref{prop:properties-regions:b-lower-m4}).
    It follows that $c$ is right of $q_1[W_R(d)]m_4$.
    We assumed that $c$ is left of $W_R(b'')$.
    If $q_3[W_R(a,b'')]m_2$ intersects the horizontal line containing $c$, then $W_R(b'')$ also intersects this line.
    Thus, if $c$ is right of $q_3[W_R(a,b'')]m_2$, then $c$ is right of $W_R(b'')$ (by \ref{prop:tuple-is-valid:left}), which is a contradiction.
    Finally, suppose that $c$ is right of $W_R(a,d)$.
    By \Cref{prop:properties-regions}.\ref{prop:properties-regions:p-inside}, $b$ is not right of $W_R(a,d)$.
    Since $b \prec_\uparrow m_4$, $b$ is left of $W_R(a,d)$.
    Thus, by the Darboux property, every witnessing path from $c$ to $b$ in $P$ intersects  $W_R(a,d)$.
    This contradicts $a \parallel b$ in $P$.
    Summarizing, $c$ is right of $q_1[W_R(d)]m_4$ and is not right of $W_R(b'')$, $W_R(a,b'')$, and $W_R(a,d)$.
    By \cref{obs:curves-region}, this implies that $c \in \calR$, which is a contradiction.
\end{proof}

\begin{proposition}\label{prop:cycle-gives-an-edge}
    Let $((a_1,b_1),\dots,(a_k,b_k))$ be a strict alternating cycle in $P$ contained in $I_R$.
    Then, there exist $i,j \in [k]$ with $i \neq j$ and $d \in B$ such that $(a_j,d) \in I_R$, $b_j \prec_R d \prec_R b_i$, and $((a_i,b_i),(a_j,d))$ is a strict alternating cycle in $P$. 
    %\later{Maybe $f$ instead of $d$?} 
\end{proposition}
\begin{proof}
    Without loss of generality, assume that $b_i \preccurlyeq_R b_1$ for every $i \in [k]$.
    By \cref{prop:strict-cycle-ordering-bs}, $b_2 \prec_R b_1$ implies $b_2 \prec_L b_1$.
    In particular, $b_2 \in B'(a_1,b_1)$.
    Property \ref{items:diagram_pairs:B''_*} states that $B''_*(a_1, b_1) \neq \emptyset$, and so we, fix \(b''_1 \in B''_*(a_1, b_1)\).
    By \cref{prop:satisfy-items-r}, $a_1$, $b_2$, and $b''_1$ satisfy \ref{item-regions-first}--\ref{item-regions-last}.
    Thus, we may define a region $\calR = \calR(a_1,b_2,b''_1)$ with the extreme points $(q_1,m_2,q_3,m_4)$.

    We prove by induction that for every $i \in [k] \setminus\{1\}$, $a_i \notin \calR$ and for every $i \in [k] \setminus [2]$, $b_i \prec_R b_2$.
    We will prove that (1) for each $i\in\set{2,\ldots,k}$ $b_i \preccurlyeq_R b_2$ implies $a_i \notin \calR$ and that 
    (2) for each $i\in\set{2,\ldots,k-1}$ $a_i \notin \calR$ implies $b_{i+1} \prec_R b_2$.
    This suffices to complete the inductive proof.
    For the proof of (1), assume that $b_i \preceq_R b_2$.
    Let $b_i' \in B'(a_i,b_i)$.
    Note that $b_i' \prec_R b_i \preccurlyeq_R b_2$ and $a_i \parallel b_2$ in $P$.
    By \Cref{prop:properties-regions}.\ref{prop:properties-regions:c-outside} applied with $c = a_i$ and $d = b_i'$, we obtain that $a_i \notin \calR$, as desired.
    
    Next, we prove (2), for which we assume that $a_i \notin \calR$.
    By the choice of $b_1$, either $b_{i+1} \prec_R b_2$ or $b_2 \prec_R b_{i+1} \prec_R b_1$.
    The former case is the desired outcome, thus, we suppose to the contrary that the latter holds.
    In particular, $b_2 \prec_R b_{i+1} \prec_R b''_1$.
    Since the cycle is strict, we also have $a_1 \parallel b_{i+1}$ in $P$.
    Therefore, we may apply \Cref{prop:properties-regions}.\ref{prop:properties-regions:b-inside} to obtain that $b_{i+1} \in \calR$.
    Now, we have $a_i \notin \calR$ and $b_{i+1} \in \calR$.
    Let $W$ be a witnessing path from $a_i$ to $b_{i+1}$ in $P$.
    It intersects $\partial \calR$.
    Since $a_1 \parallel b_{i+1}$ and $a_i \parallel b_2$ in $P$, $W$ intersects specifically $q_1[W_R(b''_1)]m_2$, say in an element $u$.
    Let $q = \gcpe(W_R(b_1),W_R(b''_1))$.
    Since $a_i \parallel b_1$ in $P$, $W$ intersects $q[W_R(b''_1)]m_2$ and $q < u$ in $P$.
    Since $b_{i+1} \prec_R b''_1$, $u$ is right of $W_R(b_{i+1})$, and in particular, the path $x_0[W_R(b''_1)]u[W]b_{i+1}$ 
    contains a vertex right of $W_R(b_{i+1})$, which contradicts \cref{claim:diagram:extreme_paths_exist}.
    This implies that indeed, $b_{i+1} \prec_R b_2$ and completes the proof of (2), and the inductive claim.

    We claim that $i = 1$, $j = k$, and $d = b_2$ satisfy the assertion in the statement.
    First, note that $b_k \prec_R b_2 \prec_R b_1$ and $((a_1,b_1),(a_k,b_2))$ is a strict alternating cycle in $P$.
    It suffices to show that $(a_k,b_2) \in I_R$.
    Clearly, $(a_k,b_2) \in \closure{I}$.
    We must prove that $(a_k,b_2)$ satisfies \ref{items:diagram_pairs:b'_vertical}, \ref{items:diagram_pairs:b''_vertical}, \ref{items:diagram_pairs:b'}, \ref{items:diagram_pairs:b''}, \ref{items:diagram_pairs:a_bottom}, \ref{items:diagram_pairs:right}, and \ref{items:diagram_pairs:B''_*}.
    By \ref{items:diagram_pairs:b'} for $(a_k,b_k)$, the set $B'(a_k,b_k)$ is nonempty.
    Let $b_k' \in B'(a_k,b_k)$.

    Property \ref{items:diagram_pairs:b'_vertical} states that there exists $q' \in B$ with $a_k \leq q'$ in $P$ and $q' \prec_\uparrow b_2$.
    We set $q' = b_1$.
    It suffices to prove that $b_1 \prec_\uparrow b_2$.
    Since $b_2 \prec_R b_1 \prec_R b''_1$ and $a_1 \parallel b_1$ in $P$, by \cref{prop:properties-regions}.\ref{prop:properties-regions:b-inside}, we have $b_1 \in \calR$.
    In particular, by \cref{prop:properties-regions}.\ref{prop:properties-regions:b-lower-m4}, $b_1 \prec_\uparrow m_4$, and so, $b_1 \prec_\uparrow b_2$, as desired.
    
    Property \ref{items:diagram_pairs:b''_vertical} states that there exists $q'' \in B$ with $a_k \leq q''$ in $P$ and $b_2 \prec_\uparrow q''$.
    We set $q''= b_k'$.
    It suffices to prove that $b_2 \prec_\uparrow b_k'$.
    Suppose to the contrary that $b_k' \preccurlyeq_\uparrow b_2$.
    If $x_0 \preccurlyeq_\uparrow a_k$, then let $\ell$ be the horizontal line containing $a_k$.
    Otherwise, let $\ell$ be the horizontal line containing $x_0$.
    Let $W$ be a witnessing path from $a_k$ to $b_k'$ in $P$ and let $u$ be the point of $W$ in $\ell$.
    We claim that $u$ is right of $W_R(b_2)$.
    If $a_k \prec_\uparrow x_0$, then this follows from the fact that $(a_k,b_k)$ is a right pair (by \ref{items:diagram_pairs:right}).
    Assume that $x_0 \preccurlyeq_\uparrow a_k$.
    Since $b_1 \prec_\uparrow b_2$ and $b_2 \prec_R b_1$, $b_1$ is right of $W_R(b_2)$.
    Thus, since $a_k \parallel b_2$ in $P$, by \cref{obs:where-curves-end} applied to any witnessing path from $a_k$ to $b_1$ in $P$, we obtain that $a_k = u$ is right of $W_R(b_2)$, as desired.
    Next, consider the horizontal line $\ell'$ containing $b_k'$.
    We assumed that $b_k' \preccurlyeq_\uparrow b_2$, thus $b_k'$ is left of the point of $W_R(b_2)$ in $\ell'$.
    The paths $W$ and $W_R(b_2)$ intersect $\ell$ and $\ell'$ in the opposite order, hence, by the Darboux property, they intersect.
    This is a contradiction as $a_k \parallel b_2$ in $P$.
    Thus $(a_k,b_2)$ indeed satisfies \ref{items:diagram_pairs:b''_vertical}.

    Properties \ref{items:diagram_pairs:b'} and \ref{items:diagram_pairs:b''} state that there exist $b',b'' \in B$ with $a_k \leq b'$ and $a \leq b''$ in $P$ such that $b' \prec_L b_2 \prec_L b''$ and $b' \prec_R b_2 \prec_R b''$.
    We claim that $b' = b_k'$ and $b'' = b_1$ witness the above.
    Since $b_k' \in B'(a_k,b_k)$, we have $b_k' \prec_L b_k$ and $b_k' \prec_R b_k$.
    We have already argued that $b_k \prec_R b_2 \prec_R b_1$.
    Thus, by \cref{prop:strict-cycle-ordering-bs}, $b_k \prec_L b_2 \prec_L b_1$.
    It follows that the properties hold.
     
    Property \ref{items:diagram_pairs:a_bottom} states that $a_k$ is not a bottom element, which follows directly from \ref{items:diagram_pairs:a_bottom} for $(a_k,b_k)$.
    Property \ref{items:diagram_pairs:right} states that $(a_k,b_2)$ is a right pair.
    Note that we already have that $(a_k,b_2) \in I_3$.
    Thus, by~\Cref{claim:diagram:left_right_parititon}, $(a_k,b_2)$ is either left or right.
    Since $a_k \leq b_1$ in $P$ and $b_1$ is right of $W_R(b_2)$ by \cref{claim:diagram:left_right_parititon} $(a_k,b_2)$ is a right pair.
    Property \ref{items:diagram_pairs:B''_*} states that $B_*''(a_k,b_2) \neq \emptyset$.
    \cref{prop:force-inB*}.\ref{prop:force-inB*:bb'} applied to $(a,b) = (a_1,b_1)$, $(c,d) = (a_k,b_2)$, and $b'' = b_1''$, implies that $b_1 \in B_*''(a_k,b_2)$ proving that \ref{items:diagram_pairs:B''_*} holds.
\end{proof}

\begin{proposition}\label{prop:sac-region-containment}
    Let $((a,b),(c,d))$ be a strict alternating cycle in $P$ contained in $I_R$ with $d \prec_R b$.
    Let $d' \in B'(c,d)$ and $b'' \in B''_*(a,b)$.
    Then $d \in B'(a,b)$, $b'' \in B''_*(c,d)$, and $\calR(a,d,b'') \subset \calR(c,d',b'')$.
    %\jedrzej{If we need, we can also have $a,d \in \calR(c,d',b'')$.}
\end{proposition}
\begin{proof} 
    Since $d \prec_R b$, we also have $d \prec_L b$ by \Cref{prop:one-side-implies-the-other}, and so, $d \in B'(a,b)$.
    By \cref{prop:force-inB*}.\ref{prop:force-inB*:bb'}, we also have $b'' \in B_*''(c,d)$.

    By \cref{prop:satisfy-items-r}, $a$, $d$, and $b''$ satisfy \ref{item-regions-first}--\ref{item-regions-last}, as well as $c$, $d'$, and $b''$ satisfy these conditions.
    Thus, we may define the region $\calR' = \calR(a,d,b'')$ with the extreme points $(q_1',m_2',q_3',m_4')$ and the region $\calR = \calR(c,d',b'')$ with the extreme points $(q_1,m_2,q_3,m_4)$.
    We will show that $\partial \calR' \subset \calR$, which by \Cref{obs:region_containment} will complete the proof.

    By \Cref{prop:properties-regions}.\ref{prop:properties-regions:b-inside}, $q_1$ lies in $W_R(d)$ and $q_1[W_R(d)]d\subset \calR$.
    Since $q_1$ lies in $W_R(d)$ and $W_R(b'')$, we have $q_1 \leq q_1'$ in $P$.
    It follows that $q_1'[W_R(d)]d \subset \calR$ and by \cref{prop:properties-regions}.\ref{prop:properties-regions:b''} $q_1'[W_R(d)]b'' \subset \calR$.
    
    Next, we show that $W_R(a,d) \subset \calR$.
    Since $c \parallel d$ in $P$, $W_R(a,d)$ does not intersect neither $W_R(c,b'')$ nor $W_R(c,d')$.
    Let $q = \gcpe(W_R(b),W_R(b'')$.
    By \Cref{prop:properties-regions}.\ref{prop:properties-regions:b-inside}, $q_1 \leq q$ in $P$.
    Suppose that $W_R(a,d)$ intersects $W_R(b'')$, say in an element $u$.
    If $u \leq q$ in $P$, then $a \leq b$ in $P$, which is a contradiction.
    Thus, $q < u$ in $P$.
    However, since $d \prec_R b \prec_R b''$, the path $x_0[W_R(b'')]u[W_R(a,d)]d$ is a witnessing path from $x_0$ to $d$ in $P$ neither left nor equal to $W_R(d)$, which contradicts \cref{claim:diagram:extreme_paths_exist}.
    Thus, if $W_R(a,d)$ intersects $\partial \calR$, then it intersects $q_1[W_R(d')]m_4 \setminus\{q_1,m_4\}$.
    Suppose to the contrary that $W_R(a,d)$ is not contained in $\calR$.
    Note that $q_1 \prec_\uparrow u$.
    Let $e$ be the highest edge of $W_R(a,d)$ not in $\calR$ with the higher endpoint $u$.
    There exists a point $p$ in $e$ such that $p$ is not in $\calR$ and $p$ is left of $W_R(d')$ by \cref{obs:curves-region}.
    Let $\ell$ be the horizontal line containing $p$.
    If $p \preccurlyeq_\uparrow b$, then we have an immediate contradiction with $(a,b)$ being a right pair as $d' \prec_R b$.
    Thus, we may assume that $b \prec_\uparrow p$.
    Let $\ell'$ be the horizontal line containing $b$.
    The point of $W_R(d')$ in $\ell'$ is left of $b$ (as $d' \prec_R b$ and $b \prec_\uparrow m_4 \prec_\uparrow d'$) and $b$ is left of the point of $W_R(a,d)$ in $\ell'$ (as $(a,b)$ is a right pair).
    The paths $W_R(d')$ and $W_R(a,d)$ intersect $\ell$ and $\ell'$ in the opposite order, thus by the Darboux property, they intersect in a point $p'$ with $p' \prec_\uparrow p$.
    However, by \cref{claim:extreme_subpaths}, $p'[W_R(a,d)]u = p'[W_R(d')]u$, which contradicts the definition of $p$.
    The contradiction gives that $W_R(a,d)$ is contained in $\calR$, as desired.

    To complete the proof, it suffices to show that $W_R(a,b'') \subset \calR$.
    We already know that $a \in \calR$.
    Suppose that $W_R(a,b'')$ intersects $\partial \calR$ and $u$ is the lowest element in this intersection.
    If $u$ lies in $q_1[W_R(b'')]m_2$ or $q_3[W_R(c,b'')]m_2$, then by consistency (\Cref{claim:diagram:extremal_paths_are_consistent}) and since $m_2[W_R(b'')]b'' \subset \calR$ (by \cref{prop:properties-regions}.\ref{prop:properties-regions:b''}) , $u[W_R(a,b'')]b'' \subset \calR$.
    Note that $u \preccurlyeq_\uparrow b''$.
    Suppose that $u$ lies in $q_1[W_R(d')]m_4 \setminus\{q_1\}$.
    It follows by \ref{prop:tuple-is-valid:left} that $u$ is left of $W_R(b'')$.
    However, since $b'' \in B_*''(a,b)$, we have $W_R(a,b'')$ right of $W_R(b'')$.
    This is a contradiction.
    Finally, suppose that $u$ lies in $q_3[W_R(d')]m_4 \setminus\{q_3\}$.
    In this case by \ref{prop:tuple-is-valid:left}, $u$ is right of $W_R(c,b'')$.
    It follows that $c[W_R(c,d')]u[W_R(b'')]b''$ is neither left nor equal to $W_R(c,b'')$, which contradicts \cref{claim:diagram:extreme_paths_exist}.
    
    We conclude that indeed $W_R(a,b'') \subset \calR$.
    This completes the proof that $\partial\calR' \subset \calR$ and by \cref{obs:region_containment}, we obtain $\calR' \subset \calR$, as desired.
\end{proof}

\subsection{Auxiliary oriented graphs}\label{ssec:auxiliary}
% We define two auxiliary graphs $H$ and $H'$ with vertex sets $I_R$.
% The key property of $H$ will be that in every strict alternating cycle in $P$ with all the pairs in $I_R$, there will be two pairs connected by an edge in $H$.
% In particular, a proper coloring of $H$ will give a covering of $I_R$ by as many reversible sets in $P$ as many colors the coloring uses.
% The graph $H$ will be a subgraph of $H'$ and we will find a coloring of $H$ with the number of colors depending on $\kelly_P(I_R)$ using $H'$.
% \piotr{I will come back to this paragraph once i seethe rest ofthe proof.} 
An \defin{orientation} of a graph $G$ is a function assigning to each edge $\set{u,v}$ of $G$ one of the pairs: $(u,v)$ or $(v,u)$. 
An \defin{oriented graph} is a graph with a fixed orientation. 
The \defin{vertex set} of an oriented graph $H$ is the vertex set of the underlying graph, while the \defin{edge set} of $H$ is the set of all $(u,v)$ such that $\{u,v\}$ is an edge in the underlying graph mapped to $(u,v)$ by the orientation.
We say that an edge $(u,v)$ of an oriented graph $H$ is an edge \defin{from $u$ to $v$} in $H$.
A \defin{directed path} is an oriented graph such that $\{v_0,\dots,v_k\}$ is its vertex set and $\{(v_{i-1},v_i) : i \in [k]\}$ is its edge set, where $v_0,\dots,v_k$ are pairwise distinct.
This directed path \defin{starts} in $v_0$ and \defin{ends} in $v_k$.
A \defin{directed cycle} is an oriented graph with at least three vertices such that removing each of its edges gives a directed path.
An oriented graph is \defin{acyclic} if it contains no directed cycle.

Let \defin{$H$} be an oriented graph with the vertex set $I_R$ such that $(a,b) \in I_R$ is connected by an edge oriented towards $(c,d) \in I_R$ if there exists $f \in B$ such that
\begin{enumerateNumH}
    \item $d \preccurlyeq_R f \prec_R b$, \label{itemH:ordering}
    \item $(c,f) \in I_R$,\label{itemH:pair}
    \item $((a,b),(c,f))$ is a strict alternating cycle in $P$.\label{itemH:sac}
\end{enumerateNumH}
See \Cref{fig:auxiliaryH}.
\begin{figure}
    \centering
    \includegraphics{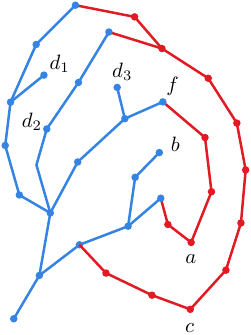}
    \caption{\(((a, b), (c, f))\) is a strict alternating cycle with \(\{(a, b), (c, d)\} \subseteq I_R\) and \(f \prec_R b\). For each \(d \in \{d_1, d_2, d_3, f\}\), we have \(d \preccurlyeq_R f\) and \((c, f) \in I_R\), so \(H\) has an edge oriented from \((a, b)\) towards \((c, f)\).}
    \label{fig:auxiliaryH}
\end{figure}
Note that every strict alternating cycle $((a,b),(c,d))$ in $P$ contained in $I_R$ is an edge in $H$ as witnessed by $d = f$.
The following is implied by \cref{prop:cycle-gives-an-edge}.

\begin{corollary}\label{cor:sacs}
    Let $((a_1,b_1),\dots,(a_k,b_k))$ be a strict alternating cycle in $P$ contained in $I_R$.
    Then, there exist $i,j \in [k]$ with $i \neq j$ such that $((a_i,b_i),(a_j,b_j)) \in E(H)$.
\end{corollary}

Let \defin{$H'$} be an oriented graph with the vertex set $I_R$ such that $(a,b) \in I_R$ is connected by an edge oriented towards $(c,d) \in I_R$ if 
\begin{enumerateNumHprim}
    \item $d \prec_R b$,\label{itemH':ordering}
    \item $B_*''(a,b) \subset B_*''(c,d)$,\label{itemH':B*}
    \item for every $b'' \in B_*''(a,b)$, we have $\calR(a,z'(a,b),b'') \subset \calR(c,z'(c,d),b'')$.\label{itemH':regions} 
    %\later{MS: Note: the original definition (with \(d \preccurlyeq_R b\)) tried to capture ``(the region of \((a, b)\)) \(\subseteq\) (the region of \((c, d)\))''. (H3') was formulated this way to allow the edge \((a, b) \to (a, b)\). If we are aiming for \(\subsetneq\) relation on the regions, it seems plausible to replace (H2') and (H3') by a condition where the LCA of all \(B''_*(a, b)\) is in \(B''_*(c, d)\): ``there exists \(d'' \in B''_*(c, d)\) such that for every \(b'' \in B''_*(a, b)\) we have \(d'' \in W_R(b'')\) and $\calR(a,z'(a,b),b'') \subset \calR(c,z'(c,d),d'')$'' }
\end{enumerateNumHprim}
We say that an edge of $H'$ is a \defin{cycle edge} whenever it is a strict alternating cycle in $P$.
Due to \ref{itemH':ordering}, $H'$ is acyclic, and we may apply to it the following algorithmic observation.
\begin{figure}
    \centering
    \includegraphics{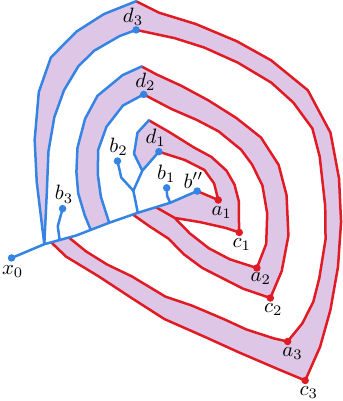}
    \caption{A path \((a_1, b_1)(c_1, d_1)(a_2, b_2)(c_2, d_2)(a_3, b_3)(c_3, d_3)\) in \(H'\) (oriented from left to right), and an element \(b'' \in B''_*(a_1, b_1)\). Each \((a_i, b_i)(c_i, d_i)\) is a strict alternating cycle and an edge of \(H\). For each \(i \in [3]\), the difference \(\calR(c_i, z'(c_i, d_i), b'') \setminus \calR(a_i, z'(a_i, b_i), b'')\) is marked purple.}
    \label{fig:path-in-H'}
\end{figure}

\begin{obs}\label{obs:algo-compute-longest}
    There exists a polynomial-time algorithm that, given an oriented acyclic graph $G'$ and its subset of vertices $E'$, for each vertex $v$ computes the maximum number of edges in $E'$ in a path starting in $v$ in $G'$. 
\end{obs}

\begin{proposition}\label{prop:H-subgraph-H'}
    $E(H) \subset E(H')$.
\end{proposition}
\begin{proof}
    Let $((a,b),(c,d)) \in E(H)$ witnessed by $f \in B$.
    Condition \ref{itemH:ordering} directly implies condition \ref{itemH':ordering}.
    For the proofs of the remaining conditions, let $b'' \in B_*''(a,b)$.
    \cref{prop:force-inB*}.\ref{prop:force-inB*:bb'} applied to a strict alternating cycle $((a,b),(c,f))$ gives $b'' \in B_*''(c,f)$.
    Since $d \preccurlyeq_R f$, this directly implies $b'' \in B_*''(c,d)$, and so, \ref{itemH':B*} holds.
    Finally, $\calR(a,z'(a,b),b'') \subset \calR(a,d,b'')$ by \cref{prop:smallest-region} and $\calR(a,d,b'') \subset \calR(c,z'(c,d),b'')$ by \cref{prop:sac-region-containment} giving \ref{itemH':regions}.
    This completes the proof.
\end{proof}

Straight from the definitions, we obtain the following facts on edges of $H'$.

\begin{obs}\label{obs:righter-b'-gives-an-edge}
    For every $(c,f),(c,d) \in I_R$, if $d \prec_R f$, then $((c,f),(c,d)) \in E(H')$.
\end{obs}

\begin{obs}\label{obs:transitive}
    $H'$ is transitive, i.e.\ $((a,b),(c,d)) \in E(H')$ and $((c,d),(e,f)) \in E(H')$ imply $((a,b),(e,f)) \in E(H')$.
\end{obs}

For every $(a,b) \in I_R$, we define
\[\defmath{\ell(a,b)} = \text{the maximum number of cycle edges in a directed path in $H'$ starting in $(a,b)$.}\]
See an example of a path in $H'$ in \cref{fig:path-in-H'}.

\begin{proposition}\label{prop:ell-plus-1}
  For every edge \(((a, b), (c, d)) \in E(H)\), we have
  \[\ell(a, b) \ge \ell(c, d) + 1.\]
\end{proposition}
\begin{proof}
    Let \(f \in B\) be an element witnessing that \(((a, b), (c, d))\) is an edge in $H$.
    The pair $((a,b),(c,f))$ is an edge in $H$, thus, by \cref{prop:H-subgraph-H'} also a cycle edge in $H'$.
    If $d = f$, this gives the statement, as we can prepend $(a,b)$ to every path in $H'$ starting in $(c,d)$.
    If $d \prec_R f$, then by \cref{obs:righter-b'-gives-an-edge}, $((c,f),(c,d))$ is an edge in $H'$, which gives the statement as we can prepend $(a,b)$ and $(c,f)$ to every path in $H'$ starting $(c,d)$.
\end{proof}

\begin{lemma}\label{lemma:shortcuts}
    There exists a polynomial-time algorithm that takes a positive integer $n$ and an edge $((a,b),(c,d))$ in $H$ with $\ell(a,b) \geq \ell(c,d) + (2n+7)$ and returns a Kelly subposet of $P$ of order $n$ based on $I_R$.
\end{lemma}
\begin{proof}
    Let $n$ be a positive integer and let $((a,b),(c,d))$ in $H$ with $\ell(a,b) \geq \ell(c,d) + (2n+7)$. 
    We fix a path in $H'$ starting in $(a,b)$ containing $\ell(a,b)$ cycle edges.
    By the transitivity of $H'$ (\Cref{obs:transitive}), we may assume that this path is of the form 
    \[(a,b)(a_1,b_1)(c_1,d_1)(a_2,b_2)(c_2,d_2)\dots(a_m,b_m)(c_m,d_m)\]
    where for every $j \in [m]$, $((a_j,b_j),(c_j,d_j))$ is a cycle edge in $H'$ and for every $j \in [m]$ either $(c_{j-1},d_{j-1}) = (a_j,b_j)$ or $((c_{j-1},d_{j-1}),(a_{j},b_{j}))$ is an edge in $H'$ where $(c_0,d_0) = (a,b)$.
    Note that $\ell(a,b) \geq 2n+7$, thus $m \geq 2n+7$.

    Fix $b'' \in B''_*(a,b)$.
    For every $i \in [m]$, by \ref{itemH':B*}, we have $b'' \in B_*''(a_i,b_i)$ and $b'' \in B_*''(c_i,d_i)$, and by \cref{prop:sac-region-containment}, we have $d_i \in B'(a_i,b_i)$.
    Thus, by \cref{prop:satisfy-items-r}, the following definitions are valid.
    For every $i \in [m]$, let 
    \[\calR_i = \calR(c_i,z'(c_i,d_i),b'') \ \text{ and } \ \calS_i = \calR(a_i,d_i,b'').\]
    By \cref{prop:smallest-region} and \cref{prop:sac-region-containment} for every $i \in [m]$, we have
        \[\calR(a_i,z'(a_i,b_i),b'') \subset \calS_i \subset \calR_i.\]
    Additionally, for every $i \in [m-1]$, by \ref{itemH':regions},
        \[\calR_i \subset \calR(a_{i+1},z'(a_{i+1},b_{i+1}),b'').\]
    It follows that
        \[\calS_1 \subset \calR_1 \subset \calS_2 \subset \dots \subset \calS_m \subset \calR_m.\]
    For each $b \in B$ and for each $a \in A$ with $a \leq b''$ in $P$, let
    \[\pi(b) = \gcpe(W_R(b),W_R(b'')) \ \text{ and } \ \pi(a) = \lcse(W_R(b''),W_R(a,b'')).\]
    Note that by \cref{prop:sandwich}, for all $d,b \in B$ if $d \prec_R b$, then $\pi(d) \leq \pi(b)$ in $P$.
    By the regions containment and \cref{prop:two-regions}.\ref{prop:two-regions:q}--\ref{prop:two-regions:m}, for every $i \in [m]$, we have 
    \begin{equation}\label{eq:z'-d-z'-and-c-a}
        \pi(z'(c_i,d_i)) \leq \pi(d_i) \leq \pi(z'(a_i,b_i)) \ \text{ and } \ \pi(c_i) \leq \pi(a_i) \text{ in } P
    \end{equation}
    and for every $i \in [m-1]$,
    \begin{equation}\label{eq:z'-z'-and-a-c}
        \pi(z'(a_{i+1},b_{i+1})) \leq \pi(z'(c_i,d_i)) \ \text{ and } \ \pi(a_{i+1}) \leq \pi(c_i) \text{ in } P.
    \end{equation}
    By \cref{prop:force-inB*}.\ref{prop:force-inB*:cleq}--\ref{prop:force-inB*:gcpe}, for every $i \in [m]$, we have
    \begin{equation}\label{eq:d-c-b}
        \pi(d_{i}) < \pi(c_{i}) \leq \pi(b_i) \text{ in } P.
    \end{equation}
    For every $i \in [m-1]$, since $z'(a_{i+1},b_{i+1}) \prec_R b_{i+1} \preccurlyeq_R d_{i}$, we have
    \begin{equation}\label{eq:z'-b-d}
        \pi(z'(a_{i+1},b_{i+1})) \leq \pi(b_{i+1}) \leq \pi(d_i) \text{ in } P.
    \end{equation}
    In particular, for every $i \in [m-1]$, we have
    \begin{equation}\label{eq:d-c-d}
        \pi(d_{i+1}) < \pi(c_{i+1}) \leq \pi(d_{i}) < \pi(c_i) \text{ in } P.
    \end{equation}

    \begin{claim}\label{claim:B-to-B}
        Let $j \in [m-2]$, let $x,y \in B$ with $x \leq y$ in $P$, 
        If $x \in \calR_j$ then $y \in \calR_{j+2}$.
    \end{claim}
    \begin{proofclaim}
        Assume that $x \in \calR_j$.
        Let $z_j' = z'(c_j,d_j)$ and $z_{j+2}' = z'(c_{j+2},d_{j+2})$.
        Observe that
        %\[\pi(z_{j+2}') \overset{\eqref{eq:z'-d-z'-and-c-a}}{\leq} \pi(d_{j+2}) \overset{\eqref{eq:d-c-d}}{<}\pi(c_{j+2}) \overset{\eqref{eq:d-c-d}}{\leq} \pi(d_{j+1}) \overset{\eqref{eq:z'-d-z'-and-c-a}}{\leq} \pi(z'(a_{j+1},b_{j+1})) \overset{\eqref{eq:z'-z'-and-a-c}}{\leq} \pi(z'_j) \text { in } P. \]
        \[\pi(c_{j+2}) \overset{\eqref{eq:d-c-d}}{\leq} \pi(d_{j+1}) \overset{\eqref{eq:z'-d-z'-and-c-a}}{\leq} \pi(z'(a_{j+1},b_{j+1})) \overset{\eqref{eq:z'-z'-and-a-c}}{\leq} \pi(z'_j) \text { in } P. \]
        By \cref{prop:properties-regions}.\ref{prop:properties-regions:b-inside}, $x \in \calR_j$ implies $\pi(z_j') \leq x$ in $P$.
        In particular,
        \[\pi(c_{j+2}) \leq \pi(z_j') \leq x \leq y \text{ in } P.\]
        Therefore, by \cref{prop:properties-regions}.\ref{prop:properties-regions:d-over-m2}, every witnessing path from $x$ to $y$ in $P$ is contained in $\calR_{j+2}$, thus, $y \in \calR_{j+2}$, as desired.
       %
        % Let $W$ be a witnessing path from $x$ to $y$ in $P$.
        % If $W$ is disjoint from $\partial\calR_j$, then $y \in \calR_j \subset \calR_{j+2}$.
        % Thus, we may assume that $W$ intersects $\partial\calR_j$ and  that $x \in \partial \calR_{j}$. 
        % It suffices to prove that $W$ is disjoint from $\partial \calR_{j+2}$.
        % Suppose towards a contradiciton that $u$ is an element in $W$ and $\partial \calR_{j+2}$. 
        % Note that $x \leq u \leq y$ in $P$. 
        % Since all elements of $W$ are in $B$, by \ref{prop:tuple-is-valid:inA}, $u$ lies in $\pi(z_{j+2}')[W_R(z'_{j+2})]z'_{j+2}$ or $\pi(z'_{j+2})[W_R(b'')]\pi(c_{j+2})$.
        %
        % If $u$ lies in $\pi(z'_{j+2})[W_R(b'')]\pi(c_{j+2})$, then 
        % $u \leq \pi(z_j')$ in $P$, and so, $u \leq \pi(c_{j+2}) \leq \pi(z_j') < x$ in $P$, which is a contradiction. 
        %
        % The only remaining case is that $u$ lies in $\pi(z_{j+2}')[W_R(z'_{j+2})]z'_{j+2} \setminus\{\pi(z'_{j+2})\}$.
        % Since $\pi(z_{j+2}')< \pi(z_j')$ in $P$, 
        % we have that $W_R(z_{j+2}')$ is disjoint from $\pi(z'_j)[W_R(z'_j)]z'_j$ and moreover, $x$ is right of $W_R(z_{j+2}')$.
        % However, it implies that the path $x_0[W_R(z_j')]x[W]u[W_R(z_{j+2}')]z_{j+2}'$ is a witnessing path from $x_0$ to $z_{j+2}'$ containing a point right or $W_R(z_{j+2}')$, which contradicts \cref{claim:diagram:extreme_paths_exist}. This completes the proof of the claim.
    \end{proofclaim}

    \begin{claim}\label{claim:hit-gamma3}
        Let $j\in [m-1]$, let $x,y \in A$, let $W$ be a witnessing path from $x$ to $y$ in $P$.
        If $x \in \calR_j$, then $W$ is disjoint from $c_{j+1}[W_R(c_{j+1},b'')]\pi(c_{j+1})$.
    \end{claim}
    \begin{proofclaim}
        Suppose to the contrary that $W$ intersects $c_{j+1}[W_R(c_{j+1},b'')]\pi(c_{j+1})$ in an element $u$.
        Since $u \in \partial \calR_{j+1}$, $u \notin \Int \calR_j$.
        Since $x \in \calR_j$ and $u \notin \Int\calR_{j}$, $W$ intersects $\partial \calR_j$, say in an element $v$.
        Since all the elements of $W$ are in $A$, by \ref{prop:tuple-is-valid:inA}, $W$ intersects $W_R(c_j,b'')$ or $W_R(c_j,z'(c_j,d_j))$.
        In both cases, $c_j \leq v$ in $P$.
        We obtain 
        \[c_j \leq v \leq u \leq \pi(c_{j+1}) \overset{\eqref{eq:d-c-d}}{\leq} \pi(d_j) \leq d_j \text{ in } P,\]
        which is a contradiction that completes the proof of the claim.
    \end{proofclaim}

    \begin{claim}\label{claim:going-out}
        There exists a polynomial-time algorithm that takes $j\in [m-(n+1)]$, $x \in A$ and $y \in B$ with $x \leq y$ in $P$ such that $x \in \calR_j$ and $y \notin \Int \calS_{j+n+1}$, and outputs a Kelly subposet of $P$ of order $n$ based on $I_R$.
        % Let $j\in [m-(n+1)]$, let $x \in A$ and $y \in B$ with $x \leq y$ in $P$.
        % If $x \in \calR_j$, then either $y \in \calR_{j + n+1}$ or the algorithm may return a Kelly subposet of $P$ of order $n$ based on $I_R$.
    \end{claim}
    \begin{proofclaim}
        Let $W$ be a witnessing path from $x$ to $y$ in $P$.
        Since $x \in A$ and $y \in B$, there exist adjacent elements $u$ and $v$ in $W$ such that all elements of $x[W]u$ are in $A$ and all elements of $v[W]y$ are in $B$.
        By \cref{claim:B-to-B}, if $v \in \calR_{j+n-1}$, then $y \in \calR_{j+n+1}$.
        Thus, $v \notin \calR_{j+n-1}$.
        It follows that $x[W]u$ intersects $\partial \calR_k$ for every integer $k$ with $j \leq k\leq j+n-1$.
        If $x[W]u$ is disjoint from $W_R(c_j,z'(c_j,d_j))$, then by \cref{prop:two-regions}.\ref{prop:two-regions:gamma4} applied to $\calR' = \calR_j$ and $\calR = \calR_{j+1}$, $x[W]u$ intersects $c_{j+1}[W_R(c_{j+1},b'')]\pi(c_{j+1})$, which contradicts \cref{claim:hit-gamma3}.
        Thus, we may assume that $x[W]u$ intersects $W_R(c_j,z'(c_j,d_j))$.
        
        For each integer $0 \leq i \leq n-2$, we apply \cref{prop:two-regions}.\ref{prop:two-regions:gamma3} to $\calR' = \calR_j$ and $\calR = \calR_{j+i}$ obtaining that $x[W]u$ intersects $W_R(c_{j+i},z'(c_{j+i},d_{j+i}))$, say in an element $r_{j+i}$.
        For each integer $1 \leq i \leq n-1$, we apply \cref{prop:two-regions}.\ref{prop:two-regions:gamma3} to $\calR' = \calR_j$ and $\calR = \calS_{j+i}$ obtaining that $x[W]u$ intersects $W_R(a_{j+i},d_{j+i})$, say in an element $s_{j+i}$.
        By the regions containment,
        \[
        \calR_j \subseteq \calS_{j+1} \subseteq \calR_{j+1} \subseteq \cdots \subseteq \calS_{j+n-2} \subseteq \calR_{j+n-2} \subseteq \calS_{j+n-1} 
        \]
        we may assume that 
        \[r_j \leq s_{j+1} \leq r_{j+1} \leq \dots \leq s_{j+n-2} \leq r_{j+n-2} \leq s_{j+n-1} \text{ in } P.\]
        By \eqref{eq:d-c-d}, we have
        \[ c_{j+n-1} \leq \pi(c_{j+n-1}) < \pi(c_{j+n-2}) < \dots < \pi(c_{j+1}) \leq \pi(d_j) \leq d_j \text{ in } P.\]
        For every integer $i \in [n-2]$, we have
        \begin{align*}
        c_{j+i} &\leq r_{j+i} \leq s_{j+i+1} \leq d_{j+i+1} \textrm{ in $P$,}\\
        c_{j+i} &\leq \pi(c_{j+1}) \leq \pi(d_{j+i-1}) < d_{j+i-1} \textrm{ in $P$.}
        \end{align*}
        In the second line above, we used \eqref{eq:d-c-d}.
        % For every $i \in [n-2]$, we have
        % \[c_{j+i} \leq v_{i+1} \ \text{ and } \ v_{i}' \leq d_{j+i} \text{ in } P\]
        %For every integer $2 \leq i \leq n-1$, let $u_i = \pi(c_{j+i-1})$.
        It follows that the subposet of $P$ induced by the elements
        \[\{c_{j},\dots,c_{j+n-1},d_j,\dots,d_{j+n-1},\pi(c_{j+1}),\dots,\pi(c_{j+n-2}),r_{j+1},\dots,r_{j+n-2}\}\]
        is a Kelly subposet of $P$ of order $n$ based on $I_R$.
        Note that all the elements in the Kelly subposet above are easy to find by a polynomial-time algorithm.
    \end{proofclaim}

    \begin{claim}\label{claim-a}
        $a \in \calR_1$.
    \end{claim}
    \begin{proofclaim}
        By \eqref{eq:d-c-b}, since $b_1 \preccurlyeq_R b$, and since $a\parallel b$ in $P$ we have
        \[\pi(c_1) \leq \pi(b_1) \leq \pi(b) < \pi(a) \text{ in } P. \]
        Let $(q_1,m_2,q_3,m_4)$ be the extreme points of $\calR_1$.
        Note that $\pi(c_1) = m_2$.
        By \cref{prop:properties-regions}.\ref{prop:properties-regions:b''}, it follows that $\pi(a) \in \Int \calR_1$.
        Suppose to the contrary that $a \notin \calR_1$.
        Then, $a[W_R(a,b'')]\pi(a)$ intersects $\partial \calR_1$.
        By \ref{prop:tuple-is-valid:inA}, they intersect in $q_3[W_R(c_1,b'')]\pi(c_1)$ or $q_3[W_R(c_1,z'(c_1,d_1))]m_4$, say in an element $u$.
        If $u$ lies in $q_3[W_R(c_1,b'')]\pi(c_1)$, then $a \leq u \leq \pi(c_1) \leq b$ in $P$, which is a contradiction.
        Therefore, $u$ lies in $q_3[W_R(c_1,z'(c_1,d_1))]m_4 \setminus\{q_3\}$.
        In this case, $u$ is right of $W_R(c_1,b'')$ by \ref{prop:tuple-is-valid:left}.
        It follows that $c_1[W_R(c_1,z'(c_1,d_1))]u[W_R(a,b'')]b''$ is a witnessing path from $c_1$ to $b''$ in $P$ containing an element right of $W_R(c_1,b'')$, which contradicts \cref{claim:diagram:extreme_paths_exist}.
        Therefore, $a[W_R(a,b'')]\pi(a)$ does not intersect $\partial \calR_1$, and $a \in \calR_1$, as claimed.
    \end{proofclaim}

    Recall that $((a,b),(c,d))$ on the input is an edge in $H$.
    There exists $f \in B$ witnessing this fact.
    If $f \notin \Int \calS_{n+2}$, then using the algorithm from \Cref{claim:going-out}, we output a Kelly subposet of $P$ of order $n$ based on $I_R$.
    Thus, we assume that $f \in \Int \calS_{n+2}$.

    \begin{claim}\label{claim:c}
        If $f \in \Int \calS_{n+2}$, then $W_R(c,b'') \subset \calR_{n+3}$.
    \end{claim}
    \begin{proofclaim}
        By \eqref{eq:d-c-b}, \cref{prop:two-regions}.\ref{prop:two-regions:q}, and since $c \parallel f$ in $P$, we have
        \[\pi(c_{n+3}) < \pi(d_{n+2}) \leq \pi(f) < \pi(c) \text{ in } P.\]
        It follows that $\pi(c)[W_R(c,b'')]b'' \subset \Int\calR_{n+3}$ by \cref{prop:properties-regions}.\ref{prop:properties-regions:b''}.
        %\later{The stuff below is repeated with the previous claim.}
        Suppose to the contrary that $c[W_R(c,b'')]\pi(c) \setminus\{\pi(c)\}$ intersects $\partial\calR_{n+3}$.
        Let $(q_1,m_2,q_3,m_4)$ be the extreme points of $\calR_{n+3}$.
        By \ref{prop:tuple-is-valid:inA}, they intersect in $q_3[W_R(c_{n+3},b'')]\pi(c_{n+3})$ or $q_3[W_R(c_{n+3},z'(c_{n+3},d_{n+3}))]m_4$, say in an element $u$.
        If $u$ lies in $q_3[W_R(c_{n+3},b'')]\pi(c_{n+3})$, then $c \leq u \leq \pi(c_{n+3}) \leq \pi(f) \leq f$ in $P$, which is a contradiction.
        Therefore, $u$ lies in $q_3[W_R(c_{n+3},z'(c_{n+3},d_{n+3}))]m_4 \setminus\{q_3\}$.
        In this case, $u$ is right of $W_R(c_{n+3},b'')$ by \ref{prop:tuple-is-valid:left}.
        It follows that $c_{n+3}[W_R(c_{n+3},z'(c_{n+3},d_{n+3})]u[W_R(c,b'')]b''$ is a witnessing path from $c_{n+3}$ to $b''$ in $P$ containing an element right of $W_R(c_{n+3},b'')$, which contradicts \cref{claim:diagram:extreme_paths_exist}.
        Therefore, $c[W_R(c,b'')]\pi(c)$ does not intersect $\partial \calR_{n+3}$, as claimed.
    \end{proofclaim}

    In particular, \Cref{claim:c} implies that $c \in \calR_{n+3}$.
    If $W_R(c,z'(c,d))$ contains any element outside $\Int \calS_{2n+4}$, then using the algorithm from \Cref{claim:going-out}, we output a Kelly subposet of $P$ of order $n$ based on $I_R$.
    Thus, we assume that $W_R(c,z'(c,d)) \subset \Int \calS_{2n+4}$.

    \begin{claim}\label{claim:edge}
        If $f \in \Int \calS_{n+2}$ and $W_R(c,z'(c,d)) \subset \Int \calS_{2n+4}$, then $((c,d),(a_{2n+6},b_{2n+6}))$ is an edge in $H'$.
    \end{claim}
    \begin{proofclaim}
        Since $z'(c,d) \in \calS_{2n+4}$, $z'(c,d)$ does not lie in $W_R(d_{2n+4})$ and $d_{2n+4}$ does not lie in $W_R(z'(c,d))$ by \cref{prop:properties-regions}.\ref{prop:properties-regions:b-lower-m4}.
        By \cref{prop:properties-regions}.\ref{prop:properties-regions:d-outside}, we also do not have $z'(c,d) \prec_R d_{2n+4}$.
        It follows that $d_{2n+4} \prec_R z'(c,d)$.
        We also have $b_{2n+6} \prec_R d_{2n+4}$.
        Thus, $b_{2n+6} \prec_R d_{2n+4} \prec_R z'(c,d) \prec_R d$.
        This gives \ref{itemH':ordering} for $((c,d),(a_{2n+6},b_{2n+6}))$.

        Let $\calR' = \calR(c,z'(c,d),b'')$ and let $\calR = \calR(a_{2n+6},z'(a_{2n+6},b_{2n+6}),b'')$.
        Condition \ref{itemH':regions} states that $\calR' \subset \calR$.
        Let $(q_1,m_2,q_3,m_4)$ be the extreme points of $\calR$.
        Note that $q_1 = \pi(z'(a_{2n+6},b_{2n+6}))$.
        We prove that $\partial \calR' \subset \calR$ which by \cref{obs:region_containment} suffices.
        We assumed that $W_R(c,z'(c,d)) \subset \Int \calS_{2n+6} \subset \calR$ and by \cref{claim:c}, $W_R(c,b'') \subset \calR_{n+3} \subset \calR$.
        %\later{The argument below has already appeared several times. Maybe we can have a statement saying that for region containment it suffices to check red parts?}
        By \Cref{prop:properties-regions}.\ref{prop:properties-regions:b-inside}, $q_1$ lies in $W_R(z'(c,d))$ and $q_1[W_R(z'(c,d))]z'(c,d)\subset \calR$.
        Since $q_1$ lies in $W_R(z'(c,d))$ and $W_R(b'')$, we have $q_1 \leq \pi(z'(c,d))$ in $P$.
        It follows that $\pi(z'(c,d))[W_R(z'(c,d))]z'(c,d) \subset \calR$ and by \cref{prop:properties-regions}.\ref{prop:properties-regions:b''} $\pi(z'(c,d))[W_R(d)]b'' \subset \calR$.
        We conclude that indeed $\calR' \subset \calR$ and \ref{itemH':regions} holds.

        Condition \ref{itemH':B*} states that $B''_*(c,d) \subset B_*''(a_{2n+6},b_{2n+6})$.
        Since $b_{2n+6} \prec_R d$, we have $B''(c,d) \subset B''(a_{2n+6},b_{2n+6})$.
        Let $d'' \in B''_*(c,d)$.
        We must prove that $a_{2n+6} \leq d''$ in $P$ and $W_R(a_{2n+6},d'')$ is right of $W_R(d'')$.
        % Note that
        % \[\pi(a_{2n+6}) \leq \pi(d_{2n+4}) \leq \pi(z'(c,d)) \leq \pi(d) \text{ in } P.\]
        
        First, we show that it suffices to have $\pi(a_{2n+6})$ in $W_R(d'')$.
        Clearly, this implies $a_{2n+6} \leq d''$ in $P$.
        If $a_{2n+6} \preccurlyeq_\uparrow x_0$, then since $(a_{2n+6},b_{2n+6})$ is a right pair, $W_R(a_{2n+6},d'')$ must be right of $x_0$ and so of $W_R(d'')$ by \cref{prop:diagram:for_consistent_one_distinction_is_enough}.
        If $x_0 \prec_\uparrow a_{2n+6}$, then $a_{2n+6}$ is right of $W_R(b'')$, and so, right of $W_R(\pi(a_{2n+6}))$, and so, right of $W_R(d'')$.

        Therefore, the remainder of the proof is devoted to showing 
        \[\text{$\pi(a_{2n+6})$ in $W_R(d'')$}.\]
        By \cref{prop:properties-regions}.\ref{prop:properties-regions:d-inside} and \cref{prop:sandwich}, we have
        \[\pi(a_{2n+6}) \overset{\eqref{eq:z'-z'-and-a-c}}{\leq} \pi(c_{2n+5}) \overset{\eqref{eq:d-c-d}}{\leq} \pi(d_{2n+4}) \leq \pi(z'(c,d)) \leq \pi(d) \text{ in } P.\]
        If $b''$ lies in $W_R(d'')$, then this assertion is clear.
        Note that $z'(c,d)$ does not lie in $W_R(d'')$.
        If $d''$ lies in $W_R(b'')$, then $\pi(a_{2n+6}) < d''$ as otherwise, $c \leq d'' \leq \pi(a_{2n+6}) \leq \pi(d) \leq d$ in $P$, which is a contradiction.
        If $d'' \prec_R b''$, then as $z'(c,d) \prec_R d''$, by \cref{prop:m4=z'} and \cref{prop:properties-regions}.\ref{prop:properties-regions:b-inside}, $d \in \calR'$ and $\pi(a_{2n+6}) \leq \pi(z'(c,d)) \leq \pi(d'') \leq d''$ in $P$, as desired.
        Therefore, we may assume that $b'' \prec_R d''$.
        Let $q = \gcpe(W_R(b''),W_R(d''))$.
        If $\pi(a_{2n+6}) \leq q$ in $P$, then we again obtain the desired assertion. 
        To complete the proof, we show that the remaining case, i.e.\ $q < \pi(a_{2n+6})$ in $P$, leads to a contradiction.

        Since $b'' \in B''*(c,d)$, either $x_0 \prec_\uparrow c$ or $c$ is right of $W_R(b'')$.
        Since $c \in \calR$, by \ref{prop:tuple-is-valid:left} and \Cref{obs:curves-region}, this implies that $c$ is right of $W_R(a_{2n+6},b'')$ and left of $W_R(a_{2n+6},z'(a_{2n+6},b_{2n+6}))$.
        In particular, $a_{2n+6} \prec_\uparrow c \prec_\uparrow d''$.

        We define a horizontal line $\ell$.
        If $q \prec_\uparrow a_{2n+6}$, let $\ell$ be the horizontal line containing $a_{2n+6}$ and if $a_{2n+6} \prec_\uparrow q$, let $\ell$ be a horizontal line slightly above $q$.
        Let $p$ be the point of $W_R(d'')$ in $\ell$.
        Note that when $a_{2n+6} \prec_\uparrow q$, we may choose $\ell$ so that $p \prec_\uparrow c$, and $p$ is right of $W_R(b'')$ and left of $W_R(a_{2n+6},b'')$.
        We consider two cases.
        
        First, suppose that $p$ is right of $W_R(a_{2n+6})$.
        In this case, $q \prec_\uparrow a_{2n+6}$, thus, $p$ is right of $a_{2n+6}$.
        By \ref{prop:tuple-is-valid:inA}, $a_{2n+6}[W_R(a_{2n+6})]m_4$ is disjoint from $W_R(d'')$, thus by \cref{prop:diagram:for_consistent_one_distinction_is_enough}, $p[W_R(d'')]d''$ is right of $a_{2n+6}[W_R(a_{2n+6})]m_4$.
        However, $c$ is left of $a_{2n+6}[W_R(a_{2n+6})]m_4$, so $c$ is left of $p[W_R(d'')]d''$, which contradicts $d'' \in B''_*(c,d)$.

        Second, suppose that $p$ is left of $W_R(a_{2n+6})$.
        Recall that $p$ is right of $W_R(b'')$.
        By \cref{claim:diagram:extremal_paths_are_consistent} and \ref{prop:tuple-is-valid:inA}, $p[W_R(d'')]d''$ is disjoint from $q[W_R(b'')]b''$ and $a_{2n+6}[W_R(a_{2n+6})]m_2\setminus\{m_2\}$, respectively.
        It follows that $d'' \prec_\uparrow m_2$ as otherwise, the element of $W_R(d'')$ in the horizontal line of $m_2$ has to be both left and right of $m_2$, which is a contradiction.
        It also follows by \cref{obs:where-curves-end}, that $d''$ is left of $a_{2n+6}[W_R(a_{2n+6})]m_2\setminus\{m_2\}$.
        Since $c$ is right of $a_{2n+6}[W_R(a_{2n+6})]m_2\setminus\{m_2\}$, by the Darboux property, a witnessing path from $c$ to $d''$ intersects $a_{2n+6}[W_R(a_{2n+6})]m_2\setminus\{m_2\}$.
        However, this implies $c \leq m_2 = \pi(a_{2n+6}) \leq \pi(d) \leq d$ in $P$, which is a contradiction that completes the proof of the claim.
    \end{proofclaim}
    \cref{claim:edge} implies that under the assumptions that  $f \in \Int \calS_{n+2}$ and $W_R(c,z'(c,d)) \subset \Int \calS_{2n+4}$, 
        \[\ell(a,b) \leq \ell(c,d) + 2n+6. \]
    This contradicts the assumptions on the given input, hence, the algorithm has already terminated, outputting a Kelly subposet of $P$ of order $n$ based on $I_R$.
    This completes the proof of the lemma.
\end{proof}

\begin{proposition}\label{prop:IRalpha-reversible}
    Let $n$ be a positive integer.
    For every $\alpha \in [2n+7]$, let 
    \[I_{R,\alpha} = \{(a,b) \in I_R : \ell(a,b) \equiv \alpha \bmod (2n+7)\}.\]
    If $n > \kelly_P(I_R)$, then $I_{R,\alpha}$ is reversible for every $\alpha \in [2n+7]$.
\end{proposition}
\begin{proof}
    Let $\alpha \in [2n+7]$.
    Suppose to the contrary that there exists a strict alternating cycle in $P$ contained in $I_{R,\alpha}$.
    By \cref{cor:sacs}, this cycle has two pairs connected by an edge in $H$, say $(a,b)$ and $(c,d)$.
    By \cref{prop:ell-plus-1}, $\ell(a,b) > \ell(c,d)$, thus, by the definition of $I_{R,\alpha}$, $\ell(a,b) \geq \ell(c,d) + (2n+7)$.
    However, in this case, \Cref{lemma:shortcuts} would give a Kelly subposet in $P$ of order $n$ based on $I_R$, which contradicts $n > \kelly_P(I_R)$.
\end{proof}

\cref{lemma:shortcuts,prop:IRalpha-reversible} (and their symmetric versions) complete the proof of \cref{lem:singly-constrained-proof} as discussed in \cref{sec:outline}.
Thus, we also obtain \cref{thm:main}, and hence \cref{thm:approx-algo,thm:dim-boundedness,thm:dim-boundedness-Kelly,thm:algo-kelly}.

\bibliographystyle{plain}
\bibliography{bibliography}

\end{document}